\documentclass[12pt]{article}
\usepackage{amsmath,amsthm,amssymb,graphicx,tikz,mathtools,enumitem,tikz-cd,tkz-graph,caption,subcaption,pdflscape}
\usepackage{wrapfig,lscape,rotating,epstopdf}
\usepackage[margin=1in]{geometry}

\newcommand{\bbC}{\mathbb{C}}

\newcommand{\bbE}{\mathbb{E}}

\newcommand{\bbI}{\mathbb{I}}

\newcommand{\bbN}{\mathbb{N}}

\newcommand{\bbQ}{\mathbb{Q}}
\newcommand{\bbR}{\mathbb{R}}

\newcommand{\bbV}{\mathbb{V}}

\newcommand{\bbZ}{\mathbb{Z}}

\newcommand{\calB}{\mathcal{B}}

\newcommand{\calF}{\mathcal{F}}

\newcommand{\calH}{\mathcal{H}}

\newcommand{\calP}{\mathcal{P}}

\newcommand{\calR}{\mathcal{R}}

\newcommand{\frakl}{\mathfrak{l}}

\DeclareMathOperator{\id}{id}

\DeclareMathOperator{\C}{C}
\DeclareMathOperator{\PC}{PC}
\DeclareMathOperator{\WPC}{WPC}

\DeclareMathOperator{\SL}{SL}
\DeclareMathOperator{\GL}{GL}

\DeclarePairedDelimiter\floor{\lfloor}{\rfloor}
\DeclareMathOperator{\PA}{PA}
\DeclareMathOperator{\inte}{int}

\DeclarePairedDelimiterX{\norm}[1]{\lVert}{\rVert}{#1}

\def\fillandplacepagenumber{%
 \par\pagestyle{empty}%
 \vbox to 0pt{\vss}\vfill
 \vbox to 0pt{\baselineskip0pt
   \hbox to\linewidth{\hss}%
   \baselineskip\footskip
   \hbox to\linewidth{%
     \hfil\thepage\hfil}\vss}}

\theoremstyle{plain}
\newtheorem{thm}{Theorem}[section]
\newtheorem{prop}[thm]{Proposition}
\newtheorem{lem}[thm]{Lemma}
\newtheorem{cor}[thm]{Corollary}
\newtheorem{mainthm}{Theorem}

\newtheorem{conj}{Conjecture}

\theoremstyle{definition}
\newtheorem{defn}[thm]{Definition}
\newtheorem{ex}[thm]{Example}
\newtheorem{rmk}[thm]{Remark}

\title{Constructing pseudo-Anosovs from expanding interval maps}
\author{Ethan Farber\thanks{Boston College}}

\begin{document}
\newpage
\maketitle


\begin{abstract}
	We investigate a phenomenon observed by W. Thurston wherein one constructs a pseudo-Anosov homeomorphism on the limit set of a certain lift of a piecewise-linear expanding interval map. We reconcile this construction with a special subclass of generalized pseudo-Anosovs, first defined by de Carvalho. From there we classify the circumstances under which this construction produces a pseudo-Anosov. As an application, we produce for each $g \geq 1$ a pseudo-Anosov $\phi_g$ on the closed surface of genus $g$ that preserves an algebraically primitive translation structure and whose dilatation is a Salem number.
\end{abstract}

\section{Introduction}

In \cite{Th2}, Thurston investigates the properties of \textit{uniform $\lambda$-expanders}, piecewise linear maps $f: I \to I$ of a compact interval with $|f'|=\lambda>1$ wherever defined. Here $\lambda$ is sometimes referred to as the \textit{growth rate} of $f$. These expanders arise naturally as linear models for interval endomorphisms of positive entropy (see, e.g. \cite{MTh}). In many ways, uniform $\lambda$-expanders may be seen as one-dimensional analogues of pseudo-Anosov homeomorphisms of a surface.

Recall that a \textit{pseudo-Anosov} is a homeomorphism $\phi: S \to S$ of surface $S$ such that there exist two singular transverse measured foliations $(\calF^u,\mu^u)$, $(\calF^s, \mu^s)$ and a constant $\lambda>1$ such that $\phi_\ast(\calF^u, \mu^u)=(\calF^u,\lambda \mu^u)$ and $\phi_\ast(\calF^s, \mu^s)=(\calF^s, \lambda^{-1}\mu^s)$. The constant $\lambda$ is called the \textit{dilatation} of $\phi$, and determines the topological entropy of $\phi$: namely, $h(\phi)=\log{\lambda}$. For a uniform $\lambda$-expander $f$ we also have $h(f)=\log{\lambda}$.

In \cite{Th2} Thurston considers another parallel: it is shown by Fried \cite{Fri} that the dilatation of a pseudo-Anosov is a \textit{bi-Perron unit}, i.e. that $\lambda$ is a real, positive algebraic unit whose Galois conjugates lie in the annulus $A_\lambda=\{z : \lambda^{-1}<|z|<\lambda\}$, with the exception of $\lambda$ and possibly $\lambda^{-1}$. It is conjectured that this condition is also sufficient for $\lambda$ to be the dilatation of some pseudo-Anosov on some surface. On the other hand, Thurston argues in \cite{Th2} that $\lambda$ is the growth rate of a postcritically finite (PCF, cf. Definition \ref{defn:Markov}) uniform expander if and only if $\lambda$ is a \textit{weak Perron number}; that is to say, all Galois conjugates are contained in the disc $D_\lambda=\{z:|z| \leq \lambda\}$.

Given these and other dynamical similarities between uniform expanders and pseudo-Anosovs, it is reasonable to ask if one can lift a uniform expander to a pseudo-Anosov on some surface, or conversely if one can find a projection from a pseudo-Anosov to some uniform expander.  Thurston gives a suggestive example near the end of \cite{Th2}. Thurston's example concerns the tent map for $\lambda=(1+\sqrt{5})/2$ defined by 

\begin{equation}\label{eqn:tent}
f(x)=\begin{cases}
\lambda x & 0 \leq x < \lambda^{-1}\\
2-\lambda x & \lambda^{-1} \leq x \leq 1
\end{cases}
\end{equation}

The sole Galois conjugate of $\lambda$ is $\psi=(1-\sqrt{5})/2=-\lambda^{-1}$, and Thurston defines the map $f_G: I \times \bbR \to I \times \bbR$ by 

\[
f_G(x,y)=\begin{cases}
(\lambda x, \psi y) & 0 \leq x < \lambda^{-1} \\
(2-\lambda x, 2-\psi y) & \lambda^{-1} \leq x \leq 1
\end{cases}
\]

\begin{figure}
\begin{subfigure}{.25\paperwidth}
\centering
\includegraphics[scale=.3]{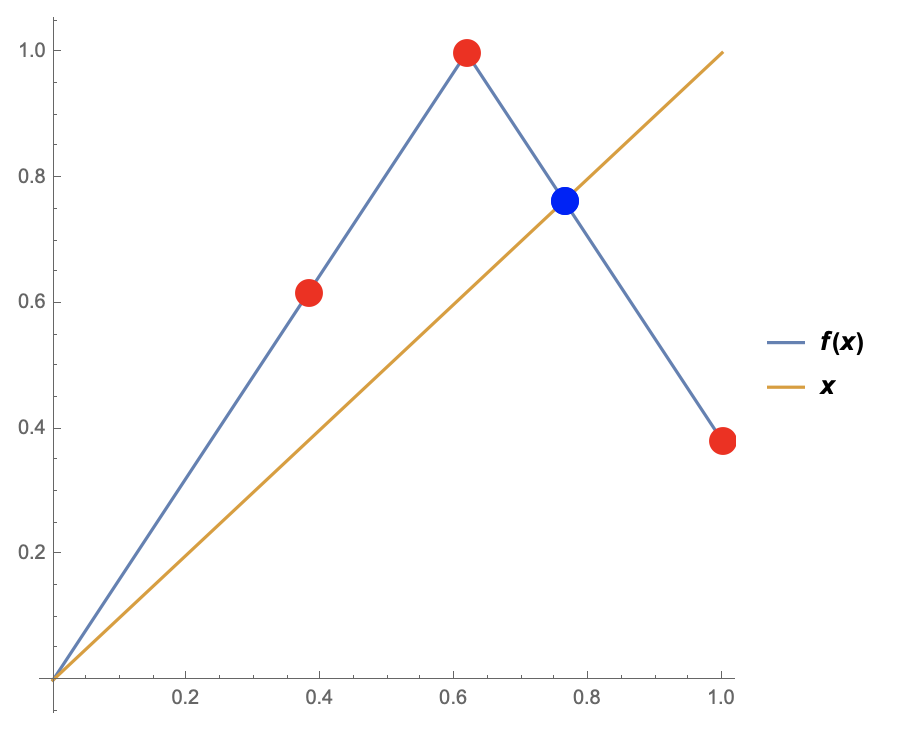}
\caption{A uniform $\lambda$-expander $f$ for $\lambda=(1+\sqrt{5})/2$.}
\end{subfigure}
\begin{subfigure}{.5\paperwidth}
\centering
\includegraphics[scale=.3]{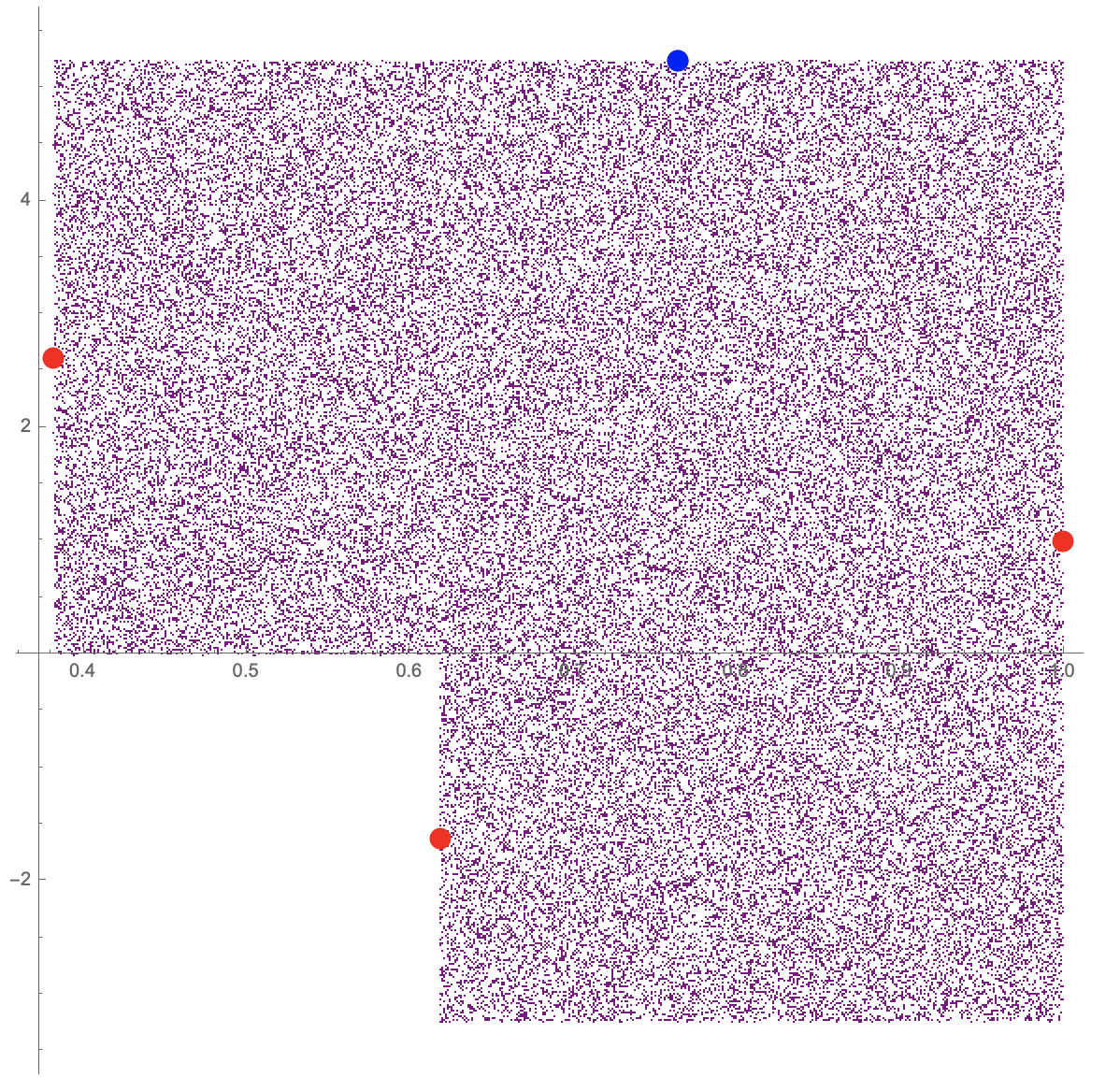}
\caption{The limit set of the Galois lift $f_G$.}
\end{subfigure}
		
\vspace{1cm}
		
\caption{Thurston's example.  He gives gluings under which the limit set is a sphere with four cone points, corresponding to the postcritical orbit (red) and a fixed point (blue) of $f$.}\label{fig:Thurston}
\end{figure}

Note that if $\pi: I \times \bbR \to I$ is the projection onto the first coordinate, then $\pi \circ f_G = f \circ \pi$. The \textit{limit set} $\Lambda_f$ of $f_G$ is defined to be the smallest closed set containing all accumulation points of orbits under $f_G$. In this example, $\Lambda_f$ is a finite union of rectangles in $I \times \bbR$, and while $f_G$ is discontinuous across the line $x=\lambda^{-1}$, one may find gluings of the edges of $\Lambda_f$ such that the induced map $\phi$ is a homeomorphism of a closed surface $\tilde{\Lambda}$. The definition of $f_G$ implies that the foliations of $\Lambda_f$ by horizontal and vertical lines are each preserved under the transformation, with the one-dimensional measure inherited from Lebesgue scaling in the horizontal direction by $\lambda$ and in the vertical direction by $|\psi|=\lambda^{-1}$. This persists in the quotient $\tilde{\Lambda}$, which is homeomorphic to $S^2$. In $\tilde{\Lambda}$ there are four cone points of angle $\pi$, each having a single preimage in $\Lambda$. These preimages have distinct $x$-coordinates and so project to four points of dynamical relevance for $f$: in particular, there is a correspondence between these points of $I$ and the cone points of $\tilde{\Lambda}$, namely

\begin{itemize}
\item Three points lying over the unique postcritical orbit of $f$ (shown in red in Figure \ref{fig:Thurston}), and 
\item One point lying over the non-zero fixed point of $f$ (shown in blue in the same figure)
\end{itemize}

Taken together, we see that the Galois lift $f_G$ admits a quotient map $\phi$ which is an orientation-reversing pseudo-Anosov homeomorphism of the sphere with four marked points. The dilatation of $\phi$ is $\lambda$, and $\phi$ projects to the $\lambda$-expander $f$. It is not clear from the discussion in \cite{Th2} whether this process is canonical, or whether it produces pseudo-Anosovs in greater generality. In this paper we classify when this construction works for interval maps of a particular form, and use it to quickly generate pseudo-Anosovs on surfaces of any genus.

\begin{rmk}
The fact that $\phi$ reverses orientation is because $\lambda$ and $\psi$ have opposite signs. Throughout this paper we will restrict our attention to orientation-preserving pseudo-Anosovs, but the methods for understanding the orientation-reversing case are similar.
\end{rmk}

\begin{rmk}
In general, we will see that for a pseudo-Anosov $\phi$ generated by thickening an interval map $f$, there is a one-pronged singularity of $\phi$ for each point in the forward orbit of a critical point for $f$. See Theorem \ref{PCP} and the discussion in Section 2. These singularities lie on the vertical boundary of $\Lambda$, whereas on the horizontal boundary there exists a unique singularity corresponding to some periodic cycle of $f$. In Thurston's example this periodic cycle is a fixed point, marked in blue, but in general the cycle is longer, generating a singularity of higher rank.
\end{rmk}

The question of how to formalize this construction was also addressed by Baik, Rafiqi, and Wu in \cite{BRW}. Their construction concerns a different class of uniform expanders, and together our work and theirs constitute special cases of a larger phenomenon wherein interval maps with certain combinatorics can be realized as train track maps.\\

\begin{defn}\label{defn:PCP}
We say that a PCF map $f: I \to I$ is \textit{postcritically periodic} (or \textit{PCP}) if each critical value $f(c)$ is periodic. In other words, $f$ acts on the finite set

\[
\PC(f)=\bigcup_{c \ \text{critical}} \{f^n(c): n \geq 1\}
\]

\noindent by a permutation. It is possible for multiple critical points $c_1, c_2$ to have the same critical value.
\end{defn}

As we will see (cf. Theorem \ref{PCP}), an interval map $f: I \to I$ being PCP is necessary for producing a pseudo-Anosov by thickening $f$.

\begin{defn}\label{defn:zig-zag}
We call a uniform $\lambda$-expander $f: I \to I$ a \textit{zig-zag map} (or \textit{zig-zag} or \textit{$\lambda$-zig-zag}) if the only critical points of $f$ are $c_i=i \cdot \lambda^{-1}$ for $i=1, \ldots, \floor{\lambda}$. In other words, $f(0) \in \partial I$, and $c \in \inte(I)$ is a critical point of $f$ if and only if $f(c) \in \partial I$. Note that for a fixed $\lambda$ there are two distinct $\lambda$-zig-zags: the \textit{positive} zig-zag satisfies $f(0)=0$, while the \textit{negative} zig-zag satisfies $f(0)=1$.
\end{defn}

\begin{figure}
\begin{subfigure}{.35\paperwidth}
\centering
\includegraphics[scale=.5]{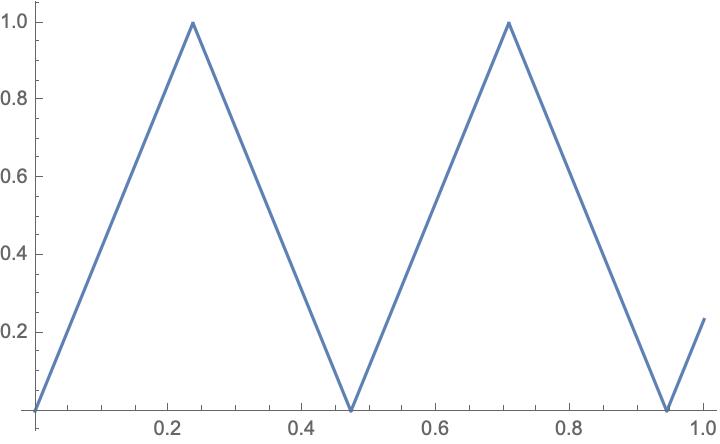}
\caption{A positive zig-zag map.}
\end{subfigure}
\begin{subfigure}{.4\paperwidth}
\centering
\includegraphics[scale=.5]{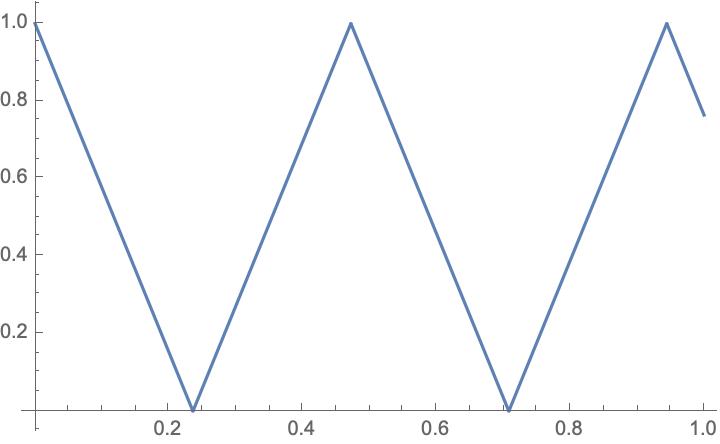}
\caption{A negative zig-zag map.}
\end{subfigure}
		
\vspace{1cm}
		
\caption{The positive and negative zig-zag maps of slope $\lambda$ equal to the Perron root of $x^2-4x-1$.}\label{fig:zigzagex}
\end{figure}

Observe that a zig-zag map $f$ is PCP if and only if $x=1$ is periodic. Throughout this paper we consider primarily the case of a PCP zig-zag map. For the sake of generality, however, we give the following two definitions for \textit{any} PCF uniform $\lambda$-expander with coefficients in $\bbQ(\lambda)$: that is, after forming the Markov partition $\calP_f=\{P_j\}$ for $f$ by cutting $I$ at the union of the critical points and their forward images (cf. Definition \ref{defn:Markov} and Remark \ref{Markov}), there exist polynomials $a_j(\lambda) \in \bbQ(\lambda)$ such that the linear map defining $f$ on the subinterval $P_j$ is given by $f_j(x)=a_j(\lambda) \pm \lambda x$. This assumption is mild, and allows us to make the following two definitions.

\begin{defn}\label{defn:Galois lift}
Let $f: I \to I$ be a uniform $\lambda$-expander with linear branches $f_j(x)=a_j(\lambda) \pm \lambda x$. Set $\tilde{f}_j(x)=a_j(\lambda^{-1}) \pm \lambda^{-1} x$. The \textit{Galois lift} of $f$ is defined to be the function $f_G: I \times \bbR \to I \times \bbR$ such that

\[
f_G(x,y)= \left ( f_j(x), \tilde{f}_j(y) \right ) \hspace{2mm} \text{if $x \in P_j$}
\]

\noindent For concreteness, we define each $P_j$ to be closed on the left, so that $f_G$ is continuous from the left. The limit set of $f_G$ will be denoted $\Lambda_f$.
\end{defn}

In this paper, we investigate when $\Lambda_f$ satisfies the following property.

\begin{defn}\label{defn:rectangular}
Let $f: I \to I$ be a PCF uniform $\lambda$-expander with canonical Markov partition $\calP_f=\{P_j\}$ and let $\pi: \Lambda_f \to I$ be projection onto the first coordinate. We say that $\Lambda_f$ is \textit{rectangular} if it has connected interior and each $R_j=\pi^{-1}(P_j)$ is a Euclidean rectangle. 
\end{defn}

\begin{rmk}
Note here that, despite the terminology ``Galois lift," we do not explicitly require $\lambda$ and $\lambda^{-1}$ to be Galois conjugate. There is, however, a natural integral polynomial $D_f(t)$ associated to a PCP $\lambda$-zig-zag $f$, called the \textit{digit polynomial} (cf. Definition \ref{defn:digit} below), and under the condition that $f_G: \Lambda_f \to \Lambda_f$ defines a pseudo-Anosov this polynomial has both $\lambda$ and $\lambda^{-1}$ as roots (cf. Theorem \ref{Galois}). Indeed, in all observed examples $\lambda$ and $\lambda^{-1}$ are Galois conjugate.
\end{rmk}

\begin{defn}\label{defn:digit}
Let $f$ be a PCP $\lambda$-zig-zag, and let $n \in \bbN$ be minimal such that $f^n(1) \in \partial I$. For each $0 \leq k < n$ let $f_k: \bbC^2 \to \bbC$ be the map of the form $f_k(x, z)=c_k \pm zx$ such that the restriction $f_k |_{\lambda}: x \mapsto f_k(x, \lambda)$ coincides with $f$ on a neighborhood of $f^k(1)$ in $I$. Then the \textit{digit polynomial} of $f$ is the degree $n$ polynomial $D_f: \bbC \to \bbC$ defined by

\[
D_f(t)= \epsilon \big [ f_{n-1} |_{z=t} \circ \cdots \circ f_0 |_{z=t}(1)-f^n(1) \big ]
\]

\noindent where $\epsilon=\pm 1$ is a normalization factor to make $D_f$ monic. By definition, $D_f(\lambda)=0$. 
\end{defn}

\begin{ex}\label{ex:tent}
Let $f: I \to I$ be the tent map from Thurston's example defined by (\ref{eqn:tent}). Then $f^3(1)=1$, and we have $f(1)=\lambda^{-2}$ and $f^2(1)=\lambda^{-1}$. Explicitly, $f_0(x, t)=2-tx$, $f_1(x,t)=tx$, and $f_2(x,t)=2-tx$. We compute

\begin{align*}
\epsilon D_f(t) & = f_2 \left ( f_1 \left ( f_0(1,t), t \right ), t \right)-f^3(1) \\
& = f_2(f_1(2-t,t), t)-1\\
& = f_2(2t-t^2,t)-1\\
& = 2-2t^2+t^3-1\\
& = t^3-2t^2+1
\end{align*}

\noindent Thus in this case $\epsilon=1$ and we find that

\[
D_f(t)=t^3-2t^2+1=(t-1)(t^2-t-1)
\]

\noindent is the digit polynomial of $f$. In particular, we verify that $\lambda=(1+\sqrt{5})/2$ is a root of $D_f$. We remark that even though $D_f(t) \neq t^3D_f(t^{-1})$, this does not provide a counterexample to Theorem \ref{Perm}, since $f$ is not technically of pseudo-Anosov type: it defines an orientation-reversing sphere homeomorphism, rather than an orientation-preserving one.
\end{ex}

\begin{ex}
Let $f: I \to I$ be the (unrestricted negative PCP) zig-zag map defined by 

\begin{equation}
f(x)=\begin{cases}
1-\lambda x & 0 \leq x < \lambda^{-1} \\
\lambda x - 1 & \lambda^{-1} \leq x < 2 \lambda^{-1} \\
3-\lambda x & 2\lambda^{-1} \leq x < 3 \lambda^{-1} \\
\lambda x - 3 & 3\lambda^{-1} \leq x \leq 1
\end{cases}
\end{equation}

\noindent where $\lambda>1$ is the largest real root of $t^4-3t^3-3t^2-3t+1$. In this case $f^4(1)=0$, $f_0(x,t)=f_1(x,t)=tx-3$, $f_2(x,t)=3-tx$, and $f_3(x,t)=1-tx$. We compute as before:

\begin{align*}
\epsilon D_f(t) & = f_3(f_2(f_1(f_0(1,t), t), t), t)-f^4(1)\\
& = f_3(f_2(f_1(t-3), t), t)\\
& = f_3(f_2(t^2-3t-3, t), t)\\
& = f_3(3+3t+3t^2-t^3,1)\\
& = 1-3t-3t^2-3t^3+t^4
\end{align*}

\noindent Here again $\epsilon=1$ and we obtain the polynomial relation $D_f(\lambda)=0$, where $D_f(t)=t^4-3t^3-3t^2-3t+1$ coincides with the minimial polynomial of $\lambda$.
\end{ex}

\begin{rmk}
The term ``digit" polynomial is chosen to reflect the relation between the coefficients of $D_f(t)$ and the digits of the $f$-expansion of $x=1$ (cf. Section 7 and Theorem \ref{Perm}). Since $D_f(t) \in \bbZ[t]$, the minimal polynomial of $\lambda$ divides the digit polynomial. The definition of $D_f(t)$ resembles that of the \textit{Parry polynomial} of the $\lambda$-expander $f$ (cf. \cite{Thom}).
\end{rmk}

Our approach appeals to the machinery of so called \textit{generalized pseudo-Anosovs}, defined in \cite{dC} by de Carvalho and further investigated in \cite{dCH} by de Carvalho and Hall. In brief terms, a \textit{generalized pseudo-Anosov} is a pseudo-Anosov except that we allow for infinitely many singularities of the measured foliations, as long as these only accumulate on finitely many points (cf. Definition \ref{defn:gpA}).\\

The first goal of this paper is Theorem \ref{Galois}, which shows that Thurston's construction for PCP zig-zags recovers the generalized pseudo-Anosov of de Carvalho and Hall in the case when the latter is a pseudo-Anosov.

For a given zig-zag $f$ there is only one thickening $F_L$ of $f$ that can produce a pseudo-Anosov, which we call the \textit{exterior left-veering} thick interval map (cf. Definition \ref{defn:ext} and Proposition \ref{veer}). We say that $f$ is \textit{of pseudo-Anosov type} if $F_L$ generates a pseudo-Anosov according to the construction of de Carvalho and Hall (cf. Definition \ref{defn:pAtype}). In this case there is associated to $F_L$ an \textit{invariant train track} $\tau_L$ of a very explicit form (cf. Theorem \ref{PCP}).

\begin{mainthm}\label{Galois}
Let $f: I \to I$ be a PCP $\lambda$-zig-zag map with $\lambda>2$. Then $f$ is of pseudo-Anosov type if and only if the following conditions are satisfied:

\begin{enumerate}
\item The digit polynomial $D_f$ of $f$ has $\lambda^{-1}$ as a root, and
\item the limit set $\Lambda_f$ of $f_G$ is rectangular (cf. Definition \ref{defn:rectangular}).
\end{enumerate}

In this case, the invariant generalized train track $\tau_L$ of $F_L$ is finite, and recovers the action of $f_G$ on $\Lambda_f$ in the following way: Let $S'$ be the closed topological disc obtained by performing the gluings indicated by the non-loop infinitesimal edges of $\tau_L$. Let $\tilde{f}: S' \to S'$ be the map induced by $F_L$. Then there is an isometry $i: S' \to \Lambda_f$ such that the following diagram commutes:

\[
\begin{tikzcd}
S' \arrow{r}{\tilde{f}} \arrow{d}{i} & S' \arrow{d}{i}\\
\Lambda_f \arrow{r}{f_G} & \Lambda_f
\end{tikzcd}
\]

Moreover, $i$ sends the horizontal and vertical foliations of $S'$ to those of $\Lambda_f$. Therefore, after identifying segments of boundary in each set so as to obtain pseudo-Anosovs $\phi_1: S \to S$ and $\phi_2: \tilde{\Lambda}_f \to \tilde{\Lambda}_f$, these systems are conjugate via an isometry that sends the (un)stable foliation of $\phi_1$ to the (un)stable foliation of $\phi_2$.
\end{mainthm}

\begin{rmk}
The content of Theorem \ref{Galois} is to reveal that, for PCP zig-zags, Thurston's construction of a pseudo-Anosov from a zig-zag is essentially the construction of a generalized pseudo-Anosov following de Carvalho and Hall. Indeed, Thurston's construction is more direct: one need only examine $\Lambda_f$ to determine if $f$ is of pseudo-Anosov type, and if so then the action of $f_G$ on $\Lambda_f$ recovers the pseudo-Anosov action given by applying the methods of de Carvalho and Hall. The proof of Theorem \ref{Galois} is given in Section 4.
\end{rmk}

\begin{rmk}
Theorem \ref{Perm} below strengthens condition (1) of Theorem 1: rather than merely vanishing at $\lambda^{-1}$, $D_f$ is in fact reciprocal, meaning that for any $\alpha \in \bbC$, $D_f(\alpha)=0$ if and only if $D_f(\alpha^{-1})=0$.
\end{rmk}

The remainder of the paper classifies for each modality $m$ the PCP zig-zags of pseudo-Anosov type, i.e. the maps whose Galois lift produces a pseudo-Anosov. These results, in particular Theorems \ref{bijection} and \ref{Perm}, are extensions of the work of Hall, who in \cite{H} classifies unimodal maps (i.e. tent maps) of pseudo-Anosov type, although not in these terms. In particular, Hall shows that there is an explicit bijection between the tent maps of pseudo-Anosov type and $\bbQ \cap (0,1/2)$, defined dynamically by the action of the tent map on its postcritical set. Theorem \ref{bijection} generalizes this to multimodal PCP zig-zags for each modality $m \geq 2$.

The key observation is that the outward winding of the exterior left-veering thickening $F_L$ places strong restrictions on how $f$ permutes the periodic orbit of the point $x=1$ (cf. Section 5). We prove that these restrictions are also sufficient to determine a PCP zig-zag of pseudo-Anosov type. The analysis is carried out in three separate cases, according to whether the modality $m$ of $f$ satisfies

\begin{itemize}
\item[(a)] $m \geq 4$ \text{even},
\item[(b)] $m \geq 3$ \text{odd}, or
\item[(c)] $m=2$
\end{itemize}


\begin{defn}
For $m \geq 2$ and $p \geq 3$, define $\PA(m,p)$ to be the set of zig-zags $f$ of pseudo-Anosov type such that

\begin{enumerate}
\item $f$ has $m$ critical points, and
\item $\#\PC(f) = p$.
\end{enumerate}

\noindent We also define the set

\[
\PA(m) = \bigcup_{p \geq 4} \PA(m,p).
\]
\end{defn}

As it turns out, the elements of $\PA(m,3)$ are relatively uninteresting: the growth rates of these maps are precisely the positive quadratic units, and all of the resulting pseudo-Anosovs are defined on the four-punctured sphere (cf. Corollary \ref{quad}). In particular, we recover hyperbolic automorphisms of the torus for each trace $t \geq 3$. The case when $p \geq 4$ is more interesting.

\begin{defn}\label{defn:ptype}
Let $f$ be a PCP interval map, and let $x_1<x_2<\ldots<x_n=1$ denote the elements of the forward orbit of $x=1$. The \textit{permutation type} of $f$ is the permutation $\rho(f) \in S_n$ such that $f(x_i)=x_{\rho(f)(i)}$.
\end{defn}

In Section 6 we show that for $m \geq 2$, if $f \in \PA(m)$ then $\rho(f)$ has a particular form. Indeed, there exist integers $k, n$ such that $2 \leq k \leq n-1$ and $\gcd(n-1, n-k)=1$ so that we have

\[
\rho(f)=\begin{cases}
\rho_e(n,k) & \text{if $m \geq 4$ even,}\\
\rho_o(n,k) & \text{if $m \geq 3$ odd, or}\\
\rho_2(n,k) & \text{if $m=2$}.
\end{cases}
\]

\noindent The permutations $\rho_e$, $\rho_o$, and $\rho_2$ are defined in Section 6. We will write $\rho(f)=\rho_m(n,k)$ when $m$ is not specified. A consequence of the following theorem is that for each $m \geq 2$, there is a bijection between the permutations $\rho_m(n,k)$ and the elements of $\PA(m)$.

\begin{mainthm}\label{bijection}
Fix $m \geq 2$ and let $\Phi: \PA(m) \to \bbQ \cap (0,1)$ be the map defined by

\[
\Phi(f)=\frac{n-k}{n-1} \hspace{5mm} \text{if $\rho(f)=\rho_m(n,k)$}
\]

\noindent Then $\Phi$ is a bijection. Moreover, for each $p \geq 4$ the image $\Phi(\PA(m,p))$ consists of the set of reduced rationals in $(0,1)$ of denominator $p-2$.
\end{mainthm}

Theorem \ref{bijection} classifies PCP zig-zags of pseudo-Anosov type of modality $m \geq 2$. As we mentioned, the proof in the case of unimodal maps of pseudo-Anosov type is due to Hall \cite{H}; our results should be interpreted as an extension of his to higher modalities. Sections 5 and 6 are dedicated to proving Theorem \ref{bijection}, after which we pass to a discussion of their kneading theoretic implications in Section 7.

We also have the following conjecture.

\begin{conj}
Let $f_1, f_2 \in \PA(m)$ for $m \geq 2$. If $\lambda_i$ is the growth rate of $f_i$ and $\Phi(f_i)=q_i$ as in Theorem \ref{bijection}, then

\[
\lambda_1<\lambda_2 \iff q_1 < q_2
\]

\end{conj}

\noindent We treat this conjecture in a forthcoming paper.

\begin{rmk}
There is a certain symmetry in the choice of rational number representing a zig-zag of pseudo-Anosov type. Namely, one may choose to assign either $q=\frac{n-k}{n-1}$ to $f$, or else $1-q=\frac{k-1}{n-1}$. One may view this symmetry as a manifestation of the fact that $\rho(f)=\rho_m(n,k)$ is essentially a rotation by $n-k$ modulo $n-1$, which may equivalently be seen as a rotation by $-(k-1)$ modulo $n-1$. Making a choice determines whether Conjecture 1 posits that the association $\lambda_i \mapsto q_i$ preserves or reverses linear order. Here we have defined $q=\frac{n-k}{n-1}$ so that this map is order-preserving. This choice also makes the statement of Theorem \ref{Perm} more natural.
\end{rmk}
 
\begin{mainthm}\label{Perm}
Suppose $f \in \PA(m)$ for $m \geq 2$ with $\Phi(f)=\tfrac{a}{b} \in \bbQ \cap (0,1)$ in lowest terms. Define $L: [0, b] \to \bbR$ by $L(t)=\frac{a}{b} \cdot t$. Then

\[
D_f(t)=t^{b+1}+1-\sum_{i=1}^b c_it^{b+1-i},
\]

\noindent where the $c_i$ satisfy

\begin{equation}
c_i=\begin{cases}
m & \text{if $L(t) \in \bbN$ some $t \in [i-1,i]$}\\
m-2 & \text{otherwise}
\end{cases}
\end{equation} 

\noindent In particular, $c_i=c_{b-i}$, so $D_f$ is reciprocal: that is,

\[
D_f(t)=t^{b+1} D_f(t^{-1}).
\]
\end{mainthm}

\noindent Theorem \ref{Perm} generalizes Lemmas 2.5 and 2.6 of \cite{H}.\\

In Section 8 we consider an infinite family of zig-zags producing an algebraically primitive translation surface for each positive genus.

\begin{mainthm}\label{Salem}
For each $g \geq 1$ define $f_g: I \to I$ to be the bimodal PCP zig-zag map of pseudo-Anosov type corresponding to $q_g=\tfrac{1}{2g} \in \bbQ \cap (0,1)$. Let $\lambda_g$ be the growth rate of $f_g$. Then the following are true for each $g \geq 1$.

\begin{enumerate}
\item $\lambda_g$ is a Salem number of degree $2g$.
\item The pseudo-Anosov $\phi_g$ obtained from $f_g$ is defined on a $(2g+2)$-punctured sphere $\Sigma_{0, 2g+2}$.
\item The translation surface $(X_g, \omega_g)$ obtained as the hyperelliptic double cover of $\Sigma_{0, 2g+2}$ is of genus $g$, and hence algebraically primitive.
\end{enumerate}
\end{mainthm}

To summarize: in Section 2 we review the construction of generalized pseudo-Anosovs, focusing particularly on the concept of a thick interval map. This is followed in Section 3 by a discussion on the different ways of thickening an interval map to a thick interval map. Here we prove that, when attempting to construct pseudo-Anosovs, the only thickening of a zig-zag we need consider is the exterior left-veering map $F_L$ (cf. Proposition \ref{veer}).

In Section 4 we prove Theorem \ref{Galois}, reconciling Thurston's construction with that of de Carvalho. Beginning in Section 5, we turn our attention to classifying zig-zag maps of pseudo-Anosov type, observing several necessary conditions on their postcritical orbit structure. Section 6 in effect demonstrates that these conditions are also sufficient, proving Theorem \ref{bijection}.

Section 7 investigates the digit polynomial $D_f(t)$, establishing Theorem \ref{Perm}.

Section 8 turns briefly to considerations of flat geometry, providing a terse introduction to part of the theory. As an application, we prove Theorem \ref{Salem}, demonstrating the existence of an infinite family of algebraically primitive translation surfaces arising from our construction, one for each genus $g \geq 1$. 

\tableofcontents


\subsection{Acknowledgements}


I gratefully acknowledge Andr\'e de Carvalho and Hyungryul Baik for their helpful and insightful comments. I am also particularly grateful to my advisor Kathryn Lindsey for introducing me to this subject and participating in many extremely helpful conversations.


\section{Review of generalized pseudo-Anosovs}


The purpose of this section is to review the theory of generalized pseudo-Anosovs, following the work of de Carvalho in \cite{dC} and de Carvalho and Hall in \cite{dCH}. A reader who is already familiar with the theory may skip this section.


\subsection{Thick intervals}


In this subsection we introduce thick interval maps and the procedure of \textit{thickening} an interval map $f: I \to I$ to a thick interval map $F$.\\

A \textit{thick interval} is a closed topological 2-disc $\bbI \subseteq S^2$ consisting of \textit{decomposition elements} which come in two types: a \textit{leaf}, homeomorphic to the interval $I=[0,1]$, and a \textit{junction}, homeomorphic to the closed $2$-disc. The intersection of the boundary of a junction with $\bbI$ may consist of one or two connected components. We allow only finitely many junctions in a thick interval.

We denote by $\bbV$ the union of the junctions of $\bbI$, and we refer to the connected components of $\bbI \setminus \bbV$ as \textit{strips}. Each strip is homeomorphic to $(0,1) \times [0,1]$, and is a union of leaves. We put coordinates $h_s: \overline{s} \to [0,1] \times [0,1]$ on the closure of each strip $s$ such that the leaves of $s$ are precisely the sets

\[
h_s^{-1}(\{x\} \times [0,1]), \hspace{5mm} x \in (0,1)
\]

\noindent Following \cite{dCH} we denote by $\bbE$ the union of the closures of the strips. See Figure \ref{fig:ti}.

The notation $\bbV$, $\bbE$ is purposefully suggestive. As we shall see, a thick interval map is meant as a dynamical blow-up of an unrestricted interval map $f: I \to I$. The junctions correspond to elements of the weak postcritical set $\WPC(f)$ (cf. Section 3), whereas the strips correspond to the subintervals between these points. It is for this reason that we sometimes refer to junctions as \textit{fat vertices} and the strips as \textit{thick edges}.

\begin{figure}
\centering
\includegraphics[scale=.2]{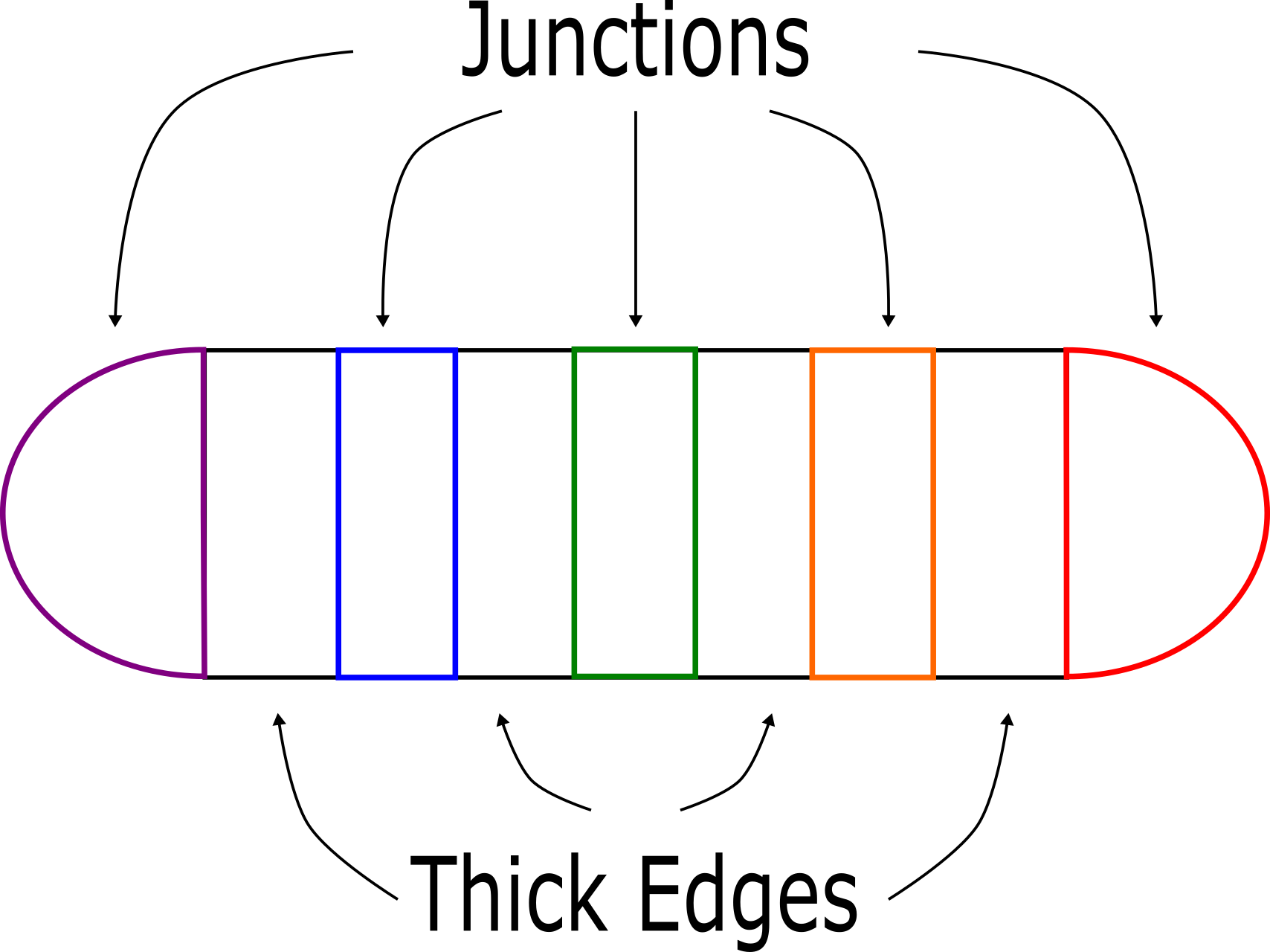}
\caption{A typical thick interval consists of alternating thick edges and two-sided junctions, bookended by a pair of one-sided junctions.}\label{fig:ti}
\end{figure}

We refer to \cite{dCH} for the following definition. We write $F: (X,A) \to (X,A)$ to represent a map $F: X \to X$ of topological spaces such that $F(A) \subseteq A$ for some $A \subseteq X$.

\begin{defn}\label{defn:thickinterval}
A \textit{thick interval map} is an orientation-preserving homeomorphism $F: (S^2, \bbI) \to (S^2, \bbI)$ such that

\begin{enumerate}
\item $F(\bbI)$ is contained in the interior of $\bbI$,
\item if $\gamma$ is a leaf of $\bbI$ then $F(\gamma)$ is contained in a decomposition element, and the diameter of $F^n(\gamma)$ with respect to the coordinates $h_s$ tends to $0$ as $n \to \infty$,
\item if $J$ is a junction of $\bbI$ then $F(J)$ is contained in a junction,
\item $F$ is linear with respect to the coordinates $h_s$: in each connected component of $s_i \cap F^{-1}(s_j)$, where $s_i$ and $s_j$ are strips, $F$ contracts vertical coordinates uniformly by a factor $\mu_j<1$ and expands horizontal coordinates uniformly by a factor $\lambda_j>1$,
\item if $J$ and $J'$ are junctions such that $F(J) \subseteq J'$ then $F(\partial J \setminus \partial \bbI) \subseteq (\partial J' \setminus \partial \bbI)$,
\item if $J$ is a junction with $F^n(J) \subseteq J$ for some $n \geq 1$ then $J$ has an attracting periodic point of period $n$ in its interior whose basin contains the interior of $J$.
\end{enumerate}
\end{defn}

\begin{figure}[h!]
\centering
\includegraphics[scale=.2]{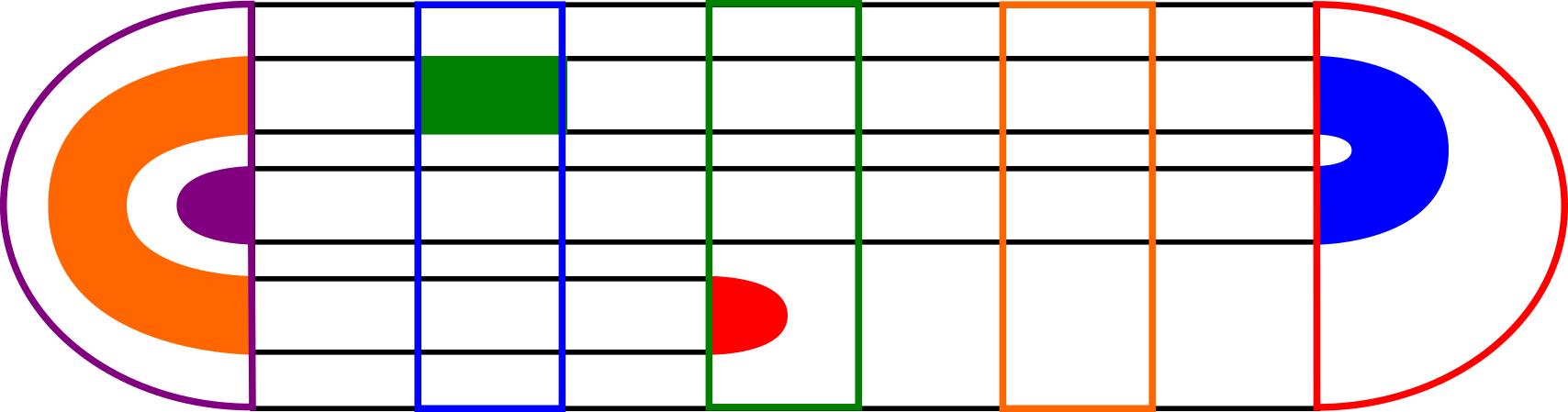}
\caption{A thick interval map. The images of the junctions have been shaded darker for clarity. The first junction is mapped into itself, the second into the fifth, the third into the second, the fourth into the first, and the fifth into the third.}\label{fig:timap}
\end{figure}

\begin{defn}\label{defn:transition1}
We associate to a thick interval map $F$ a \textit{transition matrix} $M=(m_{i,j})$ such that if $s_1, \ldots, s_n$ are the strips of $\bbI$ then $m_{i,j}$ is the number of times $F(s_j)$ crosses $s_i$. Note that since strips are separated by junctions and since junctions are mapped by $F$ to other junctions, the $m_{i,j}$ are integers; there are no partial crossings.
\end{defn}

\begin{defn}\label{defn:primitive}
A non-negative matrix $M$ is said to be \textit{primitive} if there exists some $m \in \bbN$ such that $M^m$ is positive, i.e. $(M^m)_{i,j}$ is positive for each $i, j$. 
\end{defn}

\begin{defn}\label{defn:thickthin}
Let $F:(S^2, \bbI) \to (S^2, \bbI)$ be a thick interval map. Collapse each decomposition element of $\bbI$ to a point, obtaining an interval $\tilde{I}$ and an induced map $\tilde{f}: \tilde{I} \to \tilde{I}$. If $x \in \tilde{I}$ corresponds to a leaf that $F$ maps into a junction, then $\tilde{f}$ will be constant in a neighborhood of $x$. Further collapsing these intervals of constancy produces either a single point or a new interval $I$ with an induced map $f: I \to I$. We say that $f$ is a \textit{thinning} of $F$. Similarly, we say that $F$ is a \textit{thickening} of $f$, and \textit{thickening} $f$ refers to the process of associating to $f$ a thick interval map $F$ that thins to $f$.
\end{defn}

\begin{prop}\label{prop:thinterval}[Theorem 2 in \cite{dC}]
If $F$ has primitive transition matrix $M$ then thinning $F$ always produces a nontrivial interval $I$. Moreover, the induced map $f: I \to I$ also has transition matrix $M$.
\end{prop}

\begin{rmk}
Observe that while there is a unique thinning of a thick interval map, there are multiple thickenings of an interval map with at least one critical point. In Section 3 we investigate the proper thickening to choose for a zig-zag map if one hopes to produce a pseudo-Anosov.
\end{rmk}


\subsection{Train tracks and generalized pseudo-Anosovs}


In this subsection we describe how to associate to a given thick interval map $F$ a branched 1-manifold $\tau$ invariant under $F$, up to isotopy. This $\tau$ is called a \textit{generalized invariant train track} (cf. Definition \ref{defn:track}), and provides the blueprints for constructing a \textit{generalized pseudo-Anosov} $\phi: S^2 \to S^2$ (cf. Definition \ref{defn:gpA}). In particular, $\tau$ dictates the structure of the singular invariant foliations of $\phi$.\\

To a thick interval $\bbI$ we associate the data of a finite set $A$ of points, called \textit{punctures}. Each puncture is contained in a junction, and each junction contains at most one puncture. For a strip $s$ we define the arc $\gamma_s$ to be the path $\gamma_s(t)=h_s^{-1}(t, 1/2)$. Let RE denote the set of such paths. The endpoints of each arc $\gamma_s$ are on the boundary components of $s$ and are called \textit{switches}. We denote by $L$ the set of switches.

We again take the following definition from \cite{dCH}.

\begin{defn}\label{defn:track}
Given a thick interval $\bbI \subseteq S^2$ with a set of punctures $A$, a \textit{generalized train track} $\tau \subseteq \bbI \setminus A$ is a graph with vertex set $L$ and countably many edges, each of which intersects $\partial \bbV$ only at $L$, such that

\begin{enumerate}
\item The edges of $\tau$ which intersect the interior of $\bbE$ are precisely the elements of RE, and
\item no two edges $e_1$, $e_2$ contained in a given junction $J$ are \textit{parallel}: that is, $e_1$ and $e_2$ may only bound a disc if it contains a point of $A$ or another edge of $\tau$.
\end{enumerate}

Two generalized train tracks $\tau$ and $\tau'$ are \textit{equivalent}, denoted $\tau \sim \tau'$, if they are isotopic by an isotopy supported on $\bbV \setminus A$.
\end{defn}

Condition $1$ says that the edges of $\tau$ contained in the strips of $\bbI$ are simple to describe: they are elements $\gamma_s$ of RE, which are called \textit{real edges}. The more complicated edges are those contained in the junctions, which are called \textit{infinitesimal edges}. The collection of infinitesimal edges will be denoted by IE.

The infinitesimal edges will provide extra information not already given by the (finite) incidence matrix $M$ for some thick interval map $F$. Indeed, we associate to $F$ a specific generalized train track as follows. Let $\tau_0$ denote the (disconnected) generalized train track given by the real edges $\gamma_s \in$ RE. We apply $F$ to $\tau_0$ and then perform the following series of pseudo-isotopies:

\begin{enumerate}
\item On each strip $s$ we define the pseudo-isotopy $\psi_s: \overline{s} \times [0,1] \to \overline{s}$ by 

\[
\psi_s(x,y,t)=(x,(1-t)y+t/2)
\]

\item Within each junction we define another pseudo-isotopy $\psi_{e_1, e_2}$ for each pair of infinitesimal edges that are parallel, which homotopes $e_1$ and $e_2$ together. 
\end{enumerate}

The effect of the first set of pseudo-isotopies is to collapse all components of $F(\tau_0)$ contained within $\bbE$ to the real edges $\gamma_s$, while the second set homotopes parallel infinitesimal edges and is only supported on a disc containing the relevant junction. Composing these pseudo-isotopies produces a new generalized train track, denoted $\tau_1':=F_\ast(\tau_0)$. One may check that $\tau_1'$ is isotopic, relative to $A$, to a generalized train track $\tau_1$ containing $\tau_0$. Continuing in this way, we obtain an increasing sequence $\tau_0 \subseteq \tau_1 \subseteq \cdots $ of train tracks, and the union $\tau= \cup_{n \geq 0} \tau_n$ is \textit{F-invariant}, i.e., $F_\ast(\tau)$ is isotopic to $\tau$. See Figure \ref{fig:invt1}.

\begin{defn}
Let $\psi: S^2 \times I \to S^2$ denote the composition of the pseudo-isotopies in steps 1 and 2 above. The generalized train track $\tau=\cup_{n \geq 0} \tau_n$ is called the \textit{invariant generalized train track} for $F$. The \textit{train track map} associated to $F$ is the map $\phi: \tau \to \tau$ defined by $\phi(x)=\psi(F(x), 1)$. 
\end{defn}

Let $\bbI$ be a thick interval and $F:(S^2, \bbI, A) \to (S^2, \bbI, A)$ be a thick interval map with $A$ a finite invariant set of $F$. Let $\tau$ be the associated invariant generalized train track. $\tau$ has at most countably many edges. Label the finitely many real edges $e_1, \ldots, e_n$ and then label the possibly infinitely many infinitesimal edges $e_k$ for $k \geq n+1$. We form an \textit{extended transition matrix} $N=(n_{i,j})$ by setting

\[
n_{i,j}=\text{the number of times $\phi(e_j)$ crosses $e_i$}
\]

\noindent We may write this as the block matrix

\[
N=\begin{pmatrix}
M & 0\\
B & \Pi
\end{pmatrix}
\]

\noindent where $M$ is the incidence matrix of $F$, $B$ records the transitions from real to infinitesimal edges, and $\Pi$ the transitions from infinitesimal edges to other infinitesimal edges. 

\begin{landscape}
\thispagestyle{empty}
\begin{figure}
\centering
\includegraphics[scale=.15]{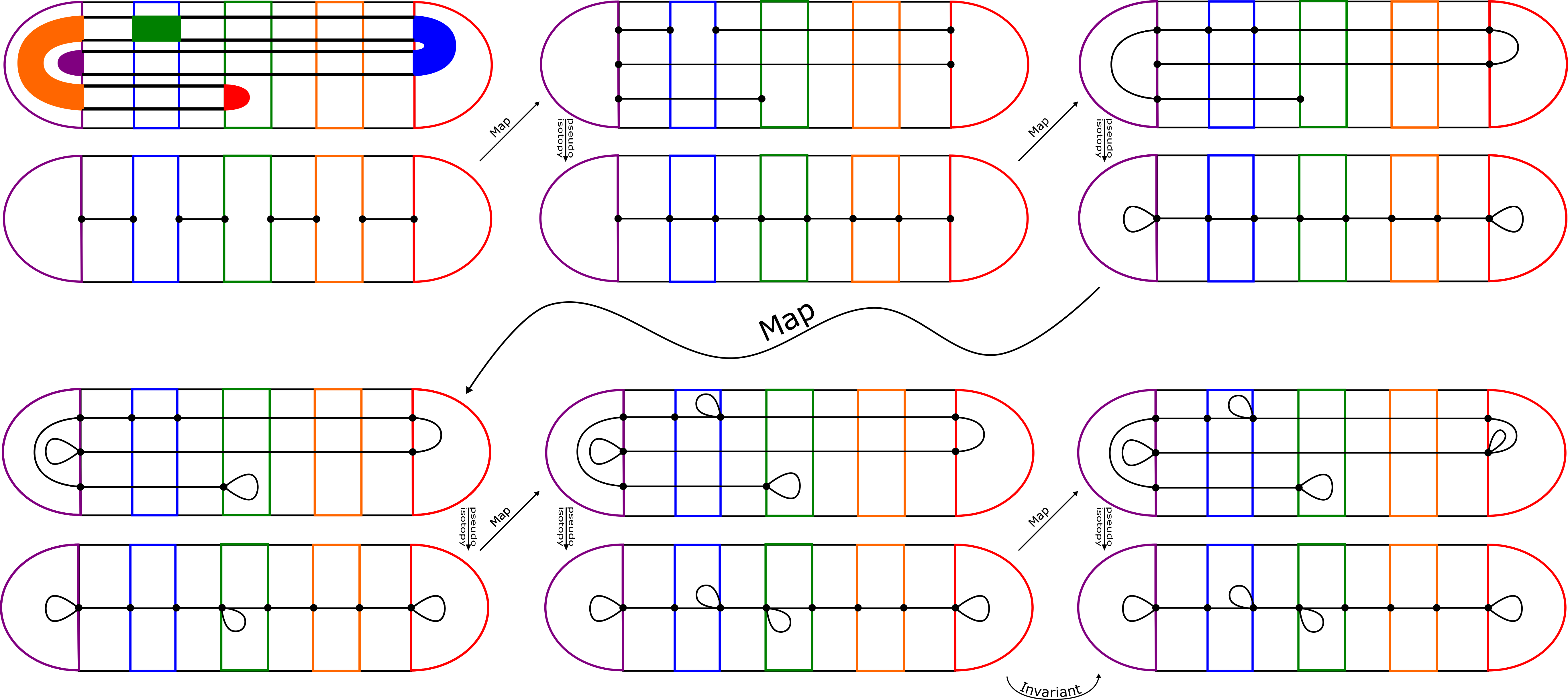}
\caption{The process for generating an invariant train track for a given thick interval map. One alternately applies the map to $\tau_n$ and the pseudo-isotopies to obtain $\tau_{n+1}$. In this case, the process terminates with $\tau_4$, which is invariant. Since the invariant generalized train track is finite, the resulting surface homeomorphism will be a pseudo-Anosov (cf. Figure \ref{fig:map}). In general, however, this process may continue indefinitely, producing a \textit{generalized} pseudo-Anosov with infinitely many singularities. See Figures \ref{fig:gpA} and \ref{fig:rect2} for such an example.}\label{fig:invt1}
\end{figure}
\end{landscape}

We assume from now on that $M$ is primitive. In this case, the Perron-Frobenius theorem states that the spectral radius $\rho(M)>0$ of $M$ is in fact a real eigenvalue of $M$, called the \textit{dominant eigenvalue} of $M$. The Perron-Frobenius theorem further states that the dominant eigenvalue is simple, and that the associated one-dimensional eigenspace is spanned by a positive eigenvector $u=(u_i)$, while no other eigenspace contains a positive eigenvector. We normalize this eigenvector to have unit $L^1$-norm: that is, $\sum_i u_i=1$.

If $\lambda=\rho(M)$ is the Perron eigenvalue of $M$, let $x'=(x_1, \ldots, x_n)$ denote the canonical positive left $\lambda$-eigenvector associated to $M$ such that $\sum_i x_i=1$, and let $y'=(y_1, \ldots, y_n)$ denote the positive right $\lambda$-eigenvector of $M$ such that $\sum_i x_iy_i=1$. One shows that these can be extended to left- and right- $\lambda$-eigenvectors of $N$, i.e. that there exist possibly infinite vectors $x$, $y$ such that $xN=\lambda x$ and $Ny=\lambda y$. Moreover, it is not hard to see that $x=(x_1, \ldots, x_n, 0 ,0, \ldots)$.

For $i=1, \ldots, n$, construct rectangles $R_i$ of dimensions $x_i \times y_i$. These are the building blocks of the surface on which the generalized pseudo-Anosov will act. The infinitesimal edges incident to one endpoint of a real edge $e_i$, along with their weights, encode how to identify segments of the corresponding vertical boundary of $R_i$. While this process is visually intuitive, a precise explanation is nonetheless elusive in the literature, so we describe it here for completeness.

In what follows, we fix a junction $J$ between two adjacent real edges $e_L$ and $e_R$. Denote by $v_L$ the endpoint of $e_L$ on $\partial J$, and similarly define $v_R$.

\begin{defn}\label{defn:end}
Let $e$ be an infinitesimal edge contained in $J$, considered as a smooth parameterized arc $e: [0,1] \to J$. Observe that $e(\{0,1\}) \subseteq \{v_L, v_R\}$. We define an \textit{end} of $e$ to be an arc of the form 

\[
\alpha=e\left ( \left [0,\frac{1}{2} \right ] \right ) \ \text{or} \ \alpha=e \left ( \left [\frac{1}{2},1\right ] \right ). 
\]

\noindent If $\alpha$ is an end of $e$, then for $0<\epsilon<\tfrac{1}{2}$ we define the \textit{$\epsilon$-subend} of $\alpha$ to be

\[
\alpha_\epsilon=\begin{cases}
e \left ( \left [0,\epsilon \right ] \right ) & \alpha=e\left ( \left [0,\frac{1}{2}\right ] \right ) \\
e \left ( \left [1-\epsilon, 1 \right ] \right ) & \alpha=e\left ( \left [\frac{1}{2},1\right ] \right )
\end{cases}
\]

\noindent If $e, f$ are two infinitesimal edges of $J$, not necessarily distinct, with ends $\alpha, \beta$ incident to $v_L$, we set

\[
\alpha \leq_L \beta \hspace{2mm} \text{if $\alpha_\delta$ is below $\beta_\delta$ for all $\delta>0$ sufficiently small}
\]

\noindent If instead $\alpha, \beta$ are incident to $v_R$, we set

\[
\alpha \leq_R \beta \hspace{2mm} \text{if $\alpha_\delta$ is below $\beta_\delta$ for all $\delta>0$ sufficiently small}
\]
\end{defn}

It is not difficult to see that $\leq_L$ is a total order on the set of ends incident to $v_L$. That is, for any two arcs $\alpha, \beta$ incident to $v_L$, we either have $\alpha \leq_L \beta$ or $\beta \leq_L \alpha$. Moreover, if 

\begin{itemize}
\item $e_{j_1}, e_{j_2}, \ldots$ are the infinitesimal edges with one end incident to $e_L$, 
\item $e_{k_1}, e_{k_2}, \ldots$ are the infinitesimal edges with two ends incident to $e_L$, and
\item the rectangle $R_{e_L}$ has height $y(L)$,
\end{itemize}

\noindent then the fact that $y$ is a right-$\lambda$-eigenvector for $N$ implies that we have

\begin{equation}\label{eqn:switch}
y(L)=\sum_i y_{j_i} + 2 \cdot \sum_l y_{k_l}
\end{equation}

\noindent Equation (\ref{eqn:switch}) is often referred to as the \textit{switch condition}.

For an end $\alpha$, denote by $e(\alpha)$ the infinitesimal edge of which $\alpha$ is an end, and for an infinitesimal edge $e$, let $y(e)$ denote the entry of the eigenvector $y$ for $N$ such that $y(e)=y_i$ if $e=e_i$. Then equation (\ref{eqn:switch}) and the fact that $\leq_L$ is a total order together imply that there is a unique way to partition the right vertical boundary of $R_{e_L}$ into segments $d_\alpha$ of length $y(e(\alpha))$ such that for two ends $\alpha$, $\beta$ incident to $e_L$, 

\[
\text{$d_\alpha$ is below $d_\beta \iff \alpha \leq_L \beta$}
\]

\noindent The same argument shows how to partition the left vertical boundary of $R_{e_R}$. It remains to describe how to identify these boundary segments. If $e$ is an infinitesimal edge incident to both $e_L$ and $e_R$, and $\alpha, \beta$ are the corresponding ends of $e$, then $y(e(\alpha))=y(e(\beta))=y(e)$ and we identify $d_\alpha$ on $\partial_VR_{e_L}$ with $d_\beta$ on $\partial_V R_{e_R}$ by an orientation-preserving isometry. If instead both ends $\alpha, \beta$ of $e$ are incident to $e_L$ then we identify the segments $d_\alpha, d_\beta$ on $\partial_VR_{e_L}$ by an orientation-reversing isometry, and similarly if both ends $\alpha, \beta$ are incident to $e_R$. See Figure \ref{fig:rect} for an example with finitely many infinitesimal edges, and Figure \ref{fig:rect2} for an example with infinitely many.

\begin{figure}[h!]
\centering
\includegraphics[scale=.2]{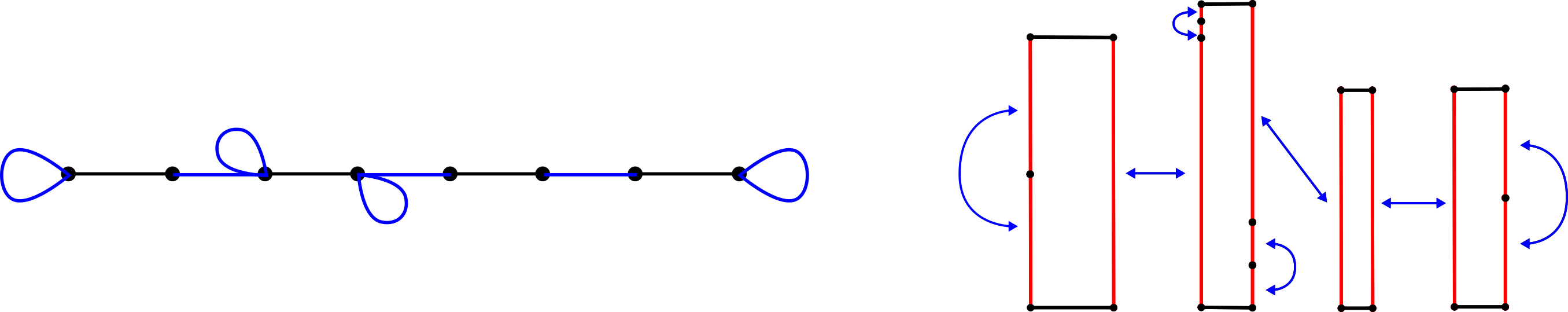}
\caption{The rectangle decomposition of a surface from the invariant train track. The black edges are real edges, while the blue edges are infinitesimal edges.}\label{fig:rect}
\end{figure}

The train track map $\phi$ induces an endomorphism $\tilde{\Phi}: \calR \to \calR$ which stretches the foliation of $\calR$ by horizontal lines by a factor of $\lambda$, and it scales the vertical foliation by a factor of $\lambda^{-1}$. This map $\tilde{\Phi}$ is a homeomorphism except on the boundary of $\calR$. This boundary is a topological circle and contains a periodic orbit of $\tilde{\Phi}$. After identifying adjacent segments of this circle that eventually map to the same segment, we obtain a homeomorphism $\Phi$ in the quotient. This new quotient surface is homeomorphic to $S^2$, and the periodic orbit on $\partial \calR$ becomes a single point, called the \textit{point at infinity}. The induced map $\Phi: S^2 \to S^2$ inherits stable and unstable foliations, and is a \textit{generalized pseudo-Anosov}.

\begin{figure}[h!]
\centering
\includegraphics[scale=.3]{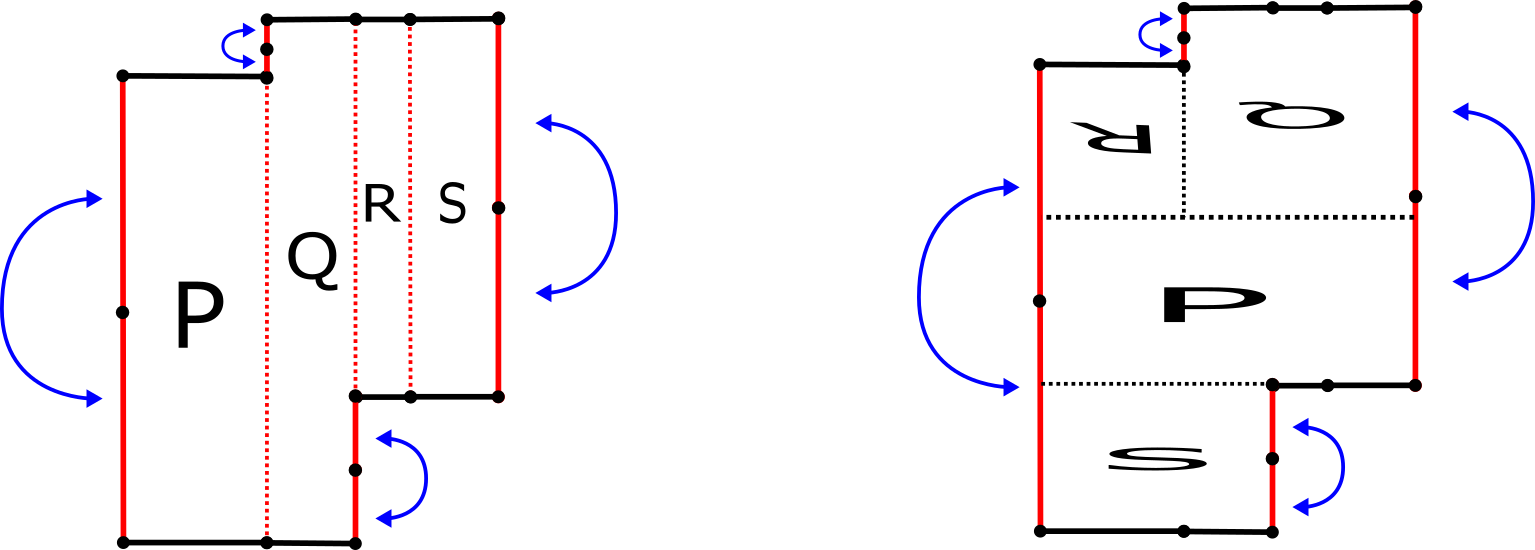}
\caption{The transformation $\tilde{\Phi}: \calR \to \calR$. Identifying points on the horizontal boundary that are eventually mapped to the same point produces a homeomorphism of $S^2$. Since there are only finitely many singularities, this is a pseudo-Anosov.}\label{fig:map}
\end{figure}

\begin{defn}\label{defn:gpA}
A \textit{generalized pseudo-Anosov} is a homeomorphism $\phi$ of a compact surface $S$ such that the following hold:
\begin{enumerate}
\item There exists a number $\lambda>1$ and two transverse singular measured foliations $(\calF_u, \mu_u)$, $(\calF_s, \mu_s)$ of $S$ such that

\[
\phi_\ast (\calF_u, \mu_u)=(\calF_u, \lambda \mu_u), \hspace{5mm} \phi_\ast(\calF_s, \mu_s)=(\calF_s, \lambda^{-1}\mu_s), 
\]
\item the singularities of $\calF_u$ and $\calF_s$, while potentially infinite in number, accumulate on only finitely many points of $S$
\end{enumerate}
\end{defn}

\begin{rmk}
Observe that in the case that the foliations have only finitely many singularities, $\phi$ is a pseudo-Anosov.
\end{rmk}

In what follows we will investigate closely the construction of generalized pseudo-Anosovs from thick interval maps, focusing particularly on the circumstances under which this process produces a pseudo-Anosov. This occurs if and only if the generalized invariant train track is finite, i.e. has only finitely many edges. In this case $\tau$ is simply a train track on a multiply-punctured sphere.
 
\begin{ex}\label{ex:gpA}
It will be instructive to see an example where the invariant generalized train track has infinitely many edges, so that the resulting homeomorphism $\Phi$ is not a pseudo-Anosov. Let $\lambda=1+\sqrt{2}$ and $f: I \to I$ the uniform $\lambda$-expander defined by

\[
f(x)=\begin{cases}
\lambda x & 0 \leq x < \lambda^{-1}\\
2-\lambda x & \lambda^{-1} \leq x < 2 \lambda^{-1}\\
\lambda x - 2 & 2 \lambda^{-1} \leq x \leq 1
\end{cases}
\]

We see that $f$ has two critical points, $c_1=\lambda^{-1}$ and $c_2=2\lambda^{-1}$. The first of these is $2$-periodic, with $f(c_1)=1$ and $f(1)=c_1$. On the other hand, $f(c_2)=0$ is a fixed point. Thus the weak postcritical set of $f$ is $\WPC(f)=\{0, c_1, c_2, 1\}$. We thicken this set to junctions and the intermediate subintervals to thick edges.

\begin{figure}[h!]
\centering
\includegraphics[scale=.17]{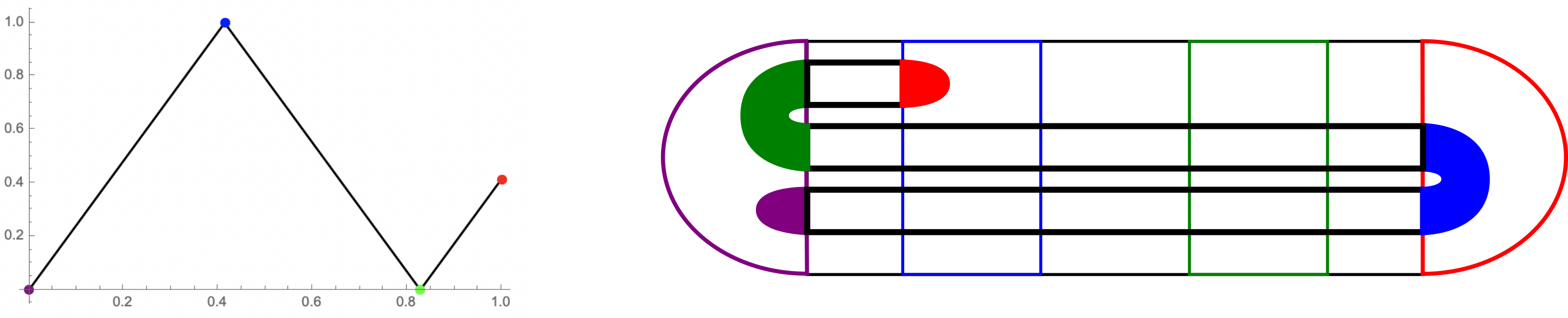}
\caption{A uniform $\lambda$-expander for $\lambda=1+\sqrt{2}$ and a thickening of it. Note that in the first junction the tightening pseudo-isotopies will produce non-parallel loops of the generalized invariant train track.}\label{fig:gpA}
\end{figure}

Figure \ref{fig:gpA} shows an example of a thickening $F$ of $f$. In this case the generalized invariant train track for $F$ has infinitely many edges, corresponding to infinitely many singularities for the resulting sphere homeomorphism. See Figure \ref{fig:rect2}. It is not hard to show that any orientation-\textit{preserving} thickening of $f$ will fail to produce a pseudo-Anosov. Interestingly enough, there does exist an orientation-reversing map that accomplishes this, and indeed if we define the Galois lift $f_G$ of $f$ by replacing all instances of $\lambda$ with its conjugate $-\lambda^{-1}$, then the limit set $\Lambda_f$ is rectangular.

\begin{figure}[h!]
\centering
\includegraphics[scale=.2]{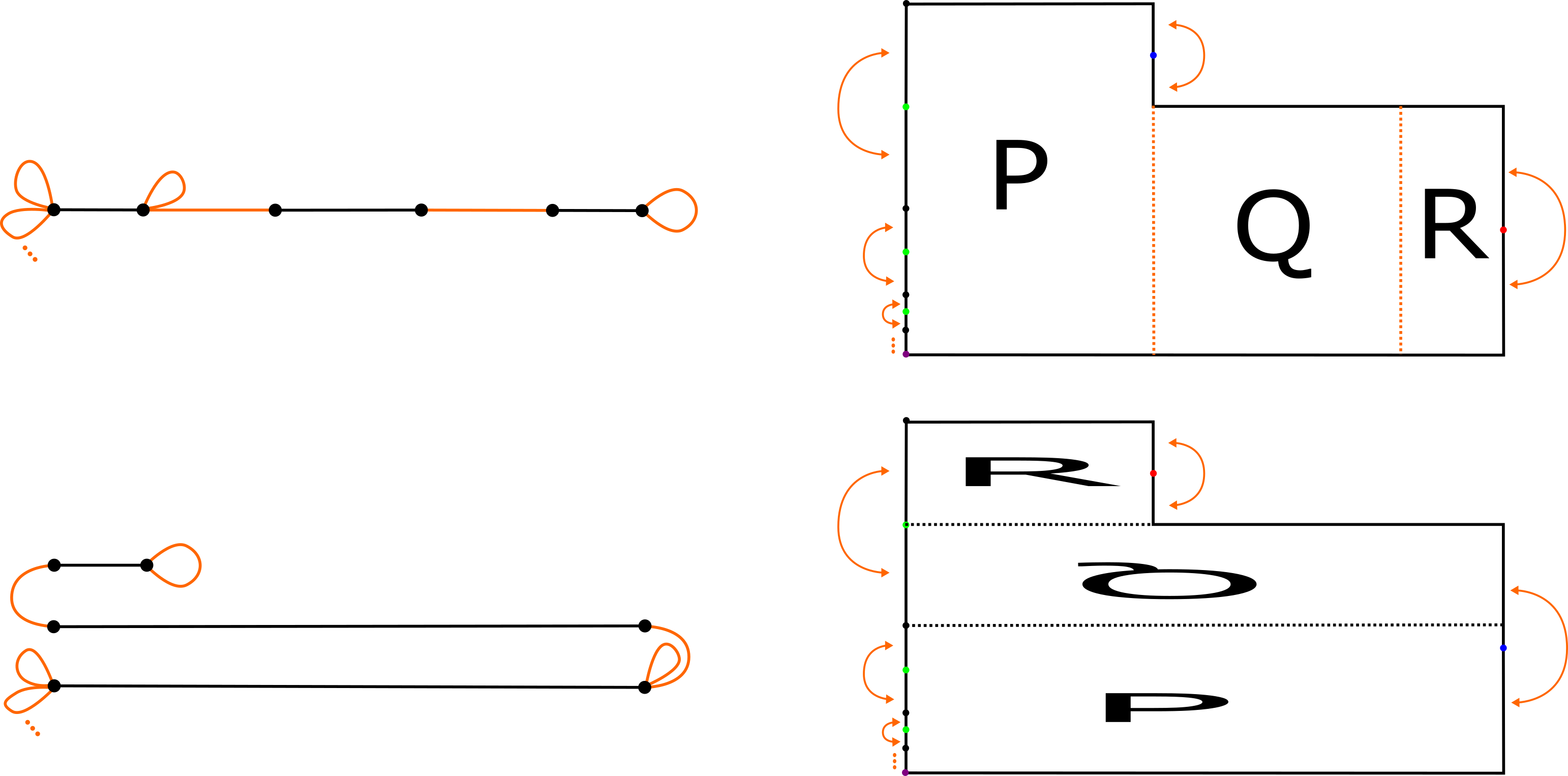}
\caption{The action of $F$ on its generalized invariant train track, and the induced map on the corresponding surface.}\label{fig:rect2}
\end{figure}

\end{ex}


 \section{First constructions}
 

We turn our attention to the task of constructing a pseudo-Anosov on $S^2$ from a zig-zag map (cf. Definition \ref{defn:zig-zag}). Our goal in this section is to show that for a zig-zag map there is a unique pair of thickenings to consider when attempting to construct a pseudo-Anosov: that is, no other possible thickenings can give a pseudo-Anosov (cf. Theorem \ref{veer}). Moreover, these thickenings produce conjugate generalized pseudo-Anosovs, so it suffices to only consider one of them.\\

Before we proceed to the proof of Theorem \ref{veer}, we must consider how to thicken zig-zag maps, and whether they can be thickened in the first place. For this we will rely on Proposition \ref{prop:thinterval}.


\subsection{Some ergodic theory for zig-zag maps}


In this subsection we show that every zig-zag map is weak-mixing. It follows that the transition matrix for a postcritically finite $\lambda$-zig-zag with $\lambda>2$ is primitive.
 
\begin{defn}
A \textit{critical point} of a map $f: I \to I$  is a point $x \in I$ where $f: I \to I$ is not a local homeomorphism. We will say that $f$ is \textit{multimodal} if it has finitely many critical points. If $c$ is a critical point, then $f(c)$ is a \textit{critical value}. Any point of the form $f^n(c)$ for $n \geq 1$ and $c$ a critical point is a \textit{postcritical point}.
\end{defn}
 
\begin{defn}
The \textit{critical set} $\C(f)$ will mean the set of critical points of $f$, and the union of the postcritical points of $f$ is the \textit{postcritical set}, denoted $\PC(f)$. Note in particular that a critical point need not be in the postcritical set, and so we define the \textit{weak postcritical set} $\WPC(f)$ to be the union of $\C(f)$ and $\PC(f)$.

If $\#\WPC(f)$ is finite, we say $f$ is \textit{postcritically finite}, or \textit{PCF}.
\end{defn}

Throughout this paper we will assume that all interval maps $f: I \to I$ are multimodal. In other words, $\#\C(f)<\infty$ and $f$ is monotone between each adjacent pair of critical points.

\begin{defn}\label{defn:Markov}
A \textit{Markov partition} for a map $f: I \to I$ is a decomposition $I=I_1 \cup \cdots \cup I_l$ into finitely many subintervals such that

\begin{enumerate}
\item $\inte(I_i) \cap \inte(I_j) = \emptyset$ if $i \neq j$,
\item for each $i$, $f(\overline{I_i})$ is a union of $\overline{I_{j_k}}$'s, and
\item for each $i$, the restriction $f|_{\inte(I_i)}$ is injective.
\end{enumerate}

\noindent In particular, if $f$ has a Markov partition then $f$ is postcritically finite.
\end{defn}

\begin{rmk}\label{Markov}
Since we will always assume that $\#\C(f)<\infty$, it follows that $f$ has a Markov partition if and only if $f$ is PCF. Moreover, $f$ has a unique minimal Markov partition in terms of inclusion: namely, the partition $W$ obtained by cutting the interval at the points of $\WPC(f)$. Unless otherwise specified, this is the Markov partition we will use. 
\end{rmk}

\begin{defn}\label{def:transition2}
Let $f: I \to I$ be a PCF multimodal map with Markov partition $\calP=\{I_1, \ldots, I_l\}$. The \textit{transition matrix} of $f$ is the $l \times l$ matrix $M=(m_{ij})$ such that

\[
m_{ij}=\begin{cases}
1 & \text{if $f(\overline{I_j}) \supseteq \overline{I_i}$}\\
0 & \text{otherwise}
\end{cases}
\]
\end{defn}

Recall that a multimodal $f: I \to I$ is \textit{PCP} if each critical value $f(c)$ is periodic. Observe that a zig-zag map is PCP if and only if the point $x=1$ is periodic.

In \cite{H} Hall classifies the $\lambda$-zig-zags of pseudo-Anosov type for $1<\lambda \leq 2$. A unimodal zig-zag map is called a \textit{tent map}. Much of our focus will be on PCP $\lambda$-zig-zags for $\lambda>2$. It is therefore important for the generalized pseudo-Anosov construction that we ensure such maps have primitive transition matrix $M$. Theorem \ref{thm:wmixing} essentially accomplishes this goal, and its proof uses a result by Wilkinson \cite{W}.

\begin{thm}\label{thm:wmixing}
If $f: I \to I$ is a $\lambda$-zig-zag for some $\lambda>2$, then $f$ is weak-mixing with respect to Lebesgue measure.
\end{thm}

\begin{proof}
Let $\floor{\lambda}=m \geq 2$ be the number of critical points of $f$ and set $P_0=[0,c_1)$, $P_m=[c_m,1]$, and $P_i=[c_i, c_{i+1}]$ for $i=1, \ldots, m-1$. Following \cite{W}, set

\[
\Delta(j_1, \ldots, j_n)=P_{j_1} \cap f^{-1}P_{j_2} \cap \cdots \cap f^{-(n-1)}P_{j_n}
\]

We say $\Delta(j_1, \ldots, j_n)$ is full of rank $n$ if $\mu(f^n(\Delta(j_1, \ldots, j_n)))=1$, where $\mu$ is Lebesgue measure; otherwise $\Delta(j_1, \ldots, j_n)$ is said to be non-full.

For the $n$-tuple $(j_1, \ldots, j_n)$ let $\frakl(j_1, \ldots, j_n)$ be the number of non-full intervals of positive measure of the form $\Delta(j_1, \ldots, j_n, i)$, for $1 \leq i \leq m+1$. Define $\frakl_n=\sup \frakl(j_1, \ldots, j_n)$, where the supremum is taken over all $n$-tuples $(j_1, \ldots, j_n)$ such that $\Delta(j_1, \ldots, j_n)$ has positive Lebesgue measure. Wilkinson shows in \cite{W} that $f$ is weak-mixing as long as 

\[
\frakl=\sup_n \frakl_n < \lambda
\]

\noindent However, note that all $P_i$ are full for $i=0, \ldots, m-1$. The full subsets of $\Delta(j_1, \ldots, j_n)$ of rank $n+1$ and positive measure are of the form $\Delta(j_1, \ldots, j_n, j)$ for $0 \leq j \leq J$, where $J$ depends on the ordered $n$-tuple $(j_1, \ldots, j_n)$. If $m \geq 2$ the only one of these that can be non-full is $\Delta(j_1, \ldots, j_n, J)$, hence $\frakl \leq 1 < \lambda$.
\end{proof}

\begin{rmk}
If a given zig-zag map $f$ is of pseudo-Anosov type (cf. Definition \ref{defn:pAtype} below), then it follows that $f$ is in fact \textit{mixing}.
\end{rmk}

\begin{defn}\label{defn:density}
A subset $E \subseteq \bbN$ is said to have \textit{density 0} if

\[
\lim_{n \to \infty} \frac{\# \left (E \cap \{1, \ldots, n\} \right )}{n} = 0
\]
\end{defn}

The union of two sets of density 0 also has density 0. It is well-known that a dynamical system $(X, \calB, \mu, T)$ is weak-mixing if and only if for every $A, B \in \calB$ there exists a set $E=E(A,B) \subseteq \bbN$ of density $0$ such that 

\[
\lim_{E \not \ni n \to \infty} \mu(T^{-n}A \cap B)=\mu(A) \mu(B)
\]

Recall that a non-negative matrix $M$ is \textit{primitive} if there is some $k \in \bbN$ such that every entry of $M^k$ is positive.

\begin{cor}\label{primitive}
If $f: I \to I$ is a PCP $\lambda$-zig-zag map for some $\lambda>2$, then the transition matrix of $f$ is primitive.
\end{cor}

\begin{proof}
Let $I_j, I_k$ be two subintervals in the Markov partition for $f$ obtained by cutting at the points of $\WPC(f)$. Since $f$ is weak-mixing and both $I_j$ and $I_k$ have positive measure, there exists a subset $E_{jk} \subseteq \bbN$ of density 0 such that $m(I_j \cap f^{-i}(I_k))$ is positive for large $i \not \in E_{jk}$. The set $E=\cup_{j,k} E_{jk}$ also has density 0, so there exists some large $i$ so that $m(I_j \cap f^{-i}(I_k))>0$ for all $j,k$. In particular, if $M$ is the transition matrix associated to the Markov partition, then $M^i$ is a positive matrix.
\end{proof}


\subsection{Passing between intervals and thick intervals}


\begin{defn}\label{defn:thin}
Recall that from a thick interval map $F: (S^2, \bbI) \to (S^2, \bbI)$ we obtain an interval map $f: I \to I$ by first collapsing all decomposition elements of $\bbI$ to points, and then further collapsing any subintervals on which the induced map is constant. We call this composition of quotient maps the \textit{thinning projection} and denote it by $\pi: \bbI \to I$.
\end{defn}

Note that by definition, $\pi \circ F|_{\bbI} = f \circ \pi$.

\begin{lem}
Let $F:(S^2, \bbI) \to (S^2, \bbI)$ be a thickening of $f: I \to I$. Then for each $x \in \PC(f)$ there exists a junction $J_x \subseteq \bbI$ such that $\pi(J_x)=x$. In particular, since thick intervals have only finitely many junctions, $f$ cannot be thickened unless it is PCF.
\end{lem}

\begin{proof}
Let $\pi: \bbI \to I$ denote the thinning projection. Let $c$ be a critical point of $f$, and let $C=\pi^{-1}(c)$ be the set of decomposition elements projecting to $c$. Similarly, let $C'=\pi^{-1}(f(c))$.

We may assume without loss of generality that $c$ is a local maximum of $f$. Therefore, there exist points $x_1, x_2$ satisfying $x_1<c<x_2$ such that $f(x_1)=f(x_2)<f(c)$. Let $X_i=\pi^{-1}(x_i)$ for $i=1, 2$. Then each $F(X_i)$ lies to the left of $C'$, whereas $F(C) \subseteq C'$. Since $F$ is linear with respect to the coordinates $h_s$ (cf. Definition \ref{defn:thickinterval}), $F(\bbI)$ must pass through a junction $J$ after $F(X_1)$ and before $F(X_2)$. Picking $x_i$ arbitrarily close to $c$ shows that in fact $F(C) \subseteq J$. But now $J \cap C' \neq \emptyset$, and since $C'$ is the collection of all decomposition elements that project to $f(c)$ it follows that $J \subseteq C'$.

$F$ maps junctions into junctions, so $F^n(C')$ must be contained in a junction for each $n \geq 1$. But $\pi(F^n(C'))=f^n(f(c))$, so by the same argument as before, $\pi^{-1}(f^n(f(c)))$ contains a junction for each $n$. Repeating this procedure for all critical points shows that $\pi^{-1}(x)$ contains a junction for each $x \in \PC(f)$.
\end{proof}

\begin{ex}
We revisit Example \ref{ex:gpA}, demonstrated in Figure \ref{fig:gpA}. $\PC(f)$ consists of the first, second, and fourth heavily drawn points, and the thick interval map $F$ has a junction for each of these. Note that $F$ also has a junction for the third point $c_2=2\lambda^{-1}$, despite the fact that $c_2$ is not a postcritical point. This junction is not strictly necessary, since no junction maps into it, but it does allow us to identify the transition matrices of $F$ and $f$: each is given by

\[
M_W=\begin{pmatrix}
1 & 1 & 1 \\
1 & 1 & 0\\
1 & 1 & 0
\end{pmatrix}
\]

\noindent One checks that $\chi_W(t)=t(t^2-2t-1)$ has dominant root $\lambda=1+\sqrt{2}$, the growth rate of $f$. Eliminating the extraneous junction over $c_2$ has the effect of combining the second and third subintervals, producing a new transition matrix

\[
M_P=\begin{pmatrix}
1 & 2 \\
1 & 1
\end{pmatrix}
\]

\noindent with characteristic polynomial $\chi_P(t)=t^2-2t-1$. See Remark \ref{partition} below.
\end{ex}

\begin{defn}\label{defn:pAtype}
We say that a map $f: I \to I$ is \textit{of pseudo-Anosov type} if there is a thickening $F: \bbI \to \bbI$ of $f$ whose invariant generalized train track is finite, i.e. is a train track in the classical sense of Thurston.
\end{defn}

\begin{thm}\label{PCP}
Let $f: I \to I$ be a PCF $\lambda$-expander of pseudo-Anosov type, and let $F$ be a thickening of $f$ that has finite invariant generalized train track $\tau$. Then $\tau$ has exactly one loop in each fat vertex corresponding to an element of $\PC(f)$. Moreover, these loops determine the one-pronged singularities of the pseudo-Anosov $\phi$, and $f$ is in fact PCP.
\end{thm}

\begin{proof}
Each loop of $\tau$ determines a one-pronged singularity of $\phi$. Multiple loops in a single vertex $V$ would imply that $\phi$ has singularities connected by a leaf of its stable (vertical) foliation. It is well known that such a thing, called a \textit{saddle connection}, is impossible for the invariant foliation of a pseudo-Anosov. Thus each vertex $V$, corresponding to a point $v \in \WPC(f)$, contains either 0 or 1 loops of $\tau$. We claim that $V$ contains a loop of $\tau$ if and only if $v$ is in the subset $\PC(f)$.

To see this, observe that $V$ contains a loop if $v$ is a critical value of $f$, since in this case $F(\tau)$ makes a turn through $V$, and after pseudo-isotopy this turn pinches to a loop. By the invariance of $\tau$, it now follows that $V$ contains a loop if $v$ is in the forward orbit of a critical point, which is to say that $v \in \PC(f)$. The only other possibility remaining for $v$ is that it is an element of $\WPC(f) \setminus \PC(f)$, i.e. a critical point that is not in the forward orbit of any critical point. But then no junction of $\bbI$ maps into $V$, so no loop of $\tau$ maps into $V$. Furthermore, $F(\tau)$ does not make a turn through $V$, since that would imply that $v$ is a critical value, which is it not. These are the only ways a loop of $\tau$ will appear in $V$ during the process of constructing $\tau$, so in fact $V$ does not contain a loop of $\tau$ if $v \not \in \PC(f)$.

To finish the proof, recall that a pseudo-Anosov permutes its $1$-prong singularities. Hence $f$ acts on the elements of $\PC(f)$ by a permutation: in other words, $f$ is PCP.
\end{proof}

\begin{defn}\label{defn:connecting}
Let $f: I \to I$ be a PCP $\lambda$-expander of pseudo-Anosov type, and let $\tau$ be the invariant train track associated to a thickening $F$ of $f$. We say an infinitesimal edge of $\tau$ is a \textit{connecting edge} if it is not a loop, i.e. if it joins distinct adjacent real edges.
\end{defn}

\begin{rmk}\label{partition}
In a sense, Theorem \ref{PCP} demonstrates that the important points of $f$ to consider are the elements of $\PC(f)$, rather than those of $\WPC(f)$, since these are the points that correspond to the singularities of any generalized pseudo-Anosov obtained from a thickening of $f$. One could define a transition matrix according to the partition of $I$ given by $\PC(f)$, although this partition would not technically be Markov, since $f$ might fail to be injective on each subinterval. Since injectivity will be helpful in our arguments, we will resort to using the weak postcritical Markov partition. In any case, there is a straightforward relationship between the two transition matrices and their characteristic polynomials. Specifically, the weakly postcritical transition matrix $M_W$ is primitive if and only if the postcritical transition matrix $M_P$ is as well, and the relationship between the characteristic polynomials is $\chi_W(t)=t^a \chi_P(t)$, where $a=\#(\WPC(f) \setminus \PC(f))$. 
\end{rmk}

\begin{figure}
\centering
\includegraphics[scale=.4]{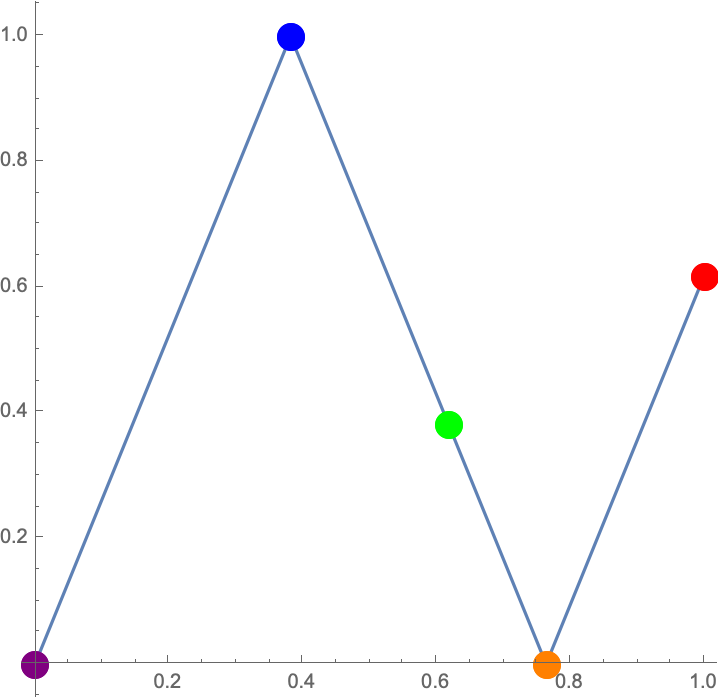}
\caption{The positive zig-zag map for $\lambda=(3+\sqrt{5})/2$. Compare with Figures \ref{fig:invt1} and \ref{fig:types}.}\label{fig:pos}
\end{figure}

\begin{ex}\label{ex:Markov1}
Let $f: I \to I$ be the positive zig-zag map of growth rate $\lambda=(3+\sqrt{5})/2$, the dominant root of $x^2-3x+1$. The orbit of $x=1$ is periodic of period 3 and includes the first critical point $c_1=\lambda^{-1}$. The other postcritical orbit is the forward orbit of $x=c_2$, which maps to the fixed point at $x=0$. See Figure \ref{fig:pos}. Thus 

\[
C(f)=\{c_1, c_2\}, \hspace{5mm} \PC(f)=\{0, c_1, v, 2\}, \hspace{5mm} \WPC(f)=\{0, c_1, v, c_2, 1\}
\]

In particular, $f$ is postcritically periodic, acting on $\PC(f)$ by the permutation $(1)(2, 4, 3)$. The postcritical and weak postcritical transition matrices for $f$ are

\[
M_P=\begin{pmatrix}
1 & 0 & 2\\
1 & 1 & 1\\
1 & 1 & 0
\end{pmatrix},
\hspace{2cm}
M_W=\begin{pmatrix}
1 & 0 & 1 & 1\\
1 & 1 & 0 & 1\\
1 & 1 & 0 & 0\\
1 & 1 & 0 & 0
\end{pmatrix}
\]

\noindent with characteristic polynomials $\chi_P(t)=(t+1)(t^2-3t+1)$ and $\chi_W(t)=t(t+1)(t^2-3t+1)$. $M_P$ and $M_W$ are each primitive: we have

\[
M_P^2=\begin{pmatrix}
3 & 2 & 2\\
3 & 2 & 3 \\
2 & 1 & 3
\end{pmatrix},
\hspace{2cm}
M_W^2=\begin{pmatrix}
3 & 2 & 1 & 1 \\
3 & 2 & 1 & 2\\
2 & 1 & 1 & 2\\
2 & 1 & 1 & 2
\end{pmatrix}
\]
\end{ex}

\begin{ex}\label{ex:Markov2}
Here is an example of a tent map, i.e. a zig-zag of growth rate $1 < \lambda < 2$. It is not hard to show that positive tent maps are not ergodic: indeed, $x=0$ is a repelling fixed point with only itself as a pre-image. Nonetheless, restricting to the subinterval $[f(1),1]$ produces an ergodic transformation with the same growth rate.

Let $g: I \to I$ be the positive zig-zag of growth rate $(1+\sqrt{5})/2$, the dominant root of $x^2-x-1$. The orbit of $x=1$ is periodic of period $3$ and includes the lone critical point $c=\lambda^{-1}$. See Figure \ref{fig:tent}. Thus $C(f)=\{c\}$ and $\PC(f)=\WPC(f)=\{0, u, c, 1\}$. The postcritical and weak postcritical transition matrices for $g$ coincide as 

\[
M=\begin{pmatrix}
1 & 0 & 0\\
1 & 0 & 1\\
0 & 1 & 1
\end{pmatrix}
\]

\noindent with characteristic polynomial $\chi_M(t)=(t-1)(t^2-t-1)$. Note that this is equal to the digit polynomial $D_f(t)$ of the tent map (cf. Example \ref{ex:tent}). Since the only subinterval that maps to the first is itself, $M$ cannot be primitive. However, the 2-by-2 minor describing the transitions between the second two subintervals is primitive:

\[
M'=\begin{pmatrix}
0 & 1 \\
1 & 1
\end{pmatrix},
\hspace{2cm} (M')^2=\begin{pmatrix}
1 & 1 \\
1 & 2
\end{pmatrix}
\]

\begin{figure}
\centering
\includegraphics[scale=.4]{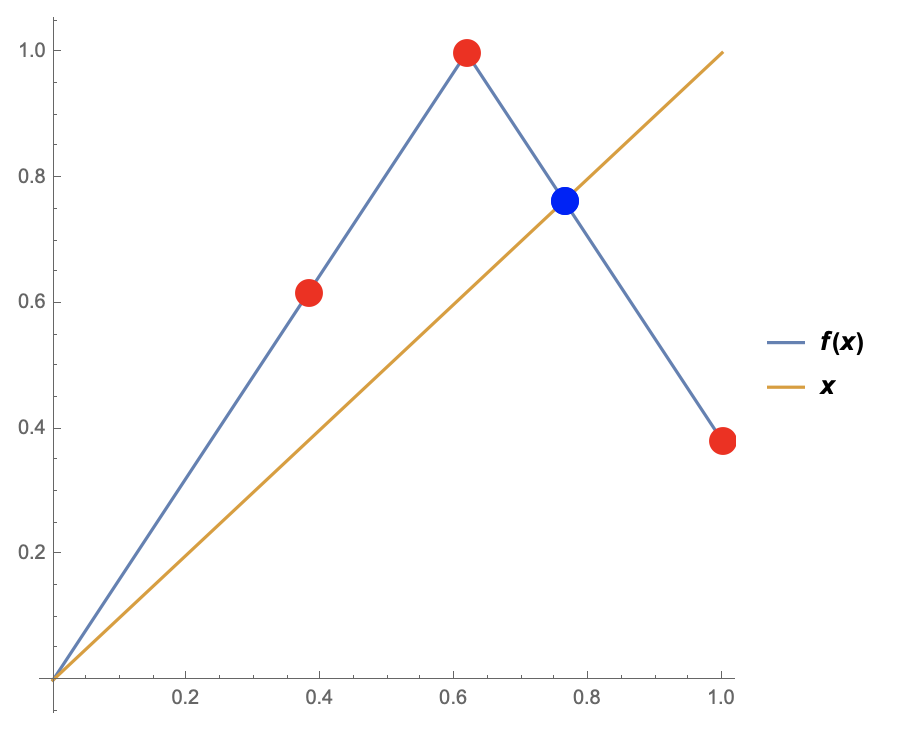}
\caption{The positive tent map for $\lambda=(1+\sqrt{5})/2$, with postcritical set shown in red.}\label{fig:tent}
\end{figure}
\end{ex}


\subsection{The exterior left-veering thickening}


In this subsection we investigate the possible thickenings of a PCP zig-zag, and show that only two of these have a chance of producing a pseudo-Anosov. These are the \textit{exterior left-} and \textit{exterior right-veering} thickenings $F_L$ and $F_R$ (cf. Definition \ref{defn:ext}). These thickenings ``swirl" out from the center, turning either always left or always right, respectively. Since these maps are topologically conjugate, we therefore restrict our analysis to $F_L$ in future sections.\\

Let $f$ be a zig-zag map of growth rate $\lambda>2$ and critical points $c_i=i \cdot \lambda^{-1}$ for $i=1, \ldots, \floor{\lambda}$. We assume that $f$ is PCP, which is equivalent to the orbit of $x=1$ being periodic. Let $0 \leq v_1<v_2<\ldots<v_n=1$ be the orbit of $1$. These, along with $x=0$ if it is not already among the $v_i$, are precisely the points to which the junctions of our thick interval $\bbI$ will project. To specify a particular thick interval map $F$ projecting to $f$, however, it remains to determine how $F$ folds $\bbI$, i.e. how the image $F(\bbI)$ turns within $\bbI$.

If $F: \bbI \to \bbI$ is a thickening of $f: I \to I$ with thinning projection $\pi$, then for brevity we will denote by \textbf{0} the junction satisfying $\pi(\textbf{0})=0$. Similarly, \textbf{1} will denote the junction satisfying $\pi(\textbf{1})=1$.

\begin{defn}
Let $f: I \to I$ be a PCP $\lambda$-zig-zag map with $\floor{\lambda}=m \geq 2$. Let $F: \bbI \to \bbI$ be a thickening of $f$. The image $F(\bbI) \subseteq \bbI$ is a collection of thick intervals stacked vertically stretching between \textbf{0} and \textbf{1} with at most a single exception. Orienting $\bbI$ from left to right, we number these thick subintervals $\bbI_j$ as we travel along $F(\bbI)$, beginning at \textbf{0} if $f$ is a positive zig-zag and at \textbf{1} if $f$ is a negative zig-zag. We may also assign a \textit{height} to each $\bbI_j$ by determining its vertical order among the other thick subintervals, counting from bottom to top. The height of $\bbI_j$ is defined to be $h(\bbI_j)=i$ if $\bbI_j$ is the $i$-th thick subinterval in this ordering. The \textit{type} of the thick interval map $F$ is then the element $\sigma$ of the permutation group $S_{m+1}$ such that

\[
\sigma(j)=i \iff h(\bbI_j)=i
\]

\noindent For an example, see Figure \ref{fig:types} below.
\end{defn}

\begin{ex}
Let $f$ be the positive zig-zag map for $\lambda=(3+\sqrt{5})/2$ from Example \ref{ex:Markov1}. We have $\floor{\lambda}=2$, and $f$ is postcritically periodic of length $3$ with $\PC(f)$ given by

\[
0<c_1=v_1<v_2<c_2<v_3=1
\]

\noindent More specifically, the orbit of $x=1$ is $1 \mapsto v_2 \mapsto v_1 \mapsto 1$. The six possible thickenings of $f$ are pictured in Figure \ref{fig:types}. These are paired according to the topological conjugacy class of the resulting generalized pseudo-Anosov. Note that the conjugacy classes in this example are the elements of the orbit space of $\iota_3$ acting on $S_3$, where $\iota_3$ is the order-reversing permutation

\[
\tau_3=\begin{pmatrix}
1 & 2 & 3\\
3 & 2 & 1
\end{pmatrix}
\]

\noindent See Definition \ref{defn:tau} below.
\end{ex}

\begin{figure}
\centering
\includegraphics[scale=.1]{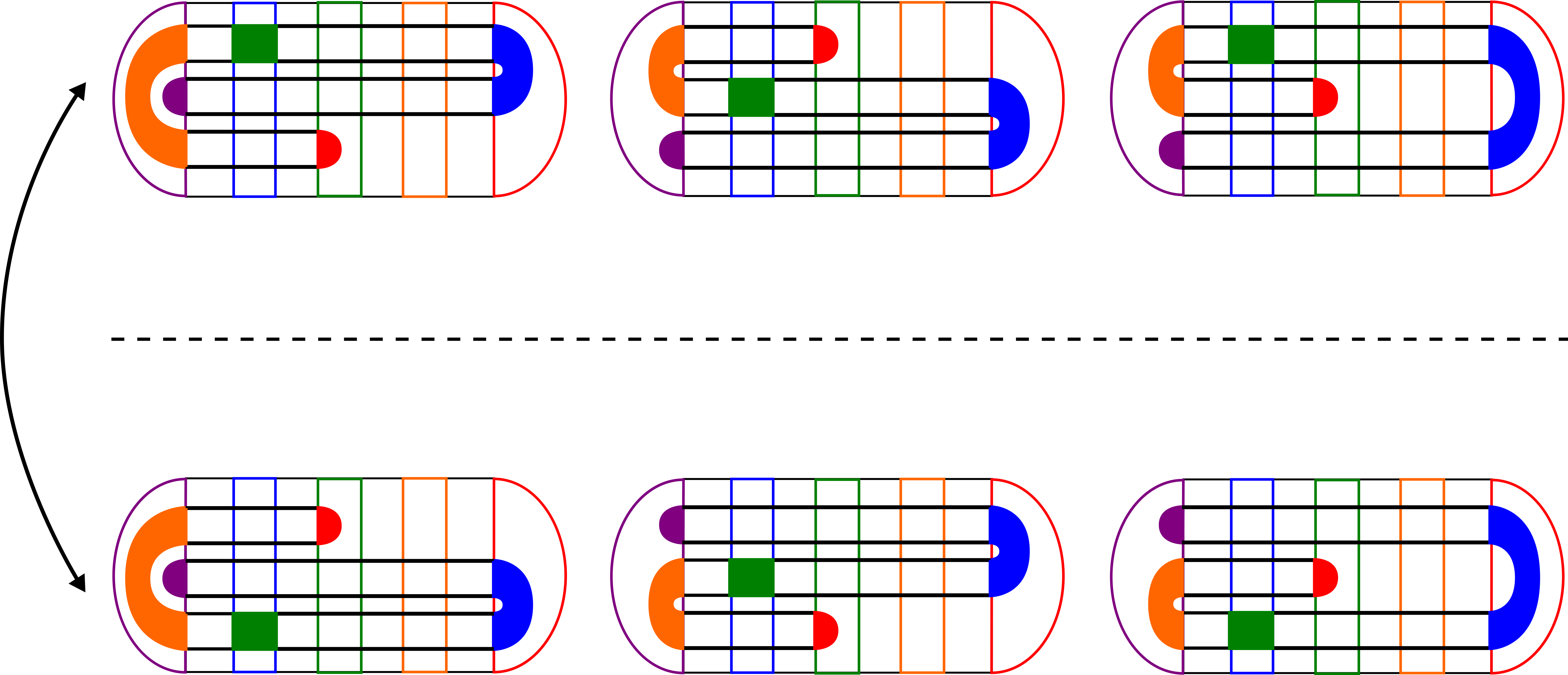}
\caption{The six possible thickenings of the positive zig-zag for $\lambda=(3+\sqrt{5})/2$ . The permutation type of each is an element of $S_3$, determined by the heights of the horizontal layers of $F(\bbI)$. For example, the top left thick interval map has permutation type $\sigma_1=(1, 2, 3)$, and the top right has permutation type $\sigma_2=(1)(2,3)$. Each column is an orbit of the action of $\tau_3$. }\label{fig:types}
\end{figure}

\begin{ex}
Not every permutation $\sigma \in S_{m+1}$ gives a valid thick interval map. For example, let $f$ be a PCP $\lambda$-zig-zag map such that $\floor{\lambda}=3$. Then regardless of the orbit of $x=1$, there can be no thickening of $f$ of type 

\[
\sigma=\begin{pmatrix}
1 & 2 & 3 & 4 \\
1 & 3 & 2 & 4
\end{pmatrix},
\]

\noindent since the image of this thick interval map would necessarily intersect itself. 
\end{ex}

\begin{defn}\label{defn:tau}
The \textit{orientation-reversing permutation} $\iota_n \in S_n$ is the element 

\[
\iota_n=\begin{pmatrix}
1 & 2 & \cdots & n\\
n & n-1 & \cdots & 1
\end{pmatrix}.
\]
\end{defn}

We will often suppress notation and write $\iota_n=\iota$ when $n$ is understood from context. Note that $\iota^2=\id$, so if we let $\iota_n$ act on $S_n$ on the left then the orbit space is parameterized by pairs of elements of $S_n$. In general $\iota_n$ is not central in $S_n$, so the orbit space is not a group. The following proposition says that the elements of the orbit space correspond to topological conjugacy classes of generalized pseudo-Anosovs.

\begin{prop}\label{conj}
Let $f$ be a PCP zig zap map and $F_1$, $F_2$ thickenings of $f$ of type $\sigma_1$, $\sigma_2$ respectively. If $\sigma_1=\iota \sigma_2$ then $F_1$ is topologically conjugate to $F_2$ via the orientation-reversing homeomorphism $i: \bbI \to \bbI$ that reflects through the horizontal midline of $\bbI$. Consequently, if $\phi_i$ is the generalized pseudo-Anosov induced by $F_i$, then $\phi_1$ is topologically conjugate to $\phi_2$. 
\end{prop}

\begin{proof}
One immediately checks that if $\sigma_1=\iota \sigma_2$ then $F_1=i \circ F_2 \circ i^{-1}$. Following through the details of construction given in Section 2, we see that the invariant generalized train tracks $\tau_i$ of $F_i$ satisfy $\tau_1 \sim i_\ast(\tau_2)$, and so the resulting generalized pseudo-Anosovs are conjugate by the homeomorphism induced by $i$.
\end{proof}

\begin{defn}\label{defn:ext}
For a fixed $n \geq 1$ the \textit{positive exterior left-veering permutation} $\sigma^+_L \in S_n$ is the permutation defined by

\[
\sigma^+_L=\begin{cases}\begin{pmatrix}
1 & 2 & 3 & 4 & \cdots & n-2 & n-1 & n\\
\frac{n}{2} & \frac{n}{2}+1 & \frac{n}{2}-1 & \frac{n}{2}+2 & \cdots & n-1 & 1 & n
\end{pmatrix} & \text{$n$ even}\\
\vspace{1mm}\\
\begin{pmatrix}
1 & 2 & 3 & 4 & \cdots & n-2 & n-1 & n\\
\frac{n+1}{2} & \frac{n+1}{2}+1 & \frac{n+1}{2}-1 & \frac{n+1}{2}+2 & \cdots & 2 & n & 1
\end{pmatrix} & \text{$n$ odd}
\end{cases}
\]

\noindent Similarly, the \textit{negative exterior left-veering permutation} $\sigma^-_L \in S_n$ is defined by

\[
\sigma^-_L=\begin{cases}\begin{pmatrix}
1 & 2 & 3 & 4 & \cdots & n-2 & n-1 & n \\
\frac{n}{2} & \frac{n}{2}-1 & \frac{n}{2}+1 & \frac{n}{2}-2 & \cdots & 2 & n & 1
\end{pmatrix} & \text{$n$ even}\\
\vspace{1mm}\\
\begin{pmatrix}
1 & 2 & 3 & 4 & \cdots & n-2 & n-1 & n\\
\frac{n+1}{2} & \frac{n+1}{2}-1 & \frac{n+1}{2}+1 & \frac{n+1}{2}-2 & \cdots & n-1 & 1 & n
\end{pmatrix} & \text{$n$ odd}
\end{cases}
\]

\noindent The \textit{positive (resp. negative) exterior right-veering permutation} is defined to be $\sigma^{\pm}_R=\iota \circ \sigma^{\pm}_L \in S_n$. If $F$ is a thick interval map of type $\sigma^{\pm}_L$ we also say $F$ is \textit{exterior left-veering}, and write $F=F^\pm_L$. Similarly we denote by $F^\pm_R$ the \textit{exterior right-veering} thick interval of permutation type $\sigma^{\pm}_R$.  
\end{defn}

\begin{ex}
For $S_3$, $\sigma^+_L$ and $\sigma^+_R$ are shown in Figure \ref{fig:types} as the top and bottom of the leftmost column, respectively. For $S_4$ we have 

\[
\sigma^+_L=\begin{pmatrix}
1 & 2 & 3 & 4\\
2 & 3 & 1 & 4
\end{pmatrix}, \hspace{5mm}
\sigma^+_R=\begin{pmatrix}
1 & 2 & 3 & 4\\
3 & 2 & 4 & 1
\end{pmatrix}
\]
\end{ex}

\begin{rmk}
Given a PCP zig-zag $f$, only one pair of thickenings is defined: either $F^+_L$ and $F^+_R$ if $f$ is positive, or $F^-_L$ and $F^-_R$ if $f$ is negative. In this case we will drop the superscripts from the notation and simply refer to $F_L$ and $F_R$.
\end{rmk}

\begin{prop}\label{veer}
Let $f: I \to I$ be a PCP zig-zag map of pseudo-Anosov type with growth rate $\lambda>2$, and let $F: \bbI \to \bbI$ be a thickening of $f$ that induces a pseudo-Anosov. Then $F=F_L$ or $F=F_R$.
\end{prop}

\begin{proof}
Since $F$ induces a pseudo-Anosov, Theorem \ref{PCP} implies that each of \textbf{0} and \textbf{1} contains a single loop of the invariant generalized train track $\tau$. For each of these vertices, the turns of $F$ that pass through them must be concentric, since otherwise a second loop would appear. Orienting $F(\bbI)$ from $F(\textbf{0})$ to $F(\textbf{1})$, it follows inductively that each turn has to be in the same direction as the previous one. Thus if the first turn is to the left, all turns are to the left and $F=F_L$. Similarly, if the first turn is to the right then $F=F_R$. 
\end{proof}

It is always possible to construct a thick interval map $F$ having permutation type $\sigma_L$ or $\sigma_R$ projecting to a given PCP zig-zag map. Indeed, such a thick interval map swirls out from the center, turning to the left if it is has type $\sigma_L$ and to the right if it has type $\sigma_R$. In particular, it follows from Propositions \ref{conj} and \ref{veer} that associated to a given PCP zig-zag map $f: I \to I$ is a canonical thick interval map to consider when investigating whether $f$ is of pseudo-Anosov type: namely, $F_L$.


\section{Reconciling two constructions}


In this section we prove Theorem \ref{Galois}, establishing the connection between generalized pseudo-Anosovs and the Galois lift $f_G$ of Thurston (cf. Definition \ref{defn:Galois lift}.) We recall the statement of the theorem.

\addtocounter{mainthm}{-4}
\begin{mainthm}
Let $f: I \to I$ be a PCP $\lambda$-zig-zag map with $\lambda>2$. Then $f$ is of pseudo-Anosov type if and only if the following conditions are satisfied:

\begin{enumerate}
\item The digit polynomial $D_f$ of $f$ has $\lambda^{-1}$ as a root, and
\item the limit set $\Lambda_f$ of $f_G$ is rectangular.
\end{enumerate}

In this case, the invariant generalized train track $\tau_L$ of $F_L$ is finite, and recovers the action of $f_G$ on $\Lambda_f$ in the following way: Let $S'$ be the closed topological disc obtained by performing the gluings indicated by the non-loop infinitesimal edges of $\tau_L$. Let $\tilde{f}: S' \to S'$ be the map induced by $F_L$. Then there is an isometry $i: S' \to \Lambda_f$ such that the following diagram commutes:

\[
\begin{tikzcd}
S' \arrow{r}{\tilde{f}} \arrow{d}{i} & S' \arrow{d}{i}\\
\Lambda_f \arrow{r}{f_G} & \Lambda_f
\end{tikzcd}
\]

Moreover, $i$ sends the horizontal and vertical foliations of $S'$ to those of $\Lambda_f$. Therefore, after identifying segments of boundary in each set so as to obtain pseudo-Anosovs $\phi_1: S \to S$ and $\phi_2: \tilde{\Lambda}_f \to \tilde{\Lambda}_f$, these systems are conjugate via an isometry that sends the (un)stable foliation of $\phi_1$ to the (un)stable foliation of $\phi_2$.
\end{mainthm}


\subsection{The reverse direction}


Throughout this subsection we assume that $f: I \to I$ is a PCP $\lambda$-zig-zag whose Galois lift $f_G$ has rectangular limit set $\Lambda_f$ and whose digit polynomial $D_f$ satisfies $D_f(\lambda^{-1})=0$. We prove that the exterior left-veering thickening $F_L$ of $f$ has finite invariant generalized train track, and that the generalized pseudo-Anosov so obtained will recover the action of $f_G$ on $\Lambda_f$.

\begin{lem}\label{erg}
Let $\mu$ be the finite measure on $\Lambda_f$ inherited from Lebesgue measure on $\bbR^2$. Then $f_G: (\Lambda_f, \mu) \to (\Lambda_f, \mu)$ is measure-preserving. 
\end{lem}

\begin{proof}
For each $i$ the affine map $\tilde{f}_i$ defining $f_G$ on the rectangle $R_i$ is measure-preserving, having Jacobian $\left ( \begin{smallmatrix} \lambda & 0 \\ 0 & \lambda^{-1} \end{smallmatrix} \right )$. Indeed, these maps are invertible with measure-preserving inverse, so $\mu(\tilde{f}_i(R_i))=\mu(R_i)$. Therefore we have

\[
\mu(\Lambda_f) = \sum_{i=0}^k \mu(R_i) = \sum_{i=0}^k \mu(\tilde{f}_i(R_i)) \geq \mu(f_G(\Lambda_f))
\]

\noindent If the final inequality is strict then $A=\Lambda_f \setminus f_G(\Lambda_f)$ has positive measure. But since $\Lambda_f$ and $f_G(\Lambda_f)$ are each finite unions of rectangles, so is $A$, hence if $\mu(A)>0$ then $A$ has non-empty interior $U$. No point $x \in U \subseteq \Lambda_f$ can be a limit point of an orbit of $f_G$, contradicting the definition of $\Lambda_f$. Hence in fact $\mu(f_G(\Lambda_f))=\mu(\Lambda_f)$, and in particular 

\[
\sum_{i=0}^k \mu(\tilde{f}_i(R_i)) = \mu(f_G(\Lambda_f))
\]

\noindent In other words, $\mu(\tilde{f}_i(R_i) \cap \tilde{f}_j(R_j)) = 0$ for all indices $i \neq j$, and so for any measurable $B \subseteq \Lambda_f$ we have

\[
\mu(f_G^{-1}(B)) = \sum_{i=0}^k \mu(\tilde{f}_i^{-1}(B))=\sum_{i=0}^k \mu(B \cap R_i) = \mu(B)
\]

\noindent Thus $f_G$ preserves $\mu$.
\end{proof}

Let $M$ be the transition matrix for an ergodic PCF uniform $\lambda$-expander $g$. The map $g$ has a unique invariant measure that is absolutely continuous, i.e. that is in the equivalence class of Lebesgue measure. Denote this measure by $\nu_0$. If $M$ is primitive, then results of Parry in \cite{P} on subshifts of finite type imply that the Perron eigenvectors of $M$ determine $\nu_0$. Namely, let $u=(u_1, \ldots, u_n)$ denote the left $\lambda$-eigenvector of $M$ with $\sum_i u_i=1$, and let $v=(v_1, \ldots, v_n)$ denote the right $\lambda$-eigenvector of $M$ such that $u \cdot v=1$. Then the density of $\nu_0$ is (up to a null set) a step function with height $v_i$ on a subinterval of Lebesgue measure $u_i$. Note that in general if we use the right $\lambda$-eigenvector $v'=cv$ for $c>0$ then we obtain the invariant measure $\nu_0'=c \cdot \nu_0$, which is a probability measure exactly when $c=1$. 

\begin{lem}\label{measure}
Let $\pi: \Lambda_f \to I$ denote projection onto the first coordinate, and define $\nu=\pi_\ast\mu$. Then $\nu=\nu_0(I) \cdot \nu_0$.
\end{lem}

\begin{proof}
Since $\mu$ is $f_G$-invariant and $\pi \circ f_G = f \circ \pi$, $\nu$ is invariant with respect to $f$. Furthermore, because $\Lambda_f$ is a connected finite union of rectangles of positive measure, we see that $\nu$ is equivalent to Lebesgue measure, and in fact has invariant density a step function whose values are given by the heights of the rectangles of $\Lambda_f$. The result now follows by uniqueness of $\nu_0$. Note in general that $\nu(I)=\mu(\Lambda_f) \neq 1$.
\end{proof}

\begin{lem}\label{periodic}
Let $\alpha=\alpha_0$ be a periodic point of $f$ of least period $p$. Set $\alpha_i=f^i(\alpha_0)$ for $0 \leq i \leq p-1$, and let $g_i: \bbR \to \bbR$ denote the affine map defining $f$ at $x=\alpha_i$ so that $f(\alpha_i)=g_i(\alpha_i)=c_i \pm \lambda \alpha_i=\alpha_{i+1}$, where the indices are understood modulo $p$. Define $\widetilde{g}_i: \bbR \to \bbR$ by $\widetilde{g}_i(y)=a_i \pm \lambda^{-1}y$.

Then there is a unique periodic point of $f_G$ projecting to $\alpha_0$, and it is given by $(\alpha_0, y_0)$ where $y_0$ is the unique solution to the equation $\widetilde{g}_{p-1} \circ \cdots \circ \widetilde{g}_0(y)=y$. In the case $\alpha_0=1$ this periodic point is $(1,1)$.
\end{lem}

\begin{figure}[h!]
\centering
\includegraphics[scale=.18]{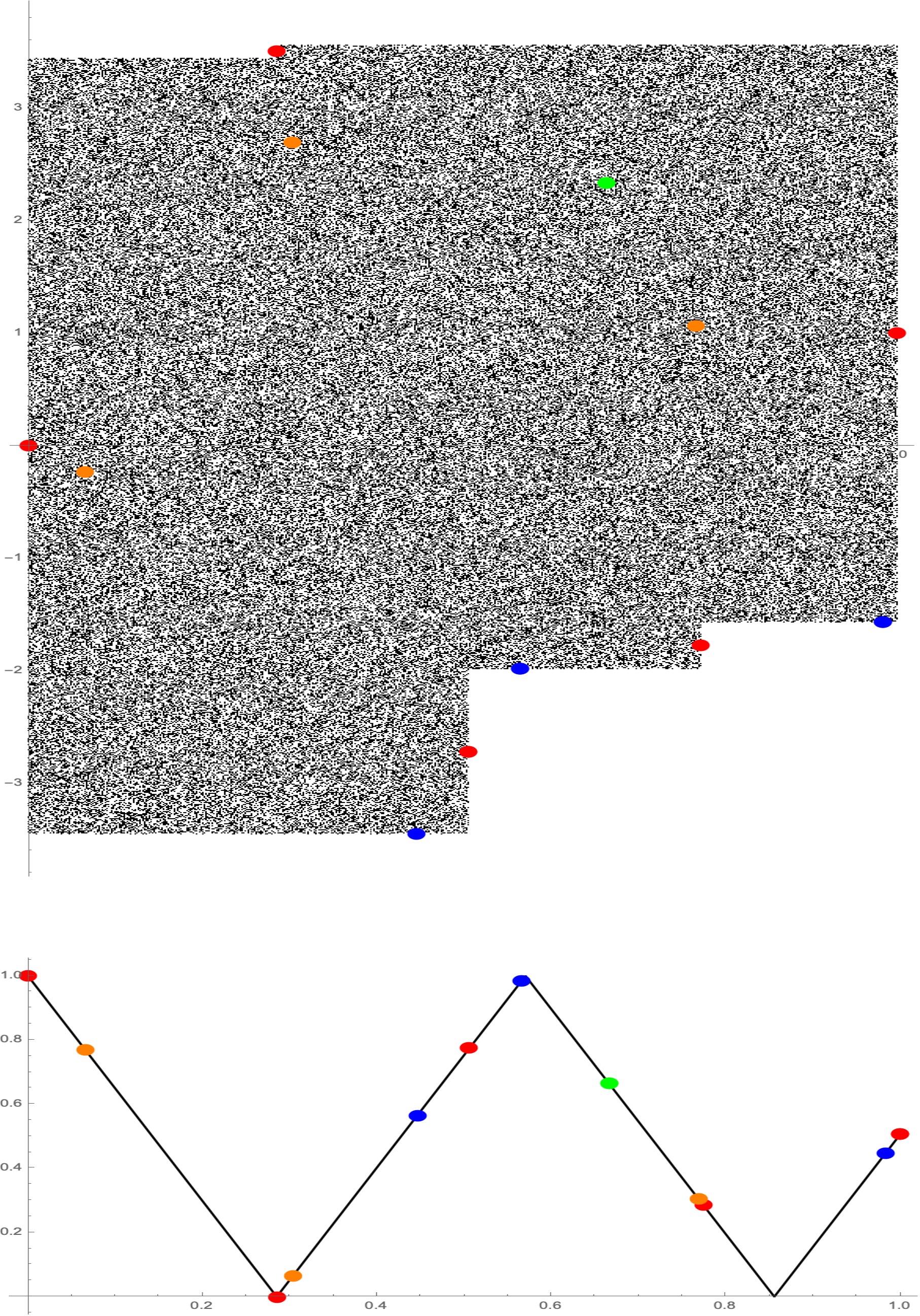}
\caption{(Below) The graph of the negative $\lambda$-zig-zag $f$, for $\lambda$ the Perron root of $D_f(t)=t^4-3t^3-t^2-3t+1$. (Above) The limit set $\Lambda_f$ of the Galois lift of $f$. In each picture, heavily marked points of the same color belong to the same periodic orbit, and the points of one color in $\Lambda_f$ are the unique periodic lifts of the points of $I$ of the same color. This limit set was drawn by plotting the orbit under $f_G$ of a single point with transcendental coordinates, to ensure that it is not eventually periodic. The same method was used to draw all limit sets in the paper.}\label{fig:periodicpoints}
\end{figure}

\begin{proof}
Observe first that $\tilde{g}_{p-1} \circ \cdots \circ \tilde{g}_0: \bbR \to \bbR$ is a contraction by a factor of $\lambda^{-k}<1$. This map has a unique fixed point $y_0$. By definition, 

\[
f_G(\alpha_0, y_0)=(g_0(\alpha_0), \tilde{g}_0(y_0))
\]

\noindent Indeed, inductively defining $y_i=\widetilde{g}_{i-1}(y_{i-1})$ we have $f_G^i(\alpha_0, y_0)=(\alpha_i, y_i)$ for $0 \leq i \leq p-1$. Since $\alpha_0$ and $y_0$ are periodic of period $p$, it now follows that $f_G^p(\alpha_0, y_0)=(\alpha_0, y_0)$.

Suppose that $(\alpha_0, z)$ is a periodic point of $f_G$ of least period $p$. Then

\[
f_G^i(\alpha_0, z)=(\alpha_i, \widetilde{g}_i \circ \cdots \circ \widetilde{g}_0(z)) \hspace{5mm} \text{for $0 \leq i \leq p-1$}
\]

\noindent In particular, $z$ is a fixed point of $\widetilde{g}_{k-1} \circ \cdots \circ \widetilde{g}_0$, hence is equal to $y_0$.

The fact that $(1,1)$ is the periodic point of $f_G$ projecting to $1 \in I$ follows immediately from the assumption that $D_f(\lambda^{-1})=0$.
\end{proof}

\begin{defn}\label{defn:align}
Let $R=[a,b] \times [c,d]$. The \textit{horizontal boundary of $R$} is $\partial_HR=[a,b] \times \{c,d\}$. The \textit{vertical boundary of $R$} is $\partial_VR=\{a,b\} \times [c,d]$.
Two rectangles $R_1=[a_1, b_1] \times [c_1,d_1]$, $R_2=[a_2,b_2] \times [c_2,d_2]$ are \textit{lower- (upper-)aligned} if $c_1=c_2$ ($d_1=d_2$).
\end{defn}

\begin{defn}
Let $f: I \to I$ be a PCP zig-zag whose Galois lift $f_G$ has rectangular limit set $\Lambda_f$. Let $R_0, \ldots, R_k$ be the rectangles defined by the canonical Markov partition of $f$ which subdivide $\Lambda_f$. The \textit{vertical boundary of $\Lambda_f$} is the set

\[
\partial_V\Lambda_f=\left  ( \partial\Lambda_f \cap \bigsqcup_{i=0}^k \partial_V R_i \right ) \setminus A,
\]

\noindent where $A$ is the set of isolated points of $\partial \Lambda_f \cap \bigsqcup_i \partial_V R_i$. The \textit{horizontal boundary of $\Lambda_f$} is the set

\[
\partial_H\Lambda_f = \bigsqcup_{i=0}^k \partial_HR_i
\]

\noindent A \textit{vertical (horizontal) component} of $\partial\Lambda_f$ is a connected component of $\partial_V\Lambda_f$ ($\partial_H\Lambda_f$).
\end{defn}

\begin{lem}\label{lem:vcomp}
Vertical components of $\partial\Lambda_f$ project to postcritical points of $f$.
\end{lem}

\begin{proof}
If $x=a \in \WPC(f) \setminus \PC(f)$ then $a$ is a critical point of $f$ that is not in the forward image of any $b \in \WPC(f)$. Therefore any rectangle $R$ projecting to an element of the weak postcritical Markov partition of $f$ such that $f_G(R)$ intersects the line $x=a$ must in fact map across this line. In other words, if $R_i$ and $R_{i+1}$ are the rectangles bordering the line $x=a$, then for any rectangle $R$ the image $f_G(R)$ crosses $R_i$ if and only if it also crosses $R_{i+1}$. Consequently $R_i$ and $R_{i+1}$ are both upper- and lower-aligned, and hence there is no vertical component of $\partial\Lambda_f$ projecting to $a \in I$. 
\end{proof}

\begin{figure}
\centering
\includegraphics[scale=.2]{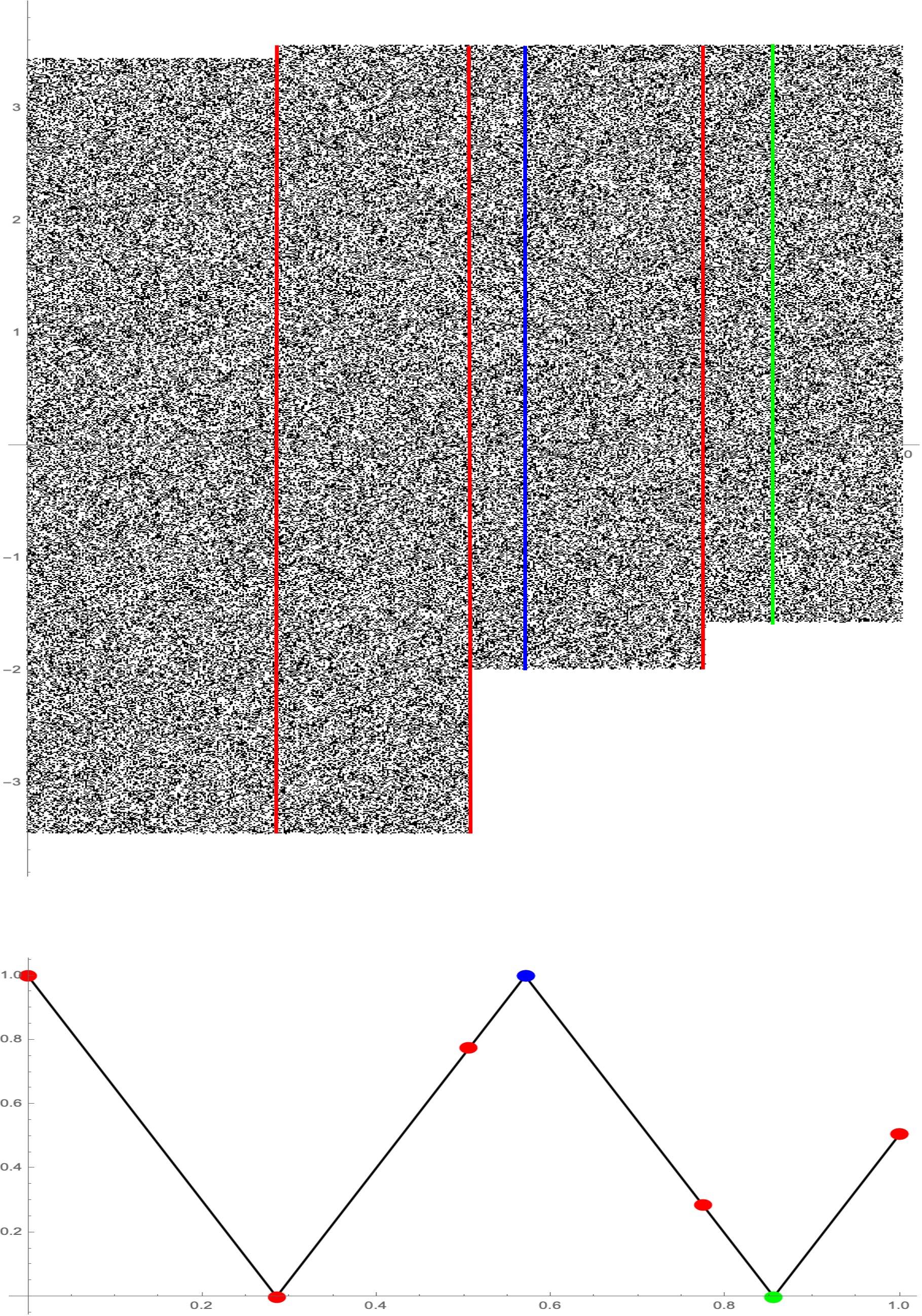}
\caption{The weak postcritical set of $f$ pulls back to a rectangular decomposition of $\Lambda_f$. Note that the vertical components of $\partial \Lambda_f$ project to the points of $\PC(f)$, in red. Observe that each rectangle is either upper- or lower-aligned with its neighbors (cf. Lemma \ref{lem:align}). Moreover, each vertical edge in this case contains at its metric center the periodic lift of the postcritical point to which it projects (cf. Figure \ref{fig:periodicpoints} and Lemma \ref{lem:center}).}\label{fig:vertedges}
\end{figure}

\begin{lem}
If $a \in \PC(f)$ then the unique periodic point of $f_G$ projecting to $a$ is contained in $\partial_V\Lambda_f$. 
\end{lem}

\begin{proof}
If $a=1$ then the conclusion holds. Since the point $x=1$ is periodic under $f$, its forward orbit contains a critical point $a \in \PC(f)$ such that $f(a)=1$. Since $\lambda^{-1}$ is a root of $D_f(t)$ the periodic point of $f_G$ projecting to $1 \in I$ is $(1,1) \in I \times \bbR$.

Moreover, if $f_i$, $f_{i+1}$ are the affine maps defining $f$ on either side of $a \in I$ then in fact $f_i(x)=a_i+\lambda x$ and $f_{i+1}(x)=\lambda x+2-a_i$ for some integer $a_i$. Therefore the affine maps defining $f_G$ on either side of the line $x=a$ are

\[
\tilde{f}_i(x,y)=(a_i+\lambda x, a_i-\lambda^{-1}y), \hspace{5mm} \tilde{f}_{i+1}(x,y)=(\lambda x+2-a_i, \lambda^{-1}y+2-a_i)
\]

\noindent If $(a,y) \in \Lambda_f$ then the images $\tilde{f}_i(a,y)=(1,a_i-\lambda^{-1}y)$ and $\tilde{f}_{i+1}(1,\lambda^{-1}y+2-a_i)$ are on the line $x=1$ and are symmetric about the point $(1,1)$. In particular, the maps $\tilde{f}_i$ and $\tilde{f}_{i+1}$ agree precisely at the periodic point $\tilde{a}$ of $f_G$ projecting to $a$. It follows that if $\tilde{a} \in \inte(R_i \cup R_{i+1})$ and $U \subseteq \inte (R_i \cup R_{i+1})$ is an open rectangle symmetric about $\tilde{a}$ then

\[
f_G(U \cap \inte(R_i)) = f_G(U \cap \inte(R_{i+1}))
\]

\noindent This contradicts the fact that $f_G$ is measure-preserving, so in fact $\tilde{a} \not \in \inte(R_i \cup R_{i+1})$. In other words, $\tilde{a} \in \partial_V\Lambda_f$.

We now proceed inductively, going through the periodic orbit of $(1,1)$ in reverse order. Suppose that $\tilde{a} \in I \times \bbR$ is a point in the orbit of $(1,1)$ that is contained in $\partial_V\Lambda_f$, and let $\tilde{b}$ be the periodic point such that $f_G(\tilde{b})=\tilde{a}$. If $\pi(\tilde{b})=b$ is a critical point of $f$, then we repeat the above argument. Otherwise $f_G$ is defined by a single affine map in a neighborhood of $\tilde{b}$. In particular, if $R_j$ and $R_{j+1}$ are the rectangles bordering the line $x=b$, if $\tilde{b} \in \inte(R_j \cup R_{j+1})$ then $f_G$ maps an open rectangular neighborhood of $\tilde{b}$ to an open rectangular neighborhood of $f_G(\tilde{b})=\tilde{a}$ in $\bbR^2$. This neighborhood cannot be a subset of $\Lambda_f$, since $\tilde{a} \in \partial_V \Lambda_f$ by assumption, but this contradictions the invariance of $\Lambda_f$.

Therefore $\tilde{b} \not \in \inte(R_j \cup R_{j+1})$, hence $\tilde{b} \in \partial_V\Lambda_f$. 
\end{proof}

\begin{lem}\label{lem:align}
There is only one vertical component of $\partial\Lambda_f$ projecting to a point $a \in \PC(f)$. Therefore, all adjacent rectangles $R_i, R_{i+1}$ of $\Lambda_f$ are either upper- or lower-aligned.
\end{lem}

\begin{proof}
If $a=0$ or $a=1$ then the claim follows immediately. Suppose otherwise. Since $f$ is PCP, there is a unique $b \in \WPC(f)$ such that $f(b)=a$. For each $i$ the connected components of $\partial_VR_i$ project to elements of $\WPC(f)$, so if $f_G$ maps such a component into the line $x=a$ then that component must project to $b$, since $\pi \circ f_G = f \circ \pi$. In other words, the preimage of any vertical component of $\partial\Lambda_f$ is contained in the intersection of $\Lambda_f$ with the line $x=b$.

Since $a \neq 0, 1$, $b$ is not a critical point of $f$, and therefore if $R_i$, $R_{i+1}$ are the rectangles of $\Lambda_f$ intersecting the line $x=b$ then $f_G$ acts on $R_i$ and $R_{i+1}$ by the same affine map $\tilde{f}$. The fact that $f_G(\partial_V R_j)$ does not intersect the line $x=b$ for any $j \neq i, i+1$ now implies that the number of vertical components of $\partial \Lambda_f$ projecting to $a \in \PC(f)$ is equal to the number of vertical components of $\partial \Lambda_f$ projecting to $b \in \PC(f)$, and $f_G$ maps the latter homeomorphically onto the former.

We proceed inductively backwards through the periodic orbit of $a \in \PC(f)$. Eventually we will arrive at $c=f(1) \in \PC(f)$. Since there is only a single vertical component of $\partial \Lambda_f$ projecting to $1 \in \PC(f)$, the above argument shows that there is similarly a single vertical component projecting to $c$. Moreover, the same statement now follows for each point in the periodic orbit of $a$, including $a$ itself.

That all adjacent rectangles are either upper- or lower-aligned now follows immediately: if two such rectangles are neither upper- nor lower-aligned, then there would exist at least two components of $\partial_V\Lambda_f$ along their intersection.
\end{proof}

\begin{lem}\label{lem:center}
Let $a \in \PC(f)$ and $\tilde{a} \in \Lambda_f$ the unique periodic point projecting to it. Then $\tilde{a}$ is at the metric center of the vertical component of $\partial\Lambda_f$ containing it.
\end{lem}

\begin{proof}
We begin by proving the statement for $a=1$. By Lemma \ref{periodic} we know $\tilde{a}=(1,1)$ in this case. If $f(c)=1$ then $c$ is a critical point that either is or is not periodic.

If $c$ is not periodic then $c \not \in \PC(f)$ and the two rectangles of $\Lambda_f$ intersecting the line $x=c$ are both upper- and lower-aligned, by Lemma \ref{lem:vcomp}. Denote these two rectangles by $R_i, R_{i+1}$ and let $\tilde{f}_i$, $\tilde{f}_{i+1}$ denote the affine maps by which $f_G$ acts on these rectangles, respectively. Then $\tilde{f}_i(R_i \cap R_{i+1})$ and $\tilde{f}_{i+1}(R_i \cap R_{i+1})$ are vertical line segments of equal length lying on the line $x=1$. Furthermore, these line segments are contained in $\Lambda_f$ and symmetric about $\tilde{a}$.

If $c'$ is another non-postcritical point satisfying $f(c')=1$ and $R_j, R_{j+1}$ the rectangles intersecting the line $x=c'$, then the images $\tilde{f}_j(R_j \cap R_{j+1})$ and $\tilde{f}_{j+1}(R_j \cap R_{j+1})$ do not intersect $\tilde{f}_i(R_i \cap R_{i+1})$ and $\tilde{f}_{i+1}(R_i \cap R_{i+1})$ except perhaps at their endpoints: otherwise two rectangles, say $R_i$ and $R_j$, would have

\[
f_G(\inte R_i) \cap f_G(\inte R_j) \neq \emptyset,
\]

\noindent contradicting the fact that $f_G$ is measure-preserving.

If $c$ is the unique periodic critical point of $f$, then there is a single vertical component of $\partial \Lambda_f$ projecting to $c$, and this component contains the periodic lift $\tilde{c}$ of $c$. Denote by $R_l, R_{l+1}$ the rectangles of $\Lambda_f$ intersecting at $x=c$. As before, $\tilde{f}_l(R_l \cap R_{l+1})$ and $\tilde{f}_{l+1}(R_l \cap R_{l+1})$ are vertical line segments symmetric about $\tilde{a}$. Furthermore, $f_G$ maps the single vertical component $V$ to a line segment containing $\tilde{a}$. The union $V \cup (R_l \cap R_{l+1})$ is a connected line segment, hence $\tilde{f}_l(V \cup (R_l \cap R_{l+1}))$ and $\tilde{f}_{l+1}(V \cup (R_l \cap R_{l+1}))$ are connected line segments such that

\[
\tilde{f}_l \left (V \cup (R_l \cap R_{l+1}) \right ) \cap \tilde{f}_{l+1}\left ( V \cup (R_l \cap R_{l+1})\right ) = f_G(V)
\]

\noindent It now follows that $f_G(V)$ is symmetric about $\tilde{a}$, and consequently so is the entire component of $\partial_V\Lambda_f$ containing $\tilde{a}$.

We now prove the claim for all points in the forward orbit of $1 \in \PC(f)$, proceeding inductively. As we argued in the proof of Lemma \ref{lem:align}, $f_G$ maps the component of $\partial_V\Lambda_f$ projecting to $a \in \PC(f)$ homeomorphically onto the component projecting to $f(a)$. In particular, if the unique periodic lift $\tilde{a}$ of $a$ lies at the metric center of the vertical component containing it, then so does $f_G(\tilde{a})$, since $f_G$ contracts the vertical direction uniformly by $\lambda^{-1}$.

If $f$ is a negative zig-zag, then it has a single postcritical orbit and so the proof is complete. If, however, $f$ is a positive zig-zag then we must still prove the claim for the fixed point $0 \in \PC(f)$. The exact argument used for $a=1$ applies here, after observing that the unique periodic lift of $a=0$ is the point $\tilde{a}=(0,0)$. 
\end{proof}

\begin{lem}\label{reverse}
The action of $f_G$ on $\Lambda_f$ can be recovered via the de Carvalho-Hall construction in the sense of Theorem \ref{Galois}. In particular, $f$ is of pseudo-Anosov type.
\end{lem}

\begin{figure}
\centering
\includegraphics[scale=.11]{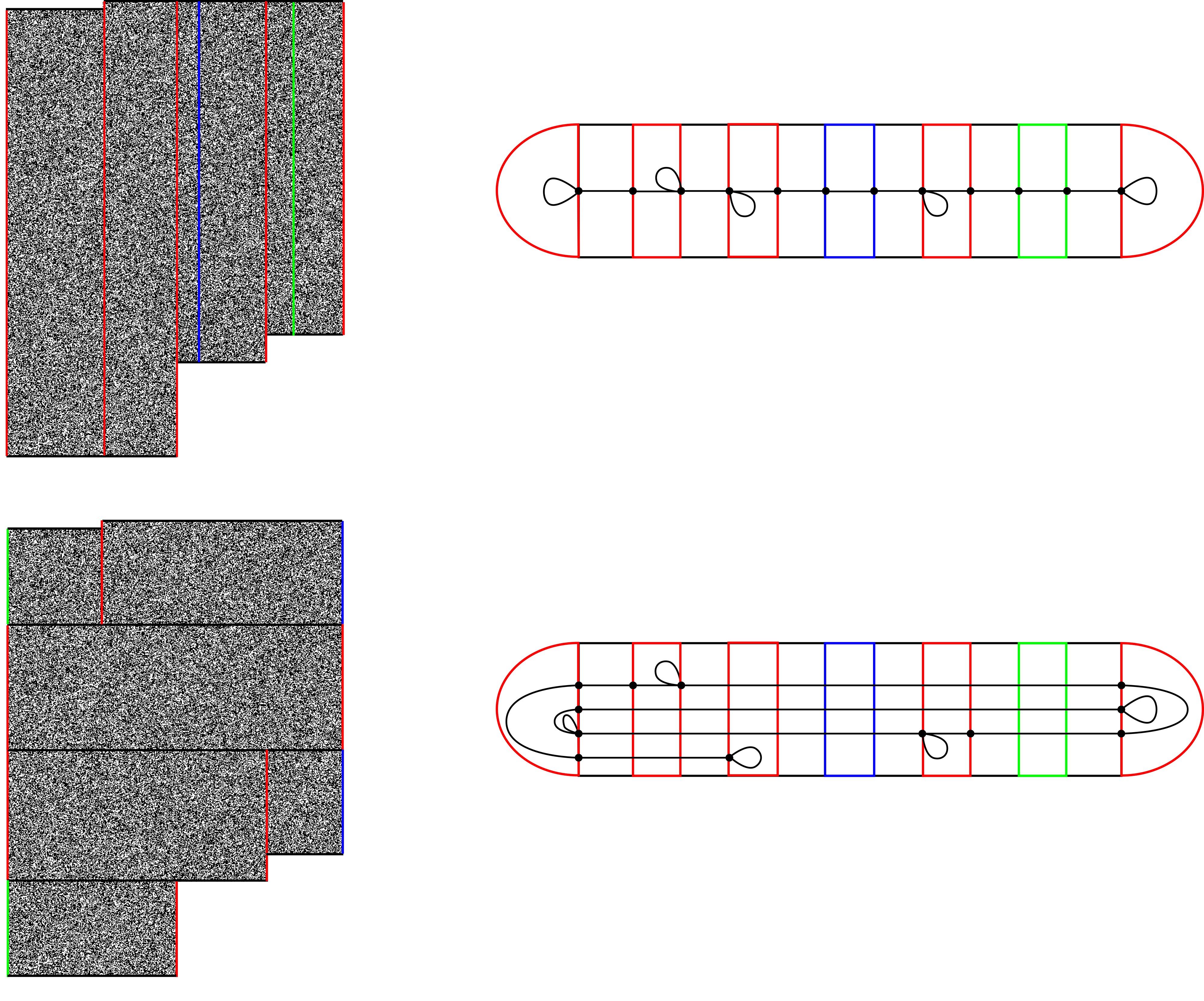}
\caption{Reverse-engineering the invariant train track $\tau$ for the exterior left-veering thickening of a zig-zag $f$ with rectangular limit set $\Lambda_f$. On the top left is the rectangular decomposition of $\Lambda_f$ projecting to $\WPC(f)$, and on the top right is a train track with junctions corresponding to the vertical boundaries of the rectangles. From left to right, each junction corresponds to a vertical boundary of the same color. Moreover, a junction contains a loop if and only if it corresponds to a line projecting to an element of $\PC(f)$. In this case, the loop is above (resp. below) the spine of $\tau$ if the rectangles adjoining the line are lower- (resp. upper-)aligned, and is on the left (resp. right) of the junction if the rectangle intersecting $\partial_V\Lambda_f$ is on the left (resp. right) of the line.}\label{fig:revtrack}
\end{figure}

\begin{proof}
We reverse engineer the invariant train track $\tau$ for the exterior left-veering thick interval map $F$ projecting to $f$. Construct $\tau$ to be the train track with a real edge $e_j$ for each rectangle $R_j$, a connecting infinitesimal edge between the real edges $e_j$, $e_{j+1}$ for adjacent rectangles $R_j$, $R_{j+1}$, and an infinitesimal loop for each $e_j$ corresponding to a $p_i$. The loop is on the left (resp. right) of $e_j$ if $p_i$ is on the left (resp. right) vertical boundary of $R_j$, and if the loop is on the same side of $e_j$ as a connecting infinitesimal edge then the loop is above (resp. below) the connecting edge if the vertical boundary component of $L$ that contains $p_i$ is above (resp. below) the adjacent rectangle.

The map induced by $f_G$ on $\tau$ leaves $\tau$ invariant under the pseudo-isotopies defined in section 2, hence $\tau$ is the invariant train track for a thick interval map projecting to $f$. By Proposition \ref{veer} this thick interval map must be either $F_L$ or $F_R$. One easily checks that the Galois conjugate coordinates imply that it must be $F_L$.

To complete the proof, observe that by Lemma \ref{measure} we may choose a right $\lambda$-eigenvector of the transition matrix $M$ of $f$ such that the rectangles obtained from $\tau$ have the exact dimensions of those in the rectangle decomposition of $L$. Therefore we may define an isometry $i: S' \to \Lambda_f$ taking the horizontal (resp. vertical) foliation of $S'$ to that of $\Lambda_f$. By construction, the map induced on $\Lambda_f$ is $f_G$.
\end{proof}


\subsection{The forward direction}


Lemma \ref{reverse} completes the proof of the reverse direction of Theorem \ref{Galois}. Now it remains to prove the forward direction.\\

We consider the case of $f: I \to I$ a positive $\lambda$-zig-zag of pseudo-Anosov type, the case for $f$ negative being essentially identical. Let $F$ be the exterior left-veering thickening of $f$. Let $S'$ be as in the statement of the theorem. $F$ induces a map $\tilde{f}: S' \to S'$. Let $p \in S'$ be the unique fixed point of $\tilde{f}$ projecting to $x=0 \in I$, and let $q \in S'$ be the unique periodic point projecting to $x=1 \in I$.

We wish to define a map $i: S' \to I \times \bbR$ such that $i \circ \tilde{f}=f_G \circ i$ and such that $i$ sends the horizontal and vertical foliations of $S'$ to those of $I \times \bbR$ in an orientation-preserving way. Fix $s, t \in \bbR$ distinct and define $i(p)=(0,s)$ and $i(q)=(1,t)$. Observe that these two choices determine $i$, and moreover the number $d=t-s$ controls the area of $i(S')$: if $A_d$ is the Lebesgue measure of $i(S')$ for any choice of $i$ satisfying $t-s=d$, then $A_d=|d| \cdot A_1$.

By construction, there is a piecewise affine $G: I \times \bbR \to I \times \bbR$ satisfying $i \circ \tilde{f} = G \circ i$. In particular, if $f_0, \ldots, f_k$ denote the linear maps defining the original zig-zag $f$ then $G$ is of the form

\[
G(x,y)=(f_i(x), g_i(y)) \hspace{2mm} \text{if $f(x)=f_i(x)$}
\]

\noindent where $g_i(y)=a_i(s,t)+(-1)^i \cdot \lambda^{-1}y$. This follows from the fact that $\tilde{f}$ is a piecewise affine orientation-preserving map projecting to $f$. The remainder of the proof is an analysis of the constants $a_i(s,t)$.

Note that $i(p)=(0,s)$ is a fixed point for $G$, and hence $y=s$ is fixed by $g_0(y)=a_0(s,t)+\lambda^{-1}y$. We thus have $a_0(s,t)=(1-\lambda^{-1})s$. To compute the remaining $a_i$ we require the following lemma.

\begin{lem}
Let $c_i=i \cdot \lambda^{-1}$ be a critical point of $f$ and set $g_{i-1}$, $g_i$ to be the defining $y$-coordinate maps of $G$ on either side of the line $x=c_i$. Then $g_{i-1}(y_i)=g_i(y_i)$ for a unique number $y_i$, and moreover this number satisfies $G(c_i, y_i)=s$ or $G(c_i, y_i)=t$ depending on whether $i$ is even or odd, respectively.
\end{lem}

\begin{proof}
Since $g_{i-1}$ and $g_i$ are affine with slopes of opposite sign they agree at a unique point $y_i$. To understand this number $y_i$, consider the corresponding situation for $\tilde{f}: S' \to S'$.

If $c_i$ is not a postcritical point of $f$, then the line $x=c_i$ partitions a rectangle of $S'$. Depending on the parity of $i$, the affine maps defining $\tilde{f'}$ on either side of the vertical line send the line to either $x=0$ or $x=1$, with opposite orientations. The fact that the unique periodic lifts of $x=0$ and $x=1$ are points of cone angle $\pi$ lying in the center of the corresponding vertical sides of $S'$ implies that these two affine maps would send the same point of $c_i \times \bbR$ to the singularity if their domains were extended to include this point.

If $c_i$ is a postcritical point, then the unique periodic lift of $c_i$ lies on the vertical boundary of exactly one of the corresponding rectangles of $S'$, and moreover it lies in the center of this boundary component. Since the image of this singularity is the sole singularity of $\tilde{f}$ on the corresponding vertical leaf of $S'$, the two affine maps defining $\tilde{f}$ on either side of $x=c_i$ must map each point of this line to points on $x=0$ or $x=1$ that are equal distances from the singularity. Because all vertical distances are scaled by the same factor, it now follows that both maps agree on the singularity projecting to $c_i$.

In either case, the same is true of the maps $g_{i-1}$ and $g_i$ after mapping into $I \times \bbR$.
\end{proof}

We return to computing the $a_i(s,t)$, thereby completing the proof of Theorem \ref{Galois}.

\begin{proof}[Proof of Theorem \ref{Galois}]
By the lemma, $g_1(y)=2t-g_0(y)$, and in particular

\[
a_1(s,t)=2t-a_0(s,t)=2t-(1-\lambda^{-1})s
\]

\noindent Similarly, we have $g_2(y)=2s-g_1(y)$, hence $a_2(s,t)=2s-a_1(s,t)$, and in general we have

\[
a_i(s,t)=\begin{cases}
(i+1-\lambda^{-1})s-it & \text{if $i$ is even}\\
(\lambda^{-1}-i)s + (i+1)t & \text{if $i$ is odd}
\end{cases}
\]

\noindent Setting $s=0$ and $t=1$ (and hence $d=1$), we specialize to the case

\[
a_i(s,t)=\begin{cases}
-i & \text{if $i$ is even}\\
i+1 & \text{if $i$ is odd}
\end{cases}
\]

\noindent In other words, each affine piece of $G: I \times \bbR \to I \times \bbR$ is of the form

\[
G_i(x,y)=\begin{cases}
(\lambda x-i, \lambda^{-1}y-i) & \text{if $i$ is even}\\
(i+1-\lambda x, i+1-\lambda^{-1}y) & \text{if $i$ is odd}
\end{cases}
\]

Thus $G$ is precisely the Galois lift $f_G$ of Thurston. Condition (2) of the theorem is immediately verified, so it remains to argue that the digit polynomial $D_f$ has $\lambda^{-1}$ as a root. We can define a ``vertical" digit polynomial using the orbit of $y=1$ under the $g_i$. Observe that this new polynomial is precisely $D_f$, and is necessarily satisfied by $\lambda^{-1}$. This completes the proof of Theorem \ref{Galois}.\\
\end{proof}

\begin{rmk}
One quickly verifies that

\[
-\frac{k}{1-\lambda^{-1}} a_0(s,t)+\sum_{i=1}^k a_i(s,t)=0
\]

\noindent and hence the map $(s,t) \mapsto (a_0(s,t), \ldots, a_k(s,t))$ is a linear map of $\bbR^2 \setminus \Delta$ into $\bbR^{k+1}$ whose image is a plane minus the line corresponding to $d=0$. This missing line is the image of the diagonal $\Delta \subseteq \bbR^2$, and we can foliate the image plane with lines parallel to this one, each such line corresponding to a different value of $d=t-s$. The area of the limit set $i_d(S')$ of $G_d$ scales linearly with $|d|$, i.e. if $A_d$ is the area of $i_d(S')$ then $A_d=|d| \cdot A_1$. In this way we can interpret the prohibited case $d=0$ as a degenerate limit set of area zero. Moreover, the two half planes defined by $d>0$ and $d<0$ correspond to the underlying train track map $\tau$ being left- or right-veering, respectively.\\
\end{rmk}


\section{The postcritical orbit of a zig-zag of pseudo-Anosov type}


In light of Theorem \ref{Galois}, it is natural to ask when an exterior left-veering thickening $F_L$ has a finite generalized invariant train track. This section obtains necessary conditions on the structure of the thick interval map $F_L$ associated to a PCP zig-zag of pseudo-Anosov type. In particular, Proposition \ref{type} will be instrumental to our proof of Theorem \ref{bijection} in Section 6.\\

Recall that if $F_L: (S^2, \bbI) \to (S^2, \bbI)$ is a thickening of a PCF interval map $f$, then we denote by \textbf{0} the junction projecting to $0 \in I$, and similarly we denote by \textbf{1} the vertex projecting to $1 \in I$.

\begin{defn}
Let $f: I \to I$ be a $\lambda$-zig-zag map and $c_i=i \cdot \lambda^{-1}$ the critical points of $f$, for $i=1, \ldots, \floor{\lambda}$. Let $F_L: (S^2, \bbI) \to (S^2, \bbI)$ be the exterior left-veering thickening of $f$. For each $i$ we denote by $C_i$ the junction projecting to $c_i$.
\end{defn}

\begin{defn}
Let $\tau$ be the generalized invariant train track for the left-veering thickening $F_L$ of the PCP zig-zag map $f$. By a \textit{connecting infinitesimal edge} we will mean an infinitesimal edge of $\tau$ connecting two distinct real edges. The remaining infinitesimal edges of $\tau$, namely those that connect a real edge to itself, are called \textit{loops}.

We define the \textit{spine} $\tau'$ of $\tau$ to be the union of all real edges and connecting infinitesimal edges of $\tau$. We orient $\tau'$ from \textbf{0} to \textbf{1}. A loop $\gamma \subseteq \tau$ contained in an intermediate vertex of $\bbI$ is called \textit{interior} if $\gamma$ is to the left of $\tau'$ and \textit{exterior} otherwise. See Figure \ref{fig:spine_and_loops}. 
\end{defn}

\begin{rmk}
The interior loops $\gamma$ of $\tau$ are precisely the loops whose image $F(\gamma)$ points toward the horizontal midline of $\bbI$. One can see this by noting that $F$ is exterior left-veering and preserves orientation. It is for this reason that we refer to such loops as ``interior." In the arguments of this section we will see that interior loops are rare, and the existence of more than one for $\tau$ is an obstruction to the finiteness of $\tau$.
\end{rmk}

\begin{figure}
\centering
\includegraphics[scale=.3]{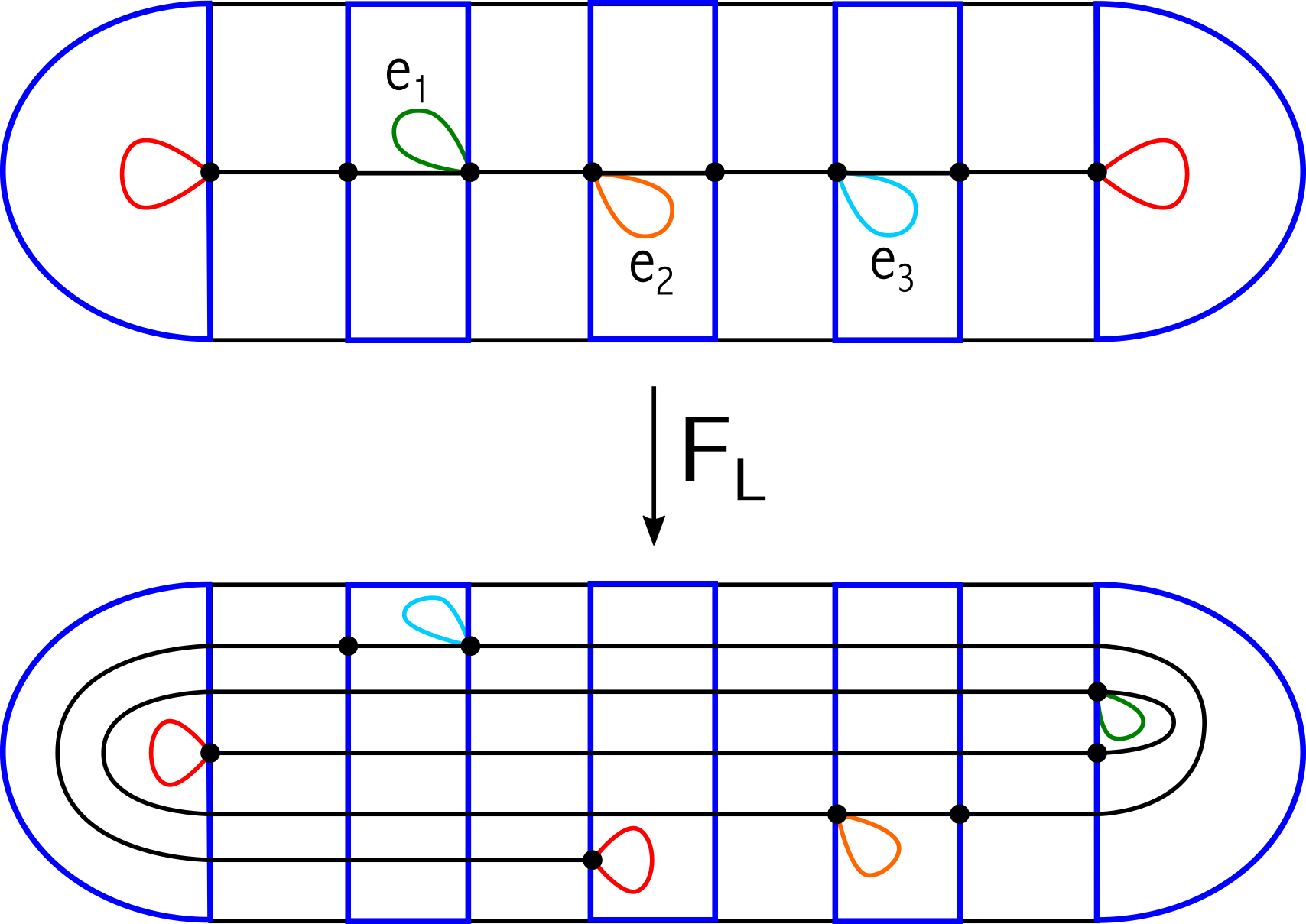}
\caption{A (finite) generalized train track $\tau$ and an exterior left-veering thick interval map preserving it. Here the spine of $\tau$ is the union of the black edges. Of the intermediate loops, $e_1$ is interior while $e_2$ and $e_3$ are exterior. Observe that under the action of $F_L$, the image of an interior loop points in toward the horizontal midline of $\bbI$, hence will be enclosed by arcs after pseudo-isotopy unless it maps into \textbf{0} or \textbf{1}. }\label{fig:spine_and_loops}
\end{figure}

Recall (cf. Definition \ref{defn:zig-zag}) that a zig-zag $f$ is \textit{positive} if $f(0)=0$, and \textit{negative} if $f(0)=1$.

\begin{prop}\label{intloop}
Let $f: I \to I$ be a PCP zig-zag map of growth rate $\lambda > 2$. If $f$ is of pseudo-Anosov type, then $c_1=\lambda^{-1}$ is in the forward orbit of $x=1$. Moreover, the loop in the vertex $C_1$ projecting to $c_1$ must be interior.
\end{prop}

\begin{proof}
Suppose first that $f$ is positive, so that $f(0)=0$. Since $f$ is PCP, some $c_i=i \cdot \lambda^{-1}$ is in the forward orbit of $x=1$. By Theorem \ref{PCP}, the corresponding junction $C_i$ contains an edge $\epsilon$ and a loop $\gamma$ of the invariant train track $\tau$. If $\gamma$ is exterior then after pseudo-isotopy the images $F(\epsilon)$ and $F(\gamma)$ will be non-parallel loops in \textbf{1}, contradicting Theorem \ref{PCP}. See Figure \ref{fig:sec5_1} below. Thus $\gamma$ must be interior.

\begin{figure}[h!]
\centering
\includegraphics[scale=.4]{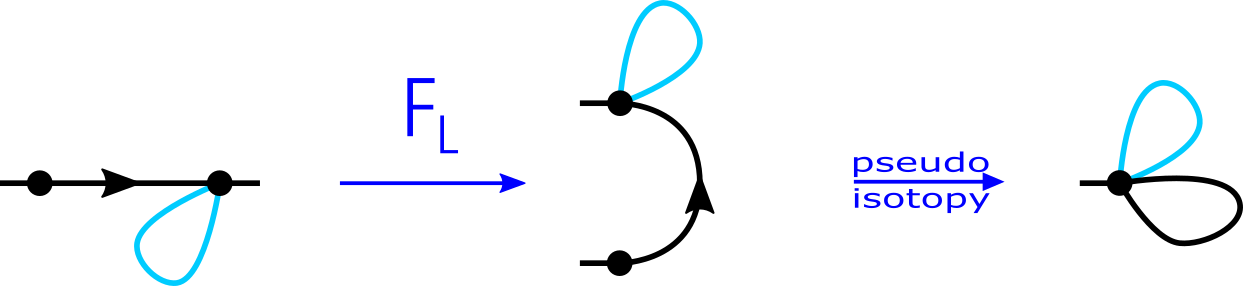}
\caption{An exterior loop in $C_i$ produces two non-parallel loops in \textbf{1}}\label{fig:sec5_1}
\end{figure}

If $i \neq 1$ then $F$ will map the connecting edge $\epsilon'$ of $\tau$ that is within $C_1$ into \textbf{1} further interior than the images of $\epsilon$ and $\gamma$, and thus after pseudo-isotopy we will obtain two non-parallel loops enclosed by a third, again contradicting Lemma \ref{PCP}. See Figure \ref{fig:sec5_2}. On the other hand, if $i=1$ then the images of $\epsilon$ and $\gamma$ will be the furthest interior in \textbf{1} and all other edges that map into \textbf{1} will pinch to parallel loops, which then combine under pseudo-isotopy.

\begin{figure}[h!]
\centering
\includegraphics[scale=.3]{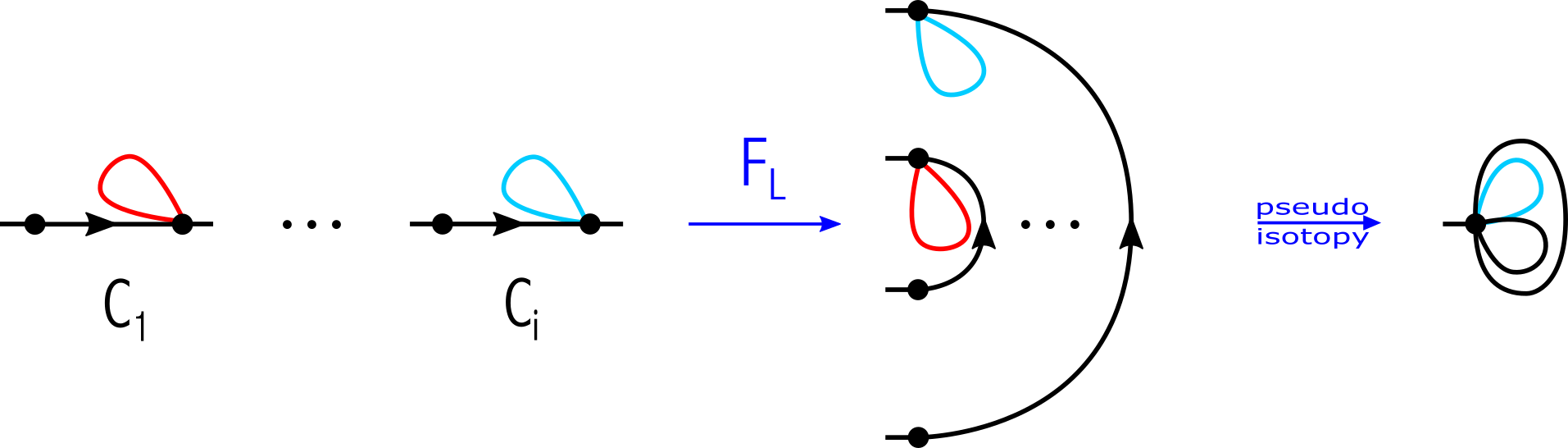}
\caption{The image of all edges in $C_1$ is further interior than that of any other $C_i$ mapping into \textbf{1}, so any loop in $C_i$ generates a loop in \textbf{1} that persists after pseudo-isotopy.}\label{fig:sec5_2}
\end{figure}

Now suppose that $f$ is negative, so that $f(0)=1$. The same argument applies, except that the role of \textbf{1} is taken by \textbf{0}, which maps into \textbf{1}. 
\end{proof}

\begin{prop}\label{singleint}
Let $f: I \to I$ be a PCP zig-zag map of growth rate $\lambda >2$. If $f$ is of pseudo-Anosov type, then the invariant train track $\tau=\tau_L$ has exactly one interior loop, namely the loop contained in the vertex $C_1$.
\end{prop}

\begin{proof}
Proposition \ref{intloop} tells us that $C_1$ contains an interior loop, so it remains to show that $\tau$ has no other interior loops. Furthermore, our argument in the proof of \ref{intloop} demonstrates that no other critical vertex $C_i$ can contain an interior loop, so if $\tau$ contains another interior loop $\gamma$ then $F(\gamma)$ cannot lie in \textbf{0} or in \textbf{1}. In other words, $F(\gamma)$ lies in some intermediate vertex $V$. However, since $\gamma$ is interior the image $F(\gamma)$ faces toward the center of $\bbI$. See Figure \ref{fig:sec5_3}. Thus after pseudo-isotopy $F(\gamma)$ will be enclosed by multiple connecting edges, contradicting the structure of $\tau$ described in Theorem \ref{PCP}.
\end{proof}

\begin{figure}[h!]
\centering
\includegraphics[scale=.3]{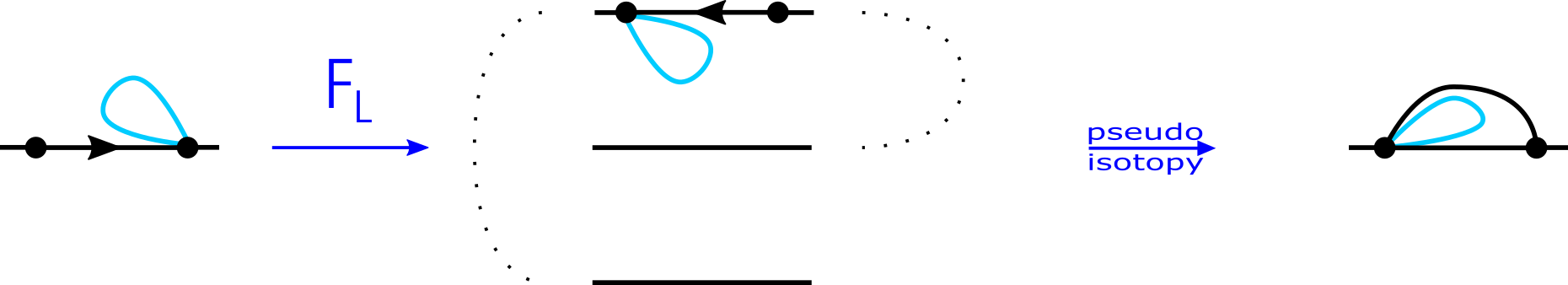}
\caption{The image of an interior loop points inward, and because all horizontal layers of $F(\tau)$ span the full length of $\bbI$ except for the last, the image of an interior loop will be trapped by other edges after pseudo-isotopy unless it maps into \textbf{0} or \textbf{1}, where it can be absorbed into the parallel loop formed by the turn.}\label{fig:sec5_3}
\end{figure}

\begin{cor}\label{quad}
Let $f: I \to I$ be a PCP zig-zag map of growth rate $\lambda>2$, and suppose $f$ is of pseudo-Anosov type.

\begin{enumerate}
\item If $f$ is positive and $\floor{\lambda}=m$ is odd, then the orbit of $x=1$ under $f$ is $1 \mapsto \lambda^{-1} \mapsto 1$.
\item If $f$ is negative and $\floor{\lambda}=m$ is even, then the orbit of $x=1$ under $f$ is $1 \mapsto \lambda^{-1} \mapsto 0 \mapsto 1$.
\end{enumerate}

\noindent In either case, $\lambda$ has minimal polynomial $p(x)=x^2-(m+1)x+1$. 
\end{cor}

\begin{proof}
First suppose $f$ is positive and $m$ is odd. The image $F_L(\gamma)$ of the loop $\gamma \subseteq \tau$ contained in \textbf{1} is above the other strands of $F_L(\tau)$, and so after pseudo-isotopy it will be an interior loop $\gamma'$. This interior loop must necessarily be contained in $C_1$ and must map into \textbf{1} by Proposition \ref{singleint}. It follows that $x=1$ has period $2$. The linear branch of $f$ that applies to $x=1$ is $x \mapsto (m+1)-\lambda x$, hence we have $\lambda^{-1}=m+1-\lambda$.

The case when $f$ is negative and $m$ is even is similar. The image $F_L(\gamma)$ is again an interior loop after pseudo-isotopy, and so must be contained in $C_1$. Since the linear branch of $f$ that applies to $x=1$ is $x \mapsto (m+1)-\lambda x$, it follows that $\lambda^{-1}=m+1-\lambda$. 
\end{proof}

\begin{rmk}
The examples described in Corollary \ref{quad} are the simplest examples of uniform expanders of pseudo-Anosov type. The corresponding pseudo-Anosovs all live on the four-punctured sphere $\Sigma_{0,4}$ and lift via a branched double cover to hyperbolic automorphisms of the torus. To avoid these simple cases, we make the following definition (see Remark \ref{rmk:four} below).
\end{rmk}

\begin{defn}\label{defn:PA}
For $m \geq 2$ and $p \geq 3$, define $\PA(m,p)$ to be the set of zig-zags $f$ of pseudo-Anosov type such that

\begin{enumerate}
\item $f$ has $m$ critical points, and
\item $\#\PC(f) = p$.
\end{enumerate}

\noindent We also define the set

\[
\PA(m) = \bigcup_{p \geq 4} \PA(m,p).
\]
\end{defn}

\begin{rmk}\label{rmk:four}
An element $f \in \PA(m,p)$ defines a pseudo-Anosov on the $(p+1)$-punctured sphere, where the point at infinity is fixed. Ripping open this point to a boundary component, we may also think about this pseudo-Anosov as a braid on the $p$-punctured disc, up to multiplication by a full twist $\Delta^2$ around the boundary. The simple examples described in Corollary \ref{quad} are precisely the maps $f \in \PA(m,3)$, and the definition of $\PA(m)$ allows us to avoid these examples in the future.
\end{rmk}

Recall that we use the notation $c_i=i \cdot \lambda^{-1}$. For a $\lambda$-zig-zag map $f$ we have the subintervals of definition $I_j=[c_j, c_{j+1})$ for $0 \leq j \leq m-1$ and $I_m=[c_m, 1]$.

\begin{prop}\label{type}
Suppose $f \in \PA(m)$ with slope $\lambda$. Then $\PC(f)$ always contains the following 3 types of points:

\begin{enumerate}
\item[(E)] The extremal points $x=0$ and $x=1$
\item[(C)] The critical point $c_1=\lambda^{-1}$ 
\item[(R)] Exactly one point in the interior of $I_{m-1}$
\end{enumerate}

\noindent If $\PC(f)$ contains other points, then they are of the following 2 types:

\begin{enumerate}[resume]
\item[($P_{m-2}$)] Points in the interior of $I_{m-2}$
\item[($P_m$)] Points in the interior of $I_m$
\end{enumerate}
\end{prop}

\begin{proof}
The existence of type $E$ points is obvious, and the existence of type $C$ follows from Proposition \ref{intloop}. The single point of type $R$ is precisely the postcritical point that maps to $c_1$. Since the loop of $\tau$ in $C_1$ must be interior, it follows that the image of the loop of $\tau$ that maps into $C_1$ must be interior, hence lies over a postcritical point at which $f$ is orientation-reversing. Since there is only one interior loop of $\tau$, this is the only postcritical point at which $f$ reverses orientation.

To finish the proof, suppose $f$ has another critical point $x=p$ other than type $E$, $C$, and $R$. Then the loop of $\tau$ over this critical point must have an image that is exterior so that $F(\tau)$ does not trap the loop after tightening. Since $F$ is exterior veering, it follows that the image of this loop must come after the third-to-last turn of $F$, hence $p$ is in the interior of  $I_{m-2} \cup I_{m-1} \cup I_m$. Moreover, since $p$ is not of type $R$, $f$ preserves orientation at $p$, hence $p \not \in I_{m-1}$. 
\end{proof}

\begin{ex}
Consider the thick interval map in Figure \ref{fig:spine_and_loops}. This is the exterior left-veering thickening of the positive $\lambda$ zig-zag $f$, where $\lambda$ is the Perron root of $D_f(t)=t^4-4t^3-2t^2-4t+1$. The graph of $f$ is shown in Figure \ref{fig:sec5_4}.
\end{ex}

\begin{figure}[h!]
\centering
\includegraphics[scale=.7]{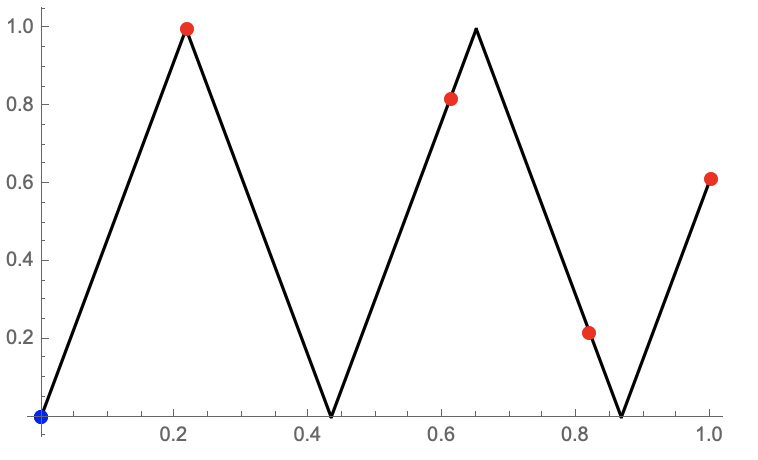}
\caption{The positive zig-zag for $\lambda$ the Perron root of $t^4-4t^3-2t^2-4t+1$. The map $f$ is of pseudo-Anosov type, as demonstrated by Figure \ref{fig:spine_and_loops}, and the orbit of $x=1$ is shown in red. In particular, $\PC(f)$ has points of type $E$, $C$, $R$, and $P_{m-2}$.}\label{fig:sec5_4}
\end{figure}

\begin{cor}
Suppose $f \in \PA(m)$ has slope $\lambda$. If $\PC(f)$ contains no points of type $P_{m-2}$ or $P_m$ then $\lambda$ is a quadratic integer.
\end{cor}

\begin{proof}
If $f$ is positive and has no postcritical points of type $P_i$ then the orbit of $x=1$ is 

\[
1 \mapsto \lambda-m \mapsto m+m \lambda -\lambda^2 \mapsto 2-m\lambda-m \lambda^2 +\lambda^3 = 1,
\]

\noindent hence $\lambda$ satisfies the relation

\[
0=\lambda^3-m\lambda^2-m\lambda+1=(\lambda+1) \left (\lambda^2-(m+1)\lambda+1 \right ).
\] 

\noindent Since $\lambda>1$ and $x^2-(m+1)x+1$ is irreducible over $\bbQ$, we see that $\lambda$ is a quadratic integer. Similarly, if $f$ is negative and has no postcritical points of type $P_i$ then from the orbit of $x=1$ we obtain the relation

\[
0=\lambda(\lambda+1)\left (\lambda^2-(m+1)\lambda+1 \right )
\]

\noindent Hence $\lambda$ is again a quadratic integer.
\end{proof}


\section{Classifying zig-zags of pseudo-Anosov type}


In this section we consider the set $\PA(m)$ of zig-zags of pseudo-Anosov type with modality $m \geq 2$. The case $m=1$ was considered by Hall in Theorem 2.1 of \cite{H}, in which he gave an explicit bijection between $\PA(1)$ and $\bbQ \cap (0, 1/2)$. We give an explicit bijection between $\PA(m)$ and $\bbQ \cap (0,1)$ for each $m \geq 2$.\\

The proof of Theorem \ref{bijection} naturally breaks into 3 cases: (1) $m \geq 4$ even, (2) $m \geq 3$ odd, and (3) $m=2$. In preparation for the proof, we investigate each of these cases separately. We treat case (1) first, since it is essential to understanding cases (2) and (3).

Throughout this section, $f$ is a PCP $\lambda$-zig-zag map with modality $m=\floor{\lambda}$. Moreover, we always assume that $\PC(f) \geq 4$, so that it is possible for $f$ to belong to $\PA(m)$ (cf. Definition \ref{defn:PA}).


\subsection{The case $m \geq 4$ even}


\begin{defn}\label{perm}
Given $n \geq 3$ and $2 \leq k \leq n-1$, define $\rho_e(n,k) \in S_n$ to be the permutation such that

\[
\rho_e(n,k)(i)=\begin{cases}
n & i=1 \\
i+(n-k) & 2 \leq i \leq k-1\\
i-(k-1) & k \leq i \leq n
\end{cases}
\]

\noindent Visually, $\rho_e(n,k)$ takes the form 

\[
\rho_e(n,k)=\begin{pmatrix}
1 & 2 & \cdots & k-1 & k & k+1 & \cdots & n-1 & n\\
n & n-k+2 & \cdots & n-1 & 1 & 2 & \cdots & n-k & n-k+1
\end{pmatrix}
\]
\end{defn}

\begin{ex}
Here are some examples of $\rho_e(n,k)$ for $n=7$.

\[
\rho_e(7,2)=\begin{pmatrix}
1 & 2 & 3 & 4 & 5 & 6 & 7\\
7 & 1 & 2 & 3 & 4 & 5 & 6
\end{pmatrix},
\hspace{5mm}
\rho_e(7,3)=\begin{pmatrix}
1 & 2 & 3 & 4 & 5 & 6 & 7\\
7 & 6 & 1 & 2 & 3 & 4 & 5
\end{pmatrix},
\]
\[
\rho_e(7,6)=\begin{pmatrix}
1 & 2 & 3 & 4 & 5 & 6 & 7\\
7 & 3 & 4 & 5 & 6 & 1 & 2
\end{pmatrix}
\]

\noindent Note that $\rho_e(7,2)$ and $\rho_e(7,6)$ are full $7$-cycles, whereas $\rho_e(7,3)=(7,5,3,1)(6,4,2)$ is not. As we will see, $\rho_e(n,k)$ is an $n$-cycle if and only if $n-k$ and $k-1$ are coprime (cf. Proposition \ref{cycle}).
\end{ex}

\begin{defn}\label{defn:layer}
Let $F$ be a thickening of a PCF map $f$, and let $\tau$ denote the invariant generalized train track of $F$. As usual, denote by \textbf{0} and \textbf{1} the junctions projecting to $0, 1 \in I$, respectively, and let $\tau'$ be the spine of $\tau$. We define a \textit{layer} of the image $F(\tau)$ to be a connected component of the complement of $\textbf{0} \cup \textbf{1}$ in $F(\tau')$. See Figure \ref{fig:layers}.
\end{defn}

\begin{figure}
\centering
\includegraphics[scale=.2]{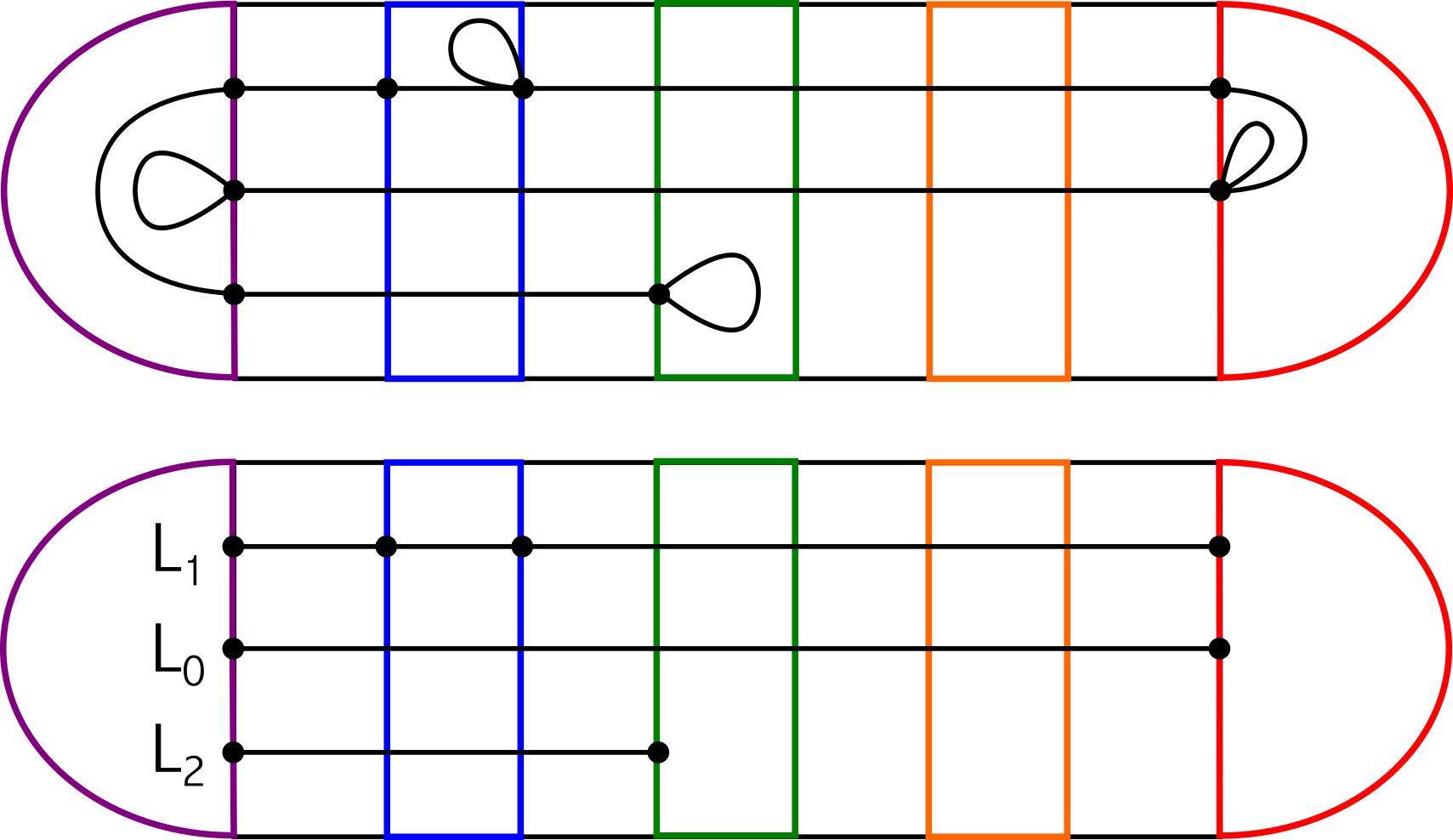}
\caption{The layers of $F(\tau)$ for the exterior left-veering thickening of a bimodal zig-zag.}\label{fig:layers}
\end{figure}

\begin{rmk}
If $f$ is an $m$-modal zig-zag and $F$ is any thickening of $f$, then $F(\tau)$ has $m+1$ layers. Indeed, if $L$ is a layer of $F(\tau)$, then $F^{-1}(L)$ projects to one of the $m+1$ different intervals of monotonicity $I_j$ for $f$. 
\end{rmk}

\begin{defn}
Let $f$ be an $m$-modal zig-zag and $F$ the exterior left-veering thickening of $f$. For each $0 \leq j \leq m$, the \textit{$j$-th layer} of $F(\tau)$ is the layer $L_j$ such that

\[
\pi(F^{-1}(L_j))=I_j
\]
\end{defn}

Recall (cf. Definition \ref{defn:ptype}) that the \textit{permutation type} of a PCP interval map $f$ is the permutation $\rho(f)$ such that $f(x_i)=x_{\rho(f)(i)}$. 

\begin{prop}\label{even}
Let $f: I \to I$ be a PCP $\lambda$-zig-zag map with $\floor{\lambda}=m \geq 4$ even. Then $f \in \PA(m)$ if and only if the following hold:

\begin{enumerate}
\item Under the ordering $x_1<x_2<\ldots<x_n=1$ of the orbit of $x=1$, we have $\rho(f)=\rho_e(n,k)$ for some $k$.
\item The $\{x_i\}$ satisfy the conclusion of Proposition \ref{type}, with $x_1$ being the sole point of type (C) and $x_k$ being the sole point of type (R).
\end{enumerate}
\end{prop}

\begin{proof}
We begin first by assuming that $f$ is of pseudo-Anosov type. Let $\tau$ be the invariant train track for the exterior left-veering thick interval map $F$. By Proposition \ref{intloop} the junction $C_1$ over $c_1=\lambda^{-1}$ contains the sole interior loop $\gamma$ of $\tau$, and this loop must map into \textbf{1}. Moreover, since $\floor{\lambda} \geq 4$, the layers $L_0$, $L_1$, and $L_2$ stretch fully between \textbf{0} and \textbf{1}, with $L_1$ above $L_3$ and below $L_2$. In particular, there can be no intermediate loops of $\tau$ before $\gamma$, since otherwise these would be trapped by $L_2$ or $L_3$. Thus if $\rho$ is the permutation describing the action of $f$ on the orbit of $x=1$, we must have $\rho(1)=n$.

Now set $\rho^{-1}(1)=k$. By Proposition \ref{singleint}, all loops in $\tau$ corresponding to $i \neq k$ must have exterior image, so their images must lie in $L_{m-2}$ or $L_m$. In particular, if $k=2$ then the remaining $n-k$ loops must be sent in order to the remaining $k-2$ junctions containing a loop. In particular, we have

\[
\rho=\begin{pmatrix}
1 & 2 & 3 & \cdots & n\\
n & 1 & 2 & \cdots & n-1
\end{pmatrix}
=\rho_e(n,2)
\]

If $k>2$, then the loops labelled $2$ through $k-1$ must map to the last $k-3$ loops not in \textbf{1}, in order. This is again because these loops must all have exterior images, and so in particular cannot be covered from below by loops on $L_m$. Now the loops labelled $k+1$ through $n$ must map to the remaining loops labelled $2$ through $n-k+1$ in order, and thus we see that $\rho=\rho_e(n,k)$.

Suppose instead that $f$ satisfies properties (1) and (2) above. Let $F$ be the exterior left-veering thickening of $f$. We construct the invariant generalized train track $\tau$ of $F$ as described in Section 2. We begin with $\tau_0$, the (disconnected) union of real edges, and after the  tightening pseudo-isotopies, $\tau_1=F_\ast(\tau_0)$ contains a loop $\gamma$ projecting to $x_n=1$. Consider the successive images $F_\ast^j(\gamma)$, which by conditions (1) and (2) are exterior until some minimal index $j=a$. Moreover, each of these initial forward images is guaranteed to be uncovered from below by layers of $F(\tau_{j-1})$, and so are the only infinitesimal edges within the corresponding junctions aside from a single connecting edge between the two adjacent real edges.

$F_\ast^a(\gamma)$ is an interior loop, and by assumption must project to $x_1=\lambda^{-1}$, since $x_k$ is the sole postcritical point of type (R) and $\rho_e(n,k)(k)=1$. In particular, $F_\ast^{a+1}(\gamma)$ is a loop over $x_n$ and since $x_1=\lambda^{-1}$ this loop is interior to every turn of $F_\ast(\tau_a)$. Hence after tightening we recover a single loop over $x_n$ and the resulting generalized train track $\tau_{a+1}=\tau$ is invariant. It follows that $\tau$ is finite and that $f$ is of pseudo-Anosov type.
\end{proof}

In order to completely characterize positive PCP zig-zags with modality $m \geq 4$ even, it remains to determine when $\rho_e(n,k)$ is an $n$-cycle.

\begin{prop}\label{cycle}
The permutation $\rho_e(n,k) \in S_n$ is an $n$-cycle if and only if $(n-k, k-1)=1$. 
\end{prop}

\begin{proof}
Set $\rho=\rho_e(n,k)$, and define $\rho' \in S_{n-1}$ to be the permutation on $\{2, \ldots, n\}$ defined by

\[
\rho'(i)=\begin{cases}
\rho(i) & 2 \leq i \neq k\\
n & i=k
\end{cases}
\]

\noindent Thus $\rho'$ is the element of $S_{n-1}$ obtained by deleting the symbol $1$ from the cycle decomposition of $\rho$. Observe that, since $1$ is  in the orbit of $n$, $\rho$ is an $n$-cycle if and only if $\rho'$ is an $(n-1)$-cycle. Shifting all labels down by $1$, we have

\[
\rho'(i)=\begin{cases}
i+(n-k) & 1 \leq i \leq k-1\\
i-(k-1) & k \leq i \leq n-1
\end{cases}
\]

\noindent Interpreting this modulo $n-1$, we see that $\rho'$ acts via addition by $n-k$: indeed, subtraction by $k-1 \pmod{n-1}$ is equivalent to addition by $n-1-(k-1) \equiv n-k \pmod{n-1}$. Thus this action is transitive if and only if $(n-k, n-1)=1$.
\end{proof}


\subsection{The case $m \geq 3$ odd}


We next consider the case when $f$ is a PCP zig-zag of modality $m \geq 3$ odd. In this case we observe that the orbit of $x=1$ ends with $c_1 \mapsto 0 \mapsto 1$.

\begin{defn}
Given $n \geq 3$ and $2 \leq k \leq n-1$, define $\rho_o(n,k) \in S_{n+1}$ to be the permutation such that

\[
\rho_o(n,k)(i) = \begin{cases}
n & i=0\\
0 & i=1\\
i+(n-k) & 2 \leq i \leq k-1\\
i-(k-1) & k \leq i \leq n
\end{cases}
\]

\noindent Alternatively, $\rho_o(n,k)$ is of the form 

\[
\rho_o(n,k)=\begin{pmatrix}
0 & 1 & 2 & \cdots & k-1 & k & k+1 & \cdots & n-1 & n\\
n & 0 & n-k+2 & \cdots & n-1 & 1 & 2 & \cdots & n-k & n-k+1
\end{pmatrix}
\]
\end{defn}

\begin{rmk}
Note that if we define $\rho' \in S_n$ to be the permutation obtained from $\rho=\rho_o(n,k)$ by deleting $0$ from the orbit of $n$, then $\rho'$ is an $n$-cycle if and only if $\rho$ is an $(n+1)$-cycle. Since $\rho'=\rho_e(n,k)$, we see that this is again the case if and only if $(n-k, n-1)=1$. We have therefore proven the following statement.
\end{rmk}

\begin{prop}\label{cycle3}
The permutation $\rho_o(n,k) \in S_{n+1}$ is an $(n+1)$-cycle if and only if $(n-k, k-1)=1$. 
\end{prop}

A nearly identical argument to that of Proposition \ref{even} proves the following analogue.

\begin{prop}\label{odd}
Let $f: I \to I$ be a PCP $\lambda$-zig-zag map with $\floor{\lambda}=m \geq 3$ odd. Then $f \in \PA(m)$ if and only if the following hold:

\begin{enumerate}
\item Under the ordering $0=x_0<x_1<\ldots<x_n=1$ of the orbit of $x=1$, we have $\rho(f)=\rho_o(n,k)$ for some $k$.
\item The $\{x_i\}$ satisfy the conclusion of Proposition \ref{type}, with $x_1$ being the sole point of type (C) and $x_k$ being the sole point of type (R).
\end{enumerate}
\end{prop}


\subsection{The case $m=2$}


It remains to consider the case of unrestricted PCP zig-zags of modality $m=2$.\\

\begin{defn}
Let $n \geq 3$. For $2 \leq k \leq n-1$ set $\kappa(n,k) = (1, 2, \ldots, k-1)(k) \cdots (n) \in S_n$ and define

\[
\rho_2(n,k)=[\kappa(n,k)]^{-1} \circ \rho_e(n,k) \circ \kappa(n,k)
\]
\end{defn}

\begin{lem}\label{lem:bimodal}
Fix $n \geq 3$ and $2 \leq k \leq n-1$ and set $\rho_2=\rho_2(n,k)$. Then the following conditions hold:

\begin{enumerate}
\item $\rho_2(k)=k-1$.
\item $\rho_2(k-1)=n$.
\item If $i \leq j<k$ then $\rho_2(i) \leq \rho_2(j)$.
\item If $k<i \leq j$ then $\rho_2(i) \leq \rho_2(j)$.
\item If $i<k<j$ then $\rho_2(j)<\rho_2(i)$. 
\end{enumerate}

\noindent Moreover, $\rho_2(n,k)$ is the only element of $S_n$ to satisfy these conditions. Finally, $\rho_2(n,k)$ is a full $n$-cycle if and only if $\gcd(n-k, k-1)=1$. 
\end{lem}

\begin{proof}
Set $\rho=\rho_e(n,k)$ and $\kappa=\kappa(n,k)$ so that $\rho_2=\kappa^{-1}\rho\kappa$. One readily checks conditions (1) and (2) from this equation. Now suppose $i \leq j<k$. Then we have $\rho_2(i)=\kappa^{-1}(i+1+(n-k))$ and $\rho_2(j)=\kappa^{-1}(j+1+(n-k))$. Since $i \leq j$ and $\kappa$ preserves this ordering except at $k$, we see that $\rho_2(i) \leq \rho_2(j)$ unless $i+1+(n-k)=1$, which is impossible. Similarly, if $k<i \leq j$, then $\rho_2(i)=\kappa^{-1}(i-(k-1))$, and $\rho_2(j)=\kappa^{-1}(j-(k-1))$. Again we see that $\rho_2(i) \leq \rho_2(j)$ unless $i-(k-1)=1$, i.e. that $i=k$. This is also impossible.

To prove condition (5), it is enough by (4) to argue that if $i<k$ then $\rho_2(n)<\rho_2(i)$. $\rho_2(i)=\kappa^{-1}(i+1+(n-k))$, whereas $\rho_2(n)=\kappa^{-1}\rho(n)=\kappa^{-1}(n-k+1)$. Again by the fact that $\kappa$ preserves ordering except at $k$, it follows that $\rho_2(n)<\rho_2(i)$.

Suppose that $\rho \in S_n$ satisfies conditions (1)-(5) above. We wish to show that $\rho=\rho_2(n,k)$. Condition (1) implies that this second set has image including $n$, and condition (2) determines the image of the last remaining index. The monotonicity conditions (3) and (4) impose a strict ordering on the two subsets $\{1, \ldots, k-1\}$ and $\{k+1, \ldots, n\}$ of remaining indices, and condition (5) implies that the image of the second subset must be completely below that of the first. It follows that $\rho(i)$ is determined for each $i$, and hence $\rho=\rho_2(n,k)$.

Finally, since $\rho_2$ is conjugate to $\rho$ the two share the same cycle type, and hence by Proposition \ref{cycle}, $\rho_2$ is an $n$-cycle if and only if $\gcd(n-k,n-1)=1$. 
\end{proof}

\begin{prop}\label{prop:bimodal}
Let $f: I \to I$ be a PCP $\lambda$-zig-zag map with $\floor{\lambda}=2$. Then $f \in \PA(2)$ if and only if the following hold:

\begin{enumerate}
\item Under the ordering $x_1<x_2<\ldots<x_n=1$ of the orbit of $x=1$, we have $\rho(f)=\rho_2(n,k)$ for some $k$.
\item The $\{x_i\}$ satisfy the conclusion of Proposition \ref{type}, with $x_{k-1}$ being the sole point of type (C) and $x_k$ being the sole point of type (R).
\end{enumerate} 
\end{prop}

\begin{proof}
Suppose that $f$ is of pseudo-Anosov type. By Proposition \ref{type}, $f$ has exactly one postcritical point in the interior of $I_1$, as well as another at $x=\lambda^{-1}$. Let $\rho \in S_n$ describe the action of $f$ on the periodic orbit of $x=1$. Let $F$ be the exterior left-veering thick interval map projecting to $f$ and let $\tau$ be the invariant generalized train track for $F$. Since $f$ is of pseudo-Anosov type, $\tau$ is finite and has structure described by Lemma \ref{PCP} and Proposition \ref{singleint}. In particular, $\tau$ has a single interior loop, and hence there is a single loop $\gamma$ of $\tau$ whose image is this interior loop. Let $x_k$ be the postcritical point of $f$ to which this loops projects. Since the single interior loop of $\tau$ necessarily maps to the loop projecting to $x_n=1$, we have $\rho^2(k)=n$.

Since $\gamma$ maps to the unique interior loop of $\tau$, $\gamma$ is the unique loop such that $F(\gamma) \in L_2$. Moreover, since $\rho^2(k)=n$ we see that $F(F_\ast(\gamma))$ is after $L_1$ and before $L_2$. In particular, there are no loops of $\tau$ after $F_\ast(\gamma)$ and before $\gamma$. Therefore $F_\ast(\gamma)$ projects to $x_{k-1}$, and so $\rho(k)=k-1$. Since $\rho^2(k)=n$ it now follows that $\rho(k-1)=n$. It is now readily checked that $\rho$ satisfies the five conditions of Lemma \ref{lem:bimodal}, and hence $\rho=\rho_2(n,k)$.

Suppose on the other hand that conditions (1) and (2) above hold. Let $F$ be the exterior left-veering thick interval map projecting to $f$, and let $\tau$ be the invariant generalized train track of $F$. Note that $\tau$ has a loop $\gamma$ projecting to $x=1$. Let $F^a$ be the first iterate that sends $\gamma$ to an interior loop. All forward images of $\gamma$ before this must be exterior loops, by the definition of $\rho_2$ and the fact that $F$ is exterior left-veering. Note that $a \geq 2$ and that $F^{a-1}(\gamma)$ is a loop projecting to $x_k$, since its image is an interior loop.

We claim that $F^{a+1}(\gamma)=\gamma$. Indeed, $x=\lambda^{-1}$ must be the $(k-1)$-st postcritical point of $f$ from left to right, since $\rho_2(n,k)(k-1)=n$. Moreover, since $f$ has a single postcritical point in $I_1$, this point must be the $k$-th postcritical point, immediately proceeding $x=\lambda^{-1}$, and since $\rho_2(n,k)(k-1)=n$ we see that this point maps to $x=\lambda^{-1}$. Thus the single interior loop of $\tau$ is mapped into the fat vertex over $x=1$, and hence after isotopy becomes simply $\gamma$. In particular, $\tau$ is finite, and hence $f$ is of pseudo-Anosov type.
\end{proof}

\begin{rmk}
Propositions \ref{even}, \ref{odd}, and \ref{prop:bimodal} imply that if $f \in \PA(m)$ then 

\[
\rho(f)=\begin{cases}
\rho_e(n,k) & \text{if $m \geq 4$ even}, \\
\rho_o(n,k) & \text{if $m \geq 3$ odd}, \\
\rho_2(n,k) & \text{if $m=2$}
\end{cases}
\]

\noindent Therefore, when $m$ is unspecified we will write $\rho(f)=\rho_m(n,k)$.
\end{rmk}


\subsection{The proof of Theorem \ref{bijection}}


We are now ready to prove Theorem \ref{bijection}, which we restate here for convenience. 

\begin{mainthm}
Fix $m \geq 2$ and let $\Phi: \PA(m) \to \bbQ \cap (0,1)$ be the map defined by

\[
\Phi(f)=\frac{n-k}{n-1} \hspace{5mm} \text{if $\rho(f)=\rho_m(n,k)$}
\]

\noindent Then $\Phi$ is a bijection. Moreover, for each $p \geq 4$ the image $\Phi(\PA(m,p))$ consists of the set of reduced rationals in $(0,1)$ of denominator $p-2$.
\end{mainthm}
\begin{proof}
To fix ideas, let $m=2$. Suppose $f$ and $g$ are two bimodal unrestricted PCP zig-zags of pseudo-Anosov type, and that $\rho(f)=\rho(g)=\rho_2(n,k)$ for some $n, k$ satisfying $(n-k, n-1)=1$. We claim that $f=g$. Indeed, $\rho_2(n,k)$ and bimodality determine the transition matrix $M$ of $f$, by Proposition \ref{prop:bimodal}. In particular, $f$ and $g$ are uniform expanders with the same topological entropy, and hence the same slope. It follows that $f=g$, and so the map is injective.

To prove surjectivity, it is enough to show that for every $\rho_2(n,k)$ with $(n-k, n-1)$ there is a bimodal unrestricted PCP zig-zag $f$ of pseudo-Anosov type that acts on the orbit of $x=1$ as $\rho_2(n,k)$. We do this by constructing the exterior left-veering thick interval map and then projecting onto the horizontal coordinate to obtain a zig-zag with the necessary combinatorics. Indeed, fix some $\rho_2(n,k)$ satisfying $(n-k, n-1)=1$ and let $\tau$ be a train track consisting of real edges $e_1, \ldots, e_n$ and infinitesimal edges $f_1, \ldots, f_{n-1}$, with $f_j$ joining $e_j$ to $e_{j+1}$. Further adorn $\tau$ with infinitesimal loops as follows:

\begin{enumerate}
\item Loops $\gamma_0$ and $\gamma_n$ on the left of $e_1$ and the right of $e_n$, respectively.
\item An upward-pointing loop $\gamma_{k-1}$ attached to the left of $e_k$.
\item Downward-pointing loops $\gamma_j$ attached to the right of $e_j$ for $j=1, \ldots, \hat{k}, \ldots, n-1$. 
\end{enumerate}

Define $F$ to be the exterior left-veering train track map permuting the $\gamma_k$ according to $\rho_2(n,k)$, while also fixing $\gamma_0$. It is not hard to see by the construction of $F$ and the structure of $\rho_2(n,k)$ that $F$ preserves $\tau$.

We claim that the transition matrix $M$ of the $e_j$ is irreducible, i.e. for every pair of edges $e_{j_1}, e_{j_2}$ some iterate of $F$ maps $e_{j_1}$ across $e_{j_2}$. Since each $\gamma_j$ for $j \geq 1$ maps to $\gamma_n$, every $e_j$ eventually maps across $e_n$, and so it is enough to prove that each $e_j$ is eventually covered by some forward image of $e_n$. Since $\gamma_n$ is on the right of $e_n$ and $\gamma_n$ maps to $g_j$ for each $j \geq 1$, it follows that $e_n$ must map across the edge to the left of $\gamma_j$ for each positive $j \neq k$: that is, $e_n$ eventually maps across $e_j$ for each $j \neq k-1$.

It remains to prove that $e_n$ eventually maps across $e_{k-1}$. But some $e_j$ does this, and since $e_n$ already maps to every edge besides $e_{k-1}$ we will be done unless $e_{k-1}$ is the only edge that maps across $e_{k-1}$. This is impossible: the image of $e_{k-1}$ is entirely before the first turn of $\tau$, since $\gamma_{k-1}$ maps to $\gamma_n$, and so the image of some other $e_j$ covers $e_{k-1}$ between the first and second turns. Thus $M$ is irreducible.

Let $\lambda$ be the spectral radius of $M$. By the Perron-Frobenius theorem, $\lambda$ is a simple eigenvalue for $M$, and $M$ has positive left- and right-eigenvectors for $\lambda$. Let $u=(u_1, \ldots, u_n)$ be the unique left eigenvector such that $\sum_i u_i=1$. Then assigning each $e_j$ the length $u_j$ (and declaring each infinitesimal edge to be length $0$) we obtain a uniform $\lambda$-expander $f$ by projecting the action of $F$ onto the horizontal coordinate. Since $u$ is positive none of the $e_j$ are collapsed, and so $f$ is a PCP zig-zag having the necessary combinatorics: $x=1$ is periodic and $f$ acts on this orbit by $\rho_2(n,k)$. Thus the map is surjective for $m=2$.

Finally, consider the image $\Phi(\PA(2,p))$. For a map $f \in \PA(2,p)$, the fraction $\Phi(f)$ is reduced and has denominator $p-2$, since $\rho(f) = \rho_2(p-1, k)$ for some $k$ such that $k-1$ is coprime to $p-2$. The fact that $\Phi(\PA(2,p))$ contains all such fractions follows because $\Phi$ is a bijection onto $\bbQ \cap (0,1)$.

A similar argument works for $m \geq 3$ odd and $m \geq 4$ even. In this case, the placement of infinitesimal loops as in steps 1-3 above is chosen accordingly to ensure that $\tau$ is invariant under the exterior left-veering thickening. 

 \end{proof}

\section{The digit polynomial $D_f$}

In this section we prove Theorem \ref{Perm}, restated below for convenience. This theorem generalizes Lemma 2.5 in \cite{H}, which treats the unimodal case. In \cite{H}, the result is phrased in terms of the kneading sequence of the unique critical point.

\begin{mainthm}\label{Perm}
Suppose $f \in \PA(m)$ for $m \geq 2$ with $\Phi(f)=\tfrac{a}{b} \in \bbQ \cap (0,1)$ in lowest terms. Define $L: [0, b] \to \bbR$ by $L(t)=\frac{a}{b} \cdot t$. Then

\[
D_f(t)=t^{b+1}+1-\sum_{i=1}^b c_it^{b+1-i},
\]

\noindent where the $c_i$ satisfy

\begin{equation}
c_i=\begin{cases}
m & \text{if $L(t) \in \bbN$ some $t \in [i-1,i]$}\\
m-2 & \text{otherwise}
\end{cases}
\end{equation} 

\noindent In particular, $c_i=c_{b-i}$, so $D_f$ is reciprocal: that is,

\[
D_f(t)=t^{b+1} D_f(t^{-1}).
\]
\end{mainthm}

\begin{ex}
Let $f, g \in \PA(2)$ be the zig-zags such that $\Phi(f)=\tfrac{1}{7}$ and $\Phi(g)=1-\tfrac{1}{7}=\tfrac{6}{7}$. According to Theorem \ref{Perm} and Figure \ref{fig:polyex1}, we have

\[
D_f(t)=t^8-2t^7-2t+1
\]

\noindent and 

\[
D_g(t)=t^8-2t^7-2t^6-2t^5-2t^4-2t^3-2t^2-2t+1
\]
\end{ex}

\begin{figure}[h!]
\begin{subfigure}{.4\paperwidth}
\centering
\includegraphics[scale=.25]{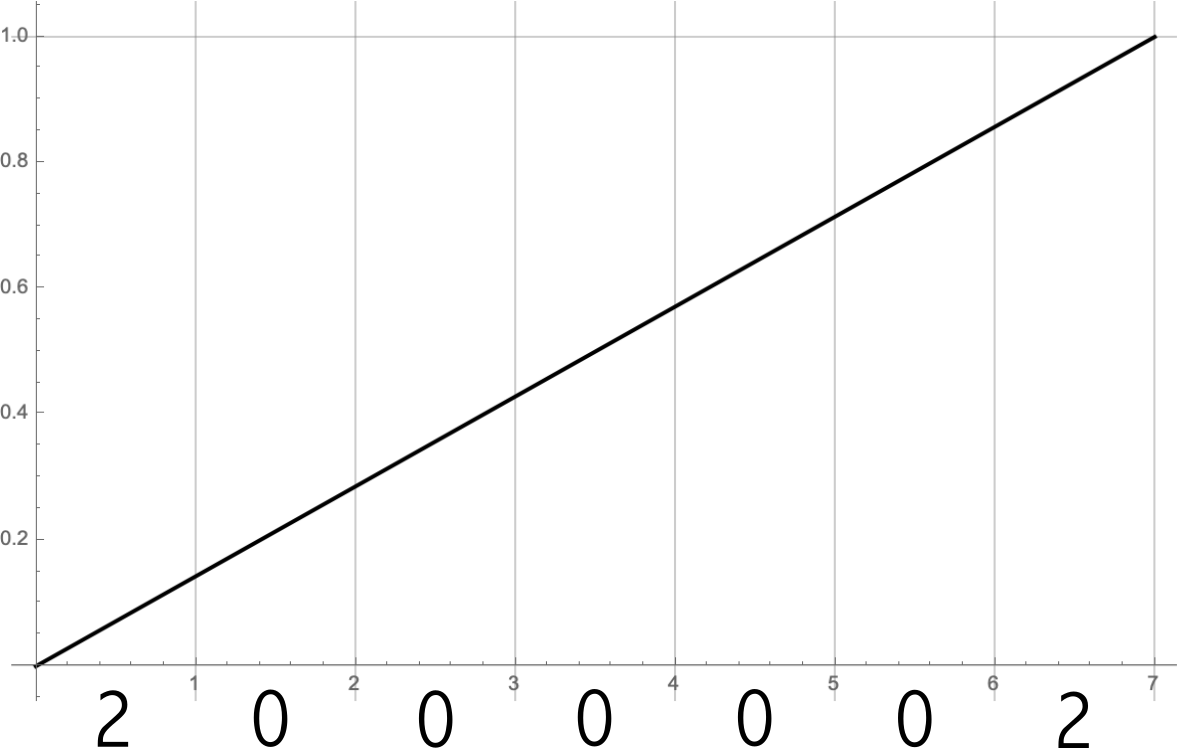}
\caption{Computing the coefficients of $D_f$ for the unique $f \in \PA(2)$ such that $\Phi(f)=\frac{1}{7}$.}
\end{subfigure}
\begin{subfigure}{.4\paperwidth}
\centering
\includegraphics[scale=.25]{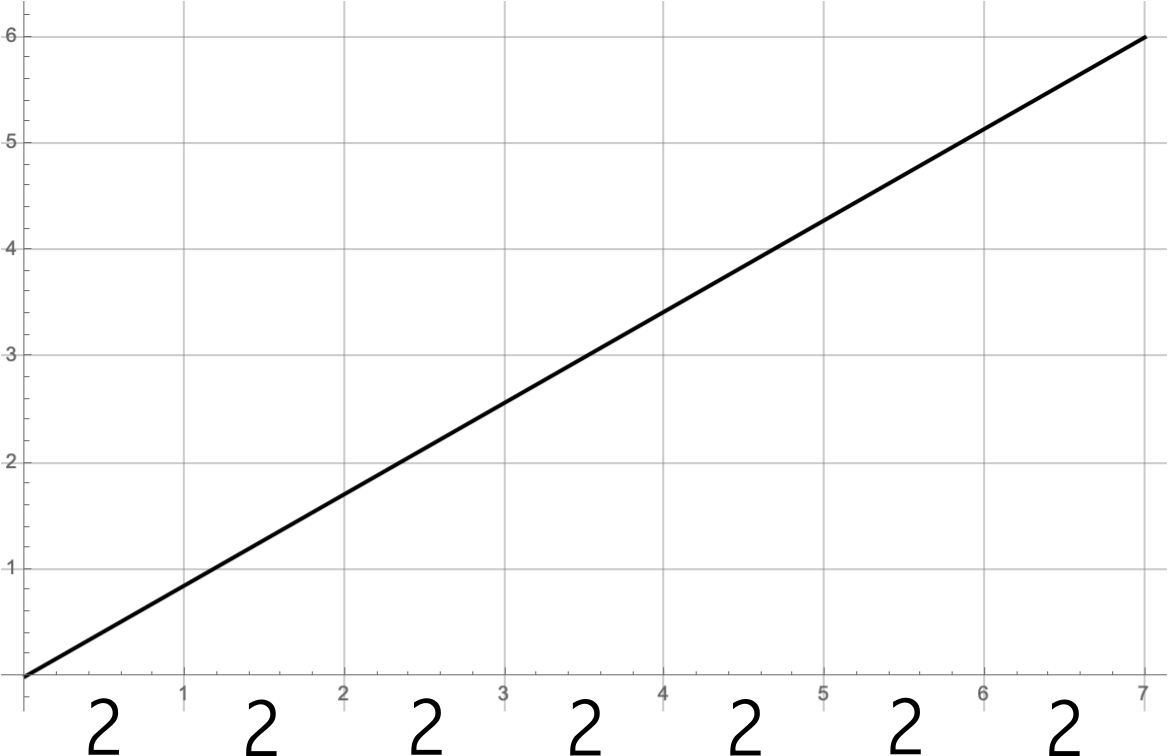}
\caption{Computing the coefficients of $D_g$ for the unique $g \in \PA(2)$ such that $\Phi(g)=\frac{6}{7}$.}
\end{subfigure}
\caption{Computing the digit polynomials of bimodal maps using Theorem \ref{Perm}.}\label{fig:polyex1}
\end{figure}

\begin{ex}
Let $p, q \in\PA(7)$ be the zig-zags such that $\Phi(p)=\tfrac{4}{13}$ and $\Phi(q)=1-\tfrac{4}{13}=\tfrac{9}{13}$. According to Theorem \ref{Perm} and Figure \ref{fig:polyex2}, we have

\[
D_p(t)=t^{14}-7t^{13}-5t^{12}-5t^{11}-7t^{10}-5t^9-5t^8-7t^7-5t^6-5t^5-7t^4-5t^3-5t^2-7t+1
\]

\noindent and 

\[
D_q(t)=t^{14}-7t^{13}-7t^{12}-7t^{11}-5t^{10}-7t^9-7t^8-5t^7-7t^6-7t^5-5t^4-7t^3-7t^2-7t+1
\]
\end{ex}

\begin{figure}[h!]
\begin{subfigure}{.35\paperwidth}
\centering
\includegraphics[scale=.2]{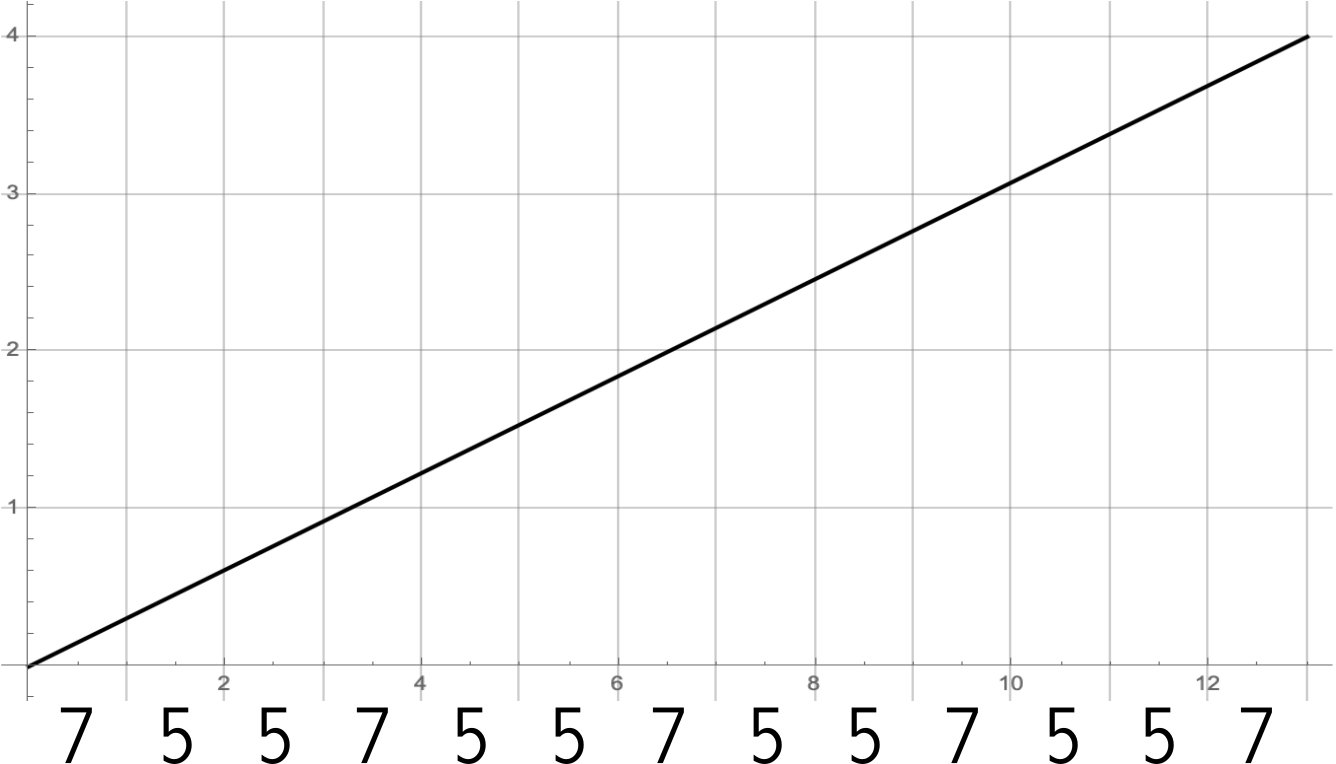}
\caption{Computing the coefficients of $D_p$ for the unique $p \in \PA(7)$ such that $\Phi(p)=\frac{4}{13}$.}
\end{subfigure}
\begin{subfigure}{.45\paperwidth}
\centering
\includegraphics[scale=.2]{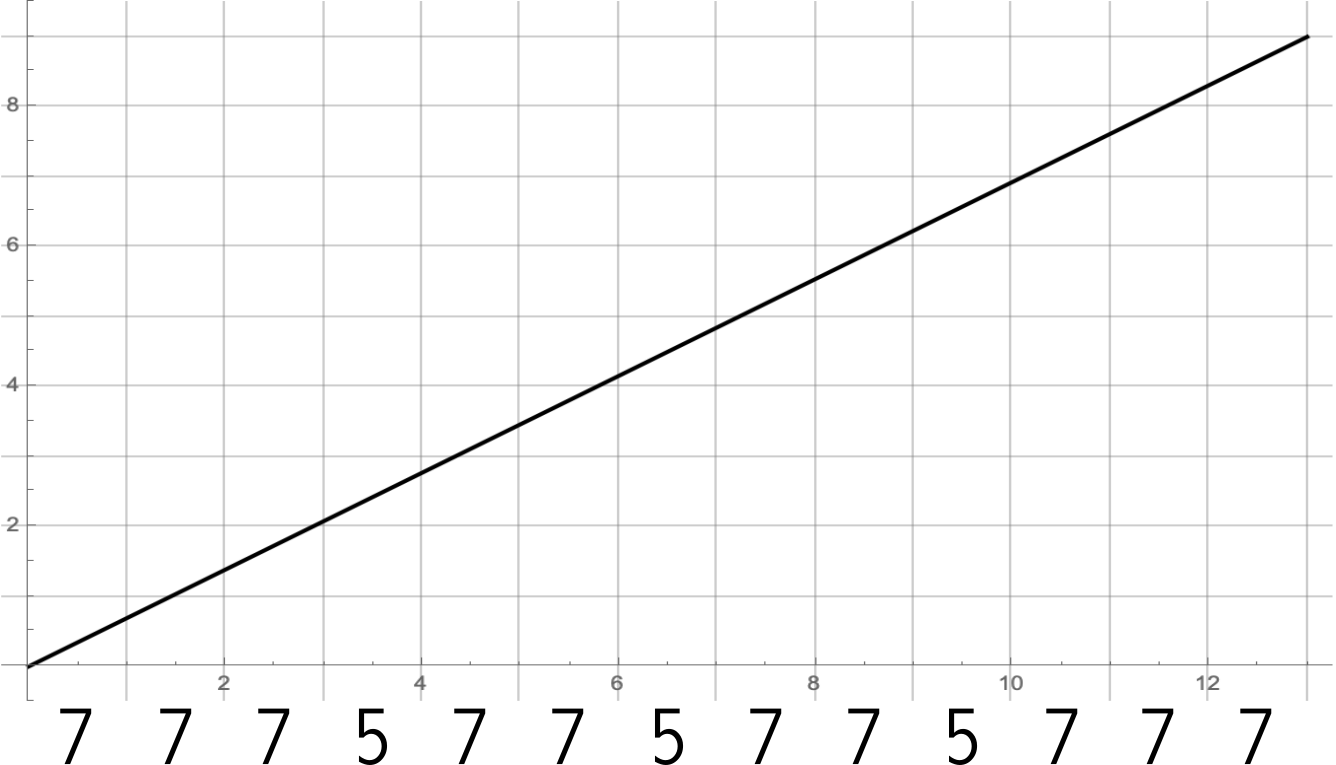}
\caption{Computing the coefficients of $D_q$ for the unique $q \in \PA(7)$ such that $\Phi(q)=\frac{9}{13}$.}
\end{subfigure}
\caption{Computing the digit polynomials of $7$-modal maps using Theorem \ref{Perm}.}\label{fig:polyex2}
\end{figure}

We first prove Theorem \ref{Perm} for the case that $f$ has modality $m \geq 4$ even. This case being completed, we will then use it to deduce the cases for $m \geq 3$ odd and $m=2$. We proceed via a sequence of lemmas.

\begin{lem}\label{coefs1}
Suppose $f \in \PA(m)$ for $m \geq 4$ even. Let $n$ be minimal such that $f^n(1)=1$, and let $f_i(x)$ be the defining linear branch of $f$ at $f^{i-1}(1)$; that is, $f^i(1)=f_i(f^{i-1}(1))$. Then there exist constants $c_i$ for $i=1, \ldots, n-2$ such that 

\[
f_i(x)=\begin{cases}
\lambda x -c_i & i=1, \ldots, n-2\\
m-\lambda x & i=n-1\\
2-\lambda x & i=n
\end{cases}
\]
\end{lem}

\begin{proof}
The sole point of type $R$ (cf. Proposition \ref{type}) in the orbit of $x=1$ must map to the point of type $C$, by Proposition \ref{singleint}, and this latter point is precisely $x=\lambda^{-1}$, and hence maps to $x=1$ by $f_n(x)=2-\lambda x$. At the point of type $R$ we have $f_{n-1}(x)=m-\lambda x$.

All other points in the orbit of $x=1$ are of type $P_1$ or $P_2$, and the corresponding $f_i$ are $f_i(x)=\lambda x - (m-2)$ and $f_i(x)=\lambda x-m$, respectively.
\end{proof}

\begin{lem}\label{coefs2}
Suppose $f \in \PA(m)$ for $m \geq 4$ even. Let $D_f(t)$ be the digit polynomial of $f$. Then

\[
D_f(t)=t^n+1-\sum_{i=1}^{n-1} c_it^{n-i}
\]

\noindent where, as in Lemma \ref{coefs1}, the $c_i$ are defined by $f_i(x)=\lambda x- c_i$ for $1 \leq i \leq n-2$ and $f_{n-1}(x)=c_{n-1}-\lambda x$. In particular, $c_1=c_{n-1}=m$.
\end{lem}

\begin{proof}
For $i=0, \ldots, n$ define $g_i(t) \in \bbZ[t]$ by $g_i(\lambda)=f^i(\lambda)$. Thus for example $g_0(\lambda)=1$ and $g_1(\lambda)=\lambda-m$. Here it is important that we are treating $\lambda$ as a formal variable, rather than an algebraic integer satisfying a polynomial relation.

Observe that $g_{i+1}(\lambda)=f_{i+1}(g_i(\lambda))$ by definition. In particular, by Lemma \ref{coefs1} for $i=0, \ldots, n-3$ we have $g_{i+1}(\lambda)=\lambda g_i(\lambda)-c_i$, and so inductively we see that for 

\[
g_i(\lambda)=\lambda^i-c_1\lambda^{i-1}-\cdots-c_{i-1}\lambda-c_i \hspace{5mm} \text{for $0 \leq i \leq n-2$}
\]

\noindent We therefore have the equalities

\[
g_{n-1}(\lambda)=c_{n-1}-\lambda^{n-1}+c_1\lambda^{n-2}+\cdots+c_{n-3}\lambda^2+c_{n-2}\lambda
\]

\noindent and

\[
g_n(\lambda)=2-c_{n-1}\lambda-c_{n-2}\lambda^2- \cdots - c_1\lambda^{n-1}+\lambda^n
\]

\noindent By the definition of $D_f(t)$ and the fact that $g_n(\lambda)=f^n(1)=1$, we have $D_f(\lambda)=g_n(\lambda)-1$. 
\end{proof}

\begin{lem}\label{coefs3}
Suppose $f \in \PA(m)$ for $m \geq 4$ even. Let $n$ be minimal such that $f^n(1)=1$ and suppose that $f$ acts on the orbit of $x=1$ by the permutation $\rho(f)=\rho_e(n,k)$ with $\gcd(n-k, k-1)=1$. Let $D_f(t)=t^n+1-\sum_{i=1}^{n-1} c_it^{n-i}$. Then for $i=1, \ldots, n-1$ we have

\[
c_i=\begin{cases}
m & \text{if} \ \rho(f)^{i-1}(n) \geq k\\
m-2 & \text{if} \ 2 \leq \rho(f)^{i-1}(n) \leq k-1
\end{cases}
\]
\end{lem}

\begin{proof}
The permutation $\rho(f)$ is defined such that postcritical points corresponding to symbols $i$ satisfying $k+1 \leq i \leq n$ are of type $P_2$, whereas those corresponding to $i$ satisfying $2 \leq i \leq k-1$ are of type $P_1$. Finally, the point corresponding to the symbol $k$ is of type $R$, specifically $f^{n-2}(1)$. The result follows by Lemmas \ref{coefs1} and \ref{coefs2}.
\end{proof}

\begin{lem}\label{coefs4}
Theorem \ref{Perm} holds for $f \in \PA(m)$ with $m \geq 4$ even.
\end{lem}

\begin{proof}
We have already shown that $c_0=c_n=1$ and that $c_i \in \{m-2, m\}$ for $1 \leq i \leq n-1$. By Lemma \ref{coefs3} it is enough to understand when $\rho^i(n) \geq k$. Let $\rho' \in S_{n-1}$ be the permutation obtained by deleting the symbol $1$ from the cycle decomposition of $\rho$ and decreasing all remaining labels by $1$, as in the proof of Proposition \ref{cycle}. As observed previously, $\rho'$ acts on $\{1, \ldots, n-1\}$ as addition by $n-k$ modulo $n-1$. Therefore, if we set $q=\frac{n-k}{n-1}$, the values $L_q(t)=qt$ for $t=1, \ldots, n-1$ satisfy

\[
L_q(t)-\floor{L_q(t)}=\frac{(\rho')^t(n-1)}{n-1}
\]

To see why this is true, observe first that the left-hand side is the fractional part of $L_q(t)$. This quantity is a rational number with denominator $n-1$, and the numerator increases by $n-k$ modulo $n-1$. Since $(\rho')^t(n-1)$ also changes in this way, it remains to note that 

\[
L_q(1)=\frac{n-k}{n-1}=\frac{(\rho')(n-1)}{n-1}
\]

\noindent Now we observe that 

\begin{align*}
c_i=m & \iff \rho^{i-1}(n) \geq k \\
& \iff (\rho')^{i-1}(n-1) \geq k-1\\
& \iff L_q(i-1)-\floor{L_q(i-1)} \geq \frac{k-1}{n-1}\\
& \iff L_q(t) \in \bbN \ \text{for some $t \in [i-1,i]$.} 
\end{align*}

\noindent The last equivalence holds because $L_q(t)=qt$ is a line of slope $\frac{n-k}{k-1}$. The proof is complete.
\end{proof}

It remains to prove Theorem \ref{Perm} for the case when the modality of $f$ is $m=2$ or $m \geq 3$ odd. Recall that $\rho_2(n,k)=\kappa^{-1}(n,k) \circ \rho_e(n,k) \circ \kappa(n,k)$, where $\kappa(n,k)=(1, 2, \ldots, k-1) \in S_n$.

\begin{lem}\label{coefs5}
For any $l \geq 0$, $\rho_2^l(n) \geq k$ if and only if $\rho_e^l(n) \geq k$, and in this case $\rho_2^l(n)=\rho_e^l(n)$.
\end{lem}

\begin{proof}
Suppose $\rho_e^l(n) \geq k$. Then $\rho_2^l(n)=\kappa^{-1}\rho_e^l\kappa(n)=\rho_e^l(n)$, since $\kappa(j)=j$ for all $j \geq k$. Similarly, if $\rho_2^l(n) \geq k$ it follows that $\rho_e^l(n) \geq k$ as well.
\end{proof}

We now complete the proof of Theorem \ref{Perm}.

\begin{proof}[Proof of Theorem \ref{Perm}]
One quickly verifies that Lemmas \ref{coefs1} through \ref{coefs3} hold for $f$ an $m$-modal zig-zag of pseudo-Anosov type for $m=2$ and $m \geq 3$ odd, after replacing all instances of $\rho_e(n,k)$ with either $\rho_2(n,k)$ or $\rho_o(n,k)$. Lemmas \ref{coefs3} and \ref{coefs5} now imply the Theorem for $m=2$.

For $m \geq 3$ odd, recall that if we delete the symbol $0$ from the cycle type of $\rho_o(n,k)$ we obtain the permutation $\rho_e(n,k)$. Deleting this symbol corresponds to ignoring the linear map $f_{n+1}(x)=1-\lambda x$. This linear branch is not used to compute $D_f$, since the fact that $f \in \PA(m)$ for $m \geq 3$ odd implies that $f^n(1)=0$, terminating the process of constructing $D_f$ before the $(n+1)$-st step (cf. Definition \ref{defn:digit}).

The arguments of Lemmas \ref{coefs1} through \ref{coefs4} now prove the Theorem in this case.
\end{proof}

\section{A family of pseudo-Anosovs with Salem dilatation}

In this section we provide an application of the theory developed over the course of this paper. Recall that a \textit{Salem number}, introduced in \cite{S}, is a real algebraic integer $\lambda>1$ such that all Galois conjugates of $\lambda$ are contained within the closed unit disc, with at least one of these conjugate lying on the unit circle. It is not hard to show that $\lambda^{-1}$ must be among the Galois conjugates of $\lambda$ in this case, and that all other conjugates lie on the unit circle. In particular, a Salem number is a Perron number of even degree $d=2g$. If $p(x) \in \bbZ[x]$ is the minimal polynomial of a Salem number, then $p(x)$ is \textit{reciprocal}, i.e. 

\[
p_\ast(x):=x^{\deg(p)}p(x^{-1})=p(x)
\]

It is well-known that if $f(x) \in \bbZ[x]$ is a reciprocal polynomial of degree $d=2g$ then $f(x)=x^gq(x+x^{-1})$ for some integral polynomial $q$. We call $q(x)$ the \textit{companion polynomial} to $f(x)$. If $\deg(f)=2g+1$ then $f(x)=(x+1)f_1(x)$ for $f_1(x)$ reciprocal of even degree, and therefore $f(x)=(x+1)x^gq(x+x^{-1})$ for some $q(x) \in \bbZ[x]$. In this case we again call $q$ the companion polynomial of $f$.

Note that there is a bijection between roots of $q$ and pairs of roots of $f$: if $\lambda$, $\lambda^{-1}$ are roots of $f$ then $\lambda+\lambda^{-1}$ is a root of $q$. Moreover, if $|\lambda|=1$, then $\lambda^{-1}=\overline{\lambda}$ and so the root $\lambda+\lambda^{-1}$ of $q$ is a real number contained in the interval $[-2,2]$. We refer to this interval, with or without its endpoints, as the \textit{critical interval}.

In the case of a Salem number $\lambda$ of degree $2g$, the companion polynomial is irreducible of degree $g$ with dominant root $\lambda+\lambda^{-1}>2$ and the remaining $g-1$ roots in the critical interval. In particular, the companion polynomial has all real roots, so $\lambda+\lambda^{-1}$ is a totally real algebraic integer of degree $g$.


Recall that a \textit{translation surface} is a pair $(X,\omega)$ of a Riemann surface $X$ equipped with a non-zero abelian differential $\omega$, i.e. a holomorphic one-form. If $X$ is of genus $g$, then $\omega$ has $2g-2$ zeros, counting multiplicity. Let $\Sigma$ be the collection of these zeros. Then away from $\Sigma$, $X$ has a Euclidean structure: in other words, $X \setminus \Sigma$ admits an atlas of charts to $\bbC$ whose transition functions are translations. In the neighborhood of a zero $p$ of order $k$, $X$ has the structure of $2(k+1)$ metric half-discs glued together, so that the total angle about $p$ is $2\pi(k+1)$.


Fixing $g>1$, let $\mu$ be a positive integer partition of $2g-2$. We think of $\mu$ as describing the multiplicities of the zeros of an abelian differential $\omega$ on $X$. The \textit{stratum} $\calH(\mu)$ is the collection of genus $g$ translations surfaces with zero orders specified by $\mu$.

Given some $(X,\omega) \in \calH(\mu)$ and $A \in \SL_2(\bbR)$, define $A \cdot (X,\omega)$ to be the translation surface obtained by post-composing the charts of $(X, \omega)$ into $\bbR^2 \cong \bbC$ by $A$.

The \textit{Veech group} of a translation surface $(X,\omega)$, denoted $\SL(X, \omega)$, is the stabilizer of $(X, \omega)$ under the action by $\SL_2(\bbR)$. The \textit{trace field} of $(X,\omega)$ is the field $K$ obtained by adjoining to $\bbQ$ the traces of all elements of $\SL(X,\omega)$. By a result of M\"oller in \cite{Mo2}, the degree of the trace field satisfies

\[
[K:\bbQ] \leq g(X)
\]

If the degree of this extension is maximal, i.e. is equal to the genus of $X$, then we say that $(X, \omega)$ is \textit{algebraically primitive}. One remarks that such surfaces cannot arise as covers of translation surfaces of lower genus: if $\pi: (X, \omega) \to (Y, \eta)$ is a translation cover, then the trace fields of $X$ and $Y$ coincide, and the result of M\"oller mentioned above now shows that if $X$ is algebraically primitive, then $g(X)=g(Y)$.

A translation surface is called \textit{Veech} if its Veech group is a lattice in $\SL_2(\bbR)$, i.e. as large as possible. The $\GL_2(\bbR)$-orbit of a Veech surface is called a \textit{Teichm\"uller curve}. By results of M\"oller \cite{Mo1} and Apisa \cite{A}, there are only finitely many algebraically primitive Teichm\"uller curves in any genus $g \geq 3$. Therefore, it is interesting to find algebraically primitive surfaces with non-trivial Veech group of arbitrarily high genus.

We now restate our last main result.

\begin{mainthm}
For each $g \geq 1$ define $f_g: I \to I$ to be the bimodal PCP zig-zag map of pseudo-Anosov type corresponding to $r_g=\tfrac{1}{2g} \in \bbQ \cap (0,1)$. Let $\lambda_g$ be the growth rate of $f_g$. Then the following are true for each $g \geq 1$.

\begin{enumerate}
\item $\lambda_g$ is a Salem number of degree $2g$.
\item The pseudo-Anosov $\phi_g$ obtained from $f_g$ is defined on a $(2g+2)$-punctured sphere $\Sigma_{0, 2g+2}$.
\item The translation surface $(X_g, \omega_g)$ obtained as the hyperelliptic double cover of $\Sigma_{0, 2g+2}$ is of genus $g$, and hence algebraically primitive.
\end{enumerate}
\end{mainthm}

\begin{rmk}
One might reasonably object that there is no such thing as a Salem number of degree $2$. Indeed, Salem numbers are normally defined to have at least one Galois conjugate on the unit circle, in which case all Salem numbers must be of degree at least four: if $\lambda$ is a Salem number and $\alpha \in S^1$ a Galois conjugate, then $\overline{\alpha}=\alpha^{-1}$ is also a conjugate of $\lambda$, and hence so is $\lambda^{-1}$. However, one may think of the quadratic units $\lambda>1$ as degenerate Salem numbers. That is what we choose to do here.
\end{rmk}


\subsection{$\lambda_g$ is Salem of degree $2g$}


As before, the proof of Theorem \ref{Salem} will proceed in a sequence of lemmas. The bulk of our efforts will be focused on establishing statement (1) of Theorem \ref{Salem}.

\begin{lem}\label{Salem1}
Let $D_g(t)$ be the digit polynomial of $f_g$. Then for all $g \geq 1$ we have

\[
D_g(t)=t^{2g+1}-2t^{2g}-2t+1=(t+1)d_g(t), \hspace{5mm} \text{where} \ d_g(t)=t^{2g}+1+3 \sum_{i=1}^{2g-1}(-1)^it^i
\]
\end{lem}

Thus for example $D_1(t)=(t+1)(t^2-3t+1)$ and $D_2(t)=(t+1)(t^4-3t^3+3t^2-3t+1)$. 

\begin{proof}
The fact that $D_g(t)=t^{2g+1}-2t^g-2t+1$ follows from Theorem \ref{Perm}. One readily checks that $(t+1)d_g(t)=D_g(t)$. 
\end{proof}

We wish to show that $d_g(t)$ is the minimal polynomial of a Salem number. To do this, we must prove that $d_g(t)$ has $2g-2$ roots on the unit circle, and also that the polynomial is irreducible. As we have seen, to accomplish the first task it will be enough to show that the companion polynomial $q_g(t)$, defined by $d_g(t)=t^gq_g(t+t^{-1})$, has $g-1$ roots in the critical interval. This is the content of the next two lemmas.

\begin{table}
\centering
\begin{tabular}{| l |}
\hline
$q_2(w)=w^2-3w+1$\\
\hline
$q_3(w)=w^3-3w^2+3$\\
\hline
$q_4(w)=w^4-3w^3-w^2+6w-1$\\
\hline
$q_5(w)=w^5-3w^4-2w^3+9w^2-w-3$\\
\hline
\end{tabular}
\caption{The first few companion polynomials $q_g$}
\label{tab1}
\end{table}

\begin{lem}\label{Salem2}
The companion polynomials $q_g(w)$ satisfy the following properties.

\begin{enumerate}
\item For all $g \geq 4$ w have the recurrence relation

\begin{equation}\label{rec}
q_{g+2}(w)=wq_{g+1}(w)-q_g(w)
\end{equation}

\item $q_g(2)=-1$ for all $n$. 
\item Each $q_g$ has a real root $\alpha_g>2$.
\item For all $g \geq 2$ we have $(-1)^gq_g(-2)>0$. In particular, $q_g(-2)$ and $q_{g+1}(-2)$ have opposite signs.
\end{enumerate}
\end{lem}

\begin{proof}
By definition, $q_g$ satisfies $q_g(t+t^{-1})=t^{-g}d_g(t)$, so relation (\ref{rec}) is equivalent to the recurrence

\[
d_{g+2}(t)=t^2 \left [ d_{g+1}(t)-d_g(t) \right ] + d_{g+1}(t)
\]

\noindent This formula is a straightforward consequence of the pattern of the coefficients of $d_g$, proving statement (1). Statement (2) follows inductively after noting that it holds for $q_2$ and $q_3$. Now note $q_2$ and $q_3$ are monic, so by equation (\ref{rec}) $q_g$ is monic for all $g$. In particular, $\lim_{w \to \infty} q_g(w)=\infty$ for all $g$, so the intermediate value theorem and statement (2) together imply statement (3).

Finally, to prove statement 4 observe that 

\[
d_g(-1)=1+3(2g-1)+1=6g-1>0
\]

\noindent for all $g \geq 2$. Since $d_g(-1)=(-1)^gq_g(-2)$, the result follows.
\end{proof}

\begin{lem}\label{Salem3}
For each $g \geq 2$ the polynomial $q_g$ has $g-1$ roots in $(-2,2)$. Moreover, if these roots are denoted $-2<a_1<\cdots<a_{g-1}<2$, then the $g$ roots of $q_{g+1}$ in $(-2,2)$, denoted $b_1, \ldots, b_g$, satisfy the ordering

\[
-2<b_1<a_1<b_2<a_2<\cdots<b_{g-1}<a_{g-1}<b_g<2
\]

\begin{proof}
We proceed inductively. The quadratic polynomial $q_2$ has a single root $a_1=\frac{3-\sqrt{5}}{2}$ in $(-2,2)$. By Lemma \ref{Salem3}(3), this is in fact the only root of $q_2$ in the critical interval. Since $q_3(-1)=-1$, $q_3(0)=3$, $q_3(1)=1$, and $q_3(2)=-1$, we see that $q_3$ has two roots $b_1, b_2 \in (-2,2)$ satisfying $-1<b_2<0$ and $1<b_2<2$. Since $0<a_1<1$, the claim is satisfied in this case.

Suppose now that the claim holds for all $n \leq g+1$. We may assume without loss of generality that $q_{g+1}(-2)<0$: the other case is essentially identical. Lemma \ref{Salem2}(4) implies that both $q_g$ and $q_{g+2}$ are positive at $w=-2$. Let $b_1<\ldots<b_g$ be the roots of $q_{g+1}$ in the critical interval. Then we have

\[
q_{g+2}(b_1)=b_1q_{g+1}(b_1)-q_g(b_1)=-q_g(b_1)
\]

\noindent By assumption, the smallest root $a_1$ of $q_g$ in $(-2,2)$ is greater than $b_1$. Since $q_g(-2)>0$, it follows that $q_g(b_1)>0$ and thus $q_{g+2}(b_1)<0$. Therefore $q_{g+2}$ has a root $c_1 \in (-2, b_1)$. Next we observe that 

\[
q_{g+2}(b_2)=b_2q_{g+1}(b_2)-q_g(b_2)=-q_g(b_2)
\]

\noindent Since $a_1<b_2<a_2$, we see that $q_g(b_2)<0$, so $q_{g+2}(b_2)>0$, implying that $q_{g+2}$ has a root $c_2 \in (b_1, b_2)$. Continuing in this fashion we find roots $c_i \in (b_{i-1}, b_i)$ for $2 \leq i \leq g$. Finally, note that since $q_g(2)=-1$ we have $q_g(w)<0$ for all $w \in (a_{g-1}, -2]$. Therefore, since $a_{g-1}<b_g<2$ we have 

\[
q_{g+2}(b_g)=-q_g(b_g)>0
\]

\noindent Since $q_{g+2}(2)=-1$, it follows that $q_{g+2}$ has a root $c_{g+1} \in (b_g,2)$. Lemma \ref{Salem2}(4) implies that $q_{g+2}$ has another root $\alpha_{g+2}>2$, so there cannot be any other roots of $q_{g+2}$. The proof is complete.
\end{proof}
\end{lem}

\begin{figure}
\centering
\includegraphics[scale=.5]{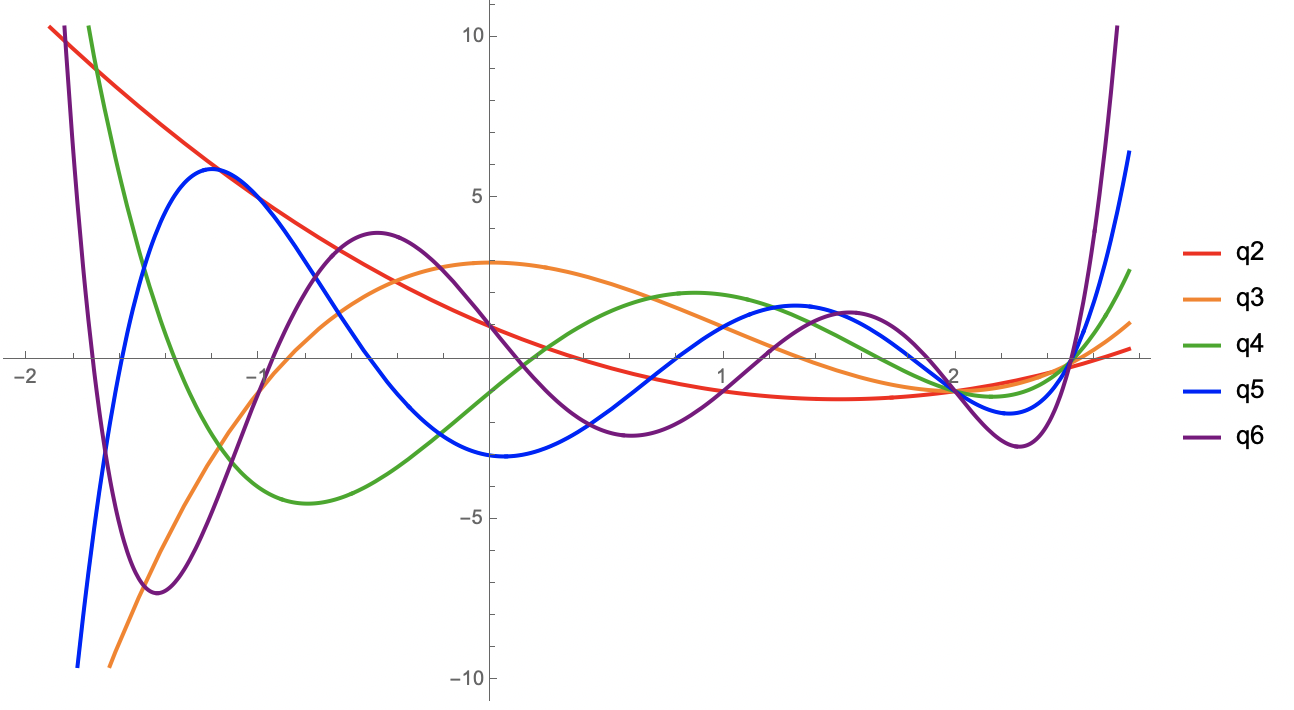}
\caption{The interlacing property of the roots of the companion polynomials $q_g$.}\label{fig:interlace}
\end{figure}

Lemma \ref{Salem3} implies that $d_g$ has $g-1$ pairs of roots on the unit circle in addition to a pair of positive real roots $\lambda_g$ and $\lambda_g^{-1}$. This is already enough to conclude that $\lambda_g$ is a Salem number. Indeed, $\lambda_g$ must be Galois conjugate to $\lambda_g^{-1}$, since otherwise the conjugates of $\lambda_g^{-1}$ would be contained in the closed unit disc, implying that they are all roots of unity, by Kronecker's theorem. This last is impossible, since $\lambda_g^{-1}<1$. Hence $\lambda_g$ and $\lambda_g^{-1}$ are Galois conjugate and $\lambda_g$ is a Salem number. It remains to determine whether $d_g$ is irreducible.

We follow a similar proof given by Shin in \cite{Sh}. Interestingly, this proof shows that the ``dual" Perron roots $\psi_g$ corresponding to the fraction $\tilde{r}_g = 1-r_g=\tfrac{2g-1}{2g}$ are also Salem of degree $2g$. 

\begin{lem}\label{Salem4}
Each $d_g(t)$ is irreducible over $\bbZ[t]$. Consequently, $\lambda_g$ is a Salem number of degree $2g$ for all $g \geq 1$.
\end{lem}

\begin{proof}
Since $\lambda_g$ is necessarily Galois conjugate to $\lambda_g^{-1}$, any factor of $d_g(t)$ other than the minimal polynomial of $\lambda_g$ must have all roots on the unit circle, and therefore must be cyclotomic, by Kronecker's theorem. Suppose therefore that $e^{2\pi i/m}$ is a root of $d_g(t)$. Then we have 

\[
D_g(e^{2\pi i/m})=e^{(2g+1)\cdot 2\pi i/m}-2e^{2g \cdot 2\pi i/m}-2e^{2 \pi i/m}+1=0
\]

\noindent We take real and imaginary parts to obtain the system of equations

\begin{equation}\label{Salem4a}
\begin{cases}
\cos \left ( \frac{(2g+1)2\pi}{m} \right ) - 2 \cos \left ( \frac{2g \cdot 2\pi}{m} \right ) - 2 \cos \left ( \frac{2\pi}{m} \right ) + 1=0\\
\sin \left ( \frac{(2g+1)2\pi}{m} \right ) - 2 \sin \left ( \frac{2g \cdot 2\pi}{m} \right ) - 2 \sin \left ( \frac{2\pi}{m} \right )=0
\end{cases}
\end{equation}\

For the first equation, we use the formula $\cos(2x)=2\cos^2(x)-1$ for the first cosine term and the formula $\cos(a)+\cos(b)=2\cos(\frac{a+b}{2})\cos(\frac{a-b}{2})$ for the latter two terms to obtain

\begin{equation}\label{Salem4b}
\cos\left ( \frac{(2g+1)\pi}{m} \right ) \left [ \cos \left ( \frac{(2g+1)\pi}{m} \right ) - 2\cos \left ( \frac{(2g-1)\pi}{m} \right ) \right ] = 0
\end{equation}

\noindent Similarly for the second equation in (\ref{Salem4a}) we use the formula $\sin(2x)=2\sin(x)\cos(x)$ and the formula $\sin(a)+\sin(b)=2\sin(\frac{a+b}{2})\cos(\frac{a-b}{2})$ to find

\begin{equation}\label{Salem4c}
\sin \left ( \frac{(2g+1)\pi}{m} \right ) \left [ \cos \left ( \frac{(2g+1)\pi}{m} \right )-2\cos \left ( \frac{(2g-1)\pi}{m} \right ) \right ]=0
\end{equation}

Since $\sin(x)$ and $\cos(x)$ do not have any common roots, it must be the case that 

\[
\cos \left ( \frac{(2g+1)\pi}{m} \right ) - 2\cos \left ( \frac{(2g-1)\pi}{m} \right )=0
\]

\noindent Setting $\varphi=\tfrac{(2g-1)\pi}{m}$ we rewrite this as 

\begin{equation}\label{Salem4d}
\cos \left (\varphi+\frac{2\pi}{m} \right )-2\cos(\varphi)=0
\end{equation}

\noindent It follows that $-1/2 \leq \cos(\varphi) \leq 1/2$. In other words, since $\cos(\pi/3)=1/2$ and $\cos(2\pi/3)=-1/2$ we must have

\[
-\frac{2\pi}{3} \leq \varphi \leq -\frac{\pi}{3} \ \text{or} \ \frac{\pi}{3} \leq \varphi \leq \frac{2\pi}{3}
\]

\noindent Moreover, by equation (\ref{Salem4d}), $\varphi$ and $\varphi+\tfrac{2\pi}{m}$ are angles on the same side of the $y$-axis, since their cosines have the same sign. We claim that $\varphi$ must be in either the second or fourth quadrant. Suppose for contradiction that $\varphi$ is in the first quadrant, so that $\tfrac{\pi}{3} \leq \varphi \leq \tfrac{\pi}{2}$. We may assume $m \geq 3$, since $x=\pm 1$ are clearly not roots of $d_g$, and thus $\tfrac{2\pi}{m} < \pi$. In particular, $\varphi+\tfrac{2\pi}{m} < \tfrac{\pi}{2}+\pi=\tfrac{3\pi}{2}$, hence cannot be in the fourth quadrant.

Since both $\varphi$ and $\varphi+\tfrac{2\pi}{m}$ are on the same side of the $y$-axis and we assumed $\varphi$ is in the first quadrant, it follows that $\tfrac{\pi}{3} \leq \varphi+\tfrac{2\pi}{m}\leq \tfrac{\pi}{2}$ as well. Since $\cos(x)$ is decreasing on this interval we have

\[
\cos(\varphi)>\cos \left (\varphi+\frac{2\pi}{m} \right ) > 0 \implies 2 \cos(\varphi)>\cos \left (\varphi+\frac{2\pi}{m} \right),
\]

\noindent contradicting equation (\ref{Salem4d}). A similar argument shows that $\varphi$ cannot be in the third quadrant. Thus we revise the restrictions on $\varphi$ to be

\[
-\frac{\pi}{2} < \varphi \leq -\frac{\pi}{3} \ \text{or} \ \frac{\pi}{2} < \varphi \leq \frac{2\pi}{3}
\]

\noindent Appealing to the formula $\cos(\theta)=\sin(\theta+\pi/2)$, we may equivalently consider the equation $\sin(\psi+2\pi/m)-2\sin(\psi)=0$ with

\[
0 < \psi \leq \frac{\pi}{6} \ \text{or} \ \pi < \psi \leq \frac{7\pi}{6}
\]

Suppose $\psi$ is in the first quadrant and write

\[
\psi=\varphi+\frac{\pi}{2}=\frac{(2g-1)\pi}{m}+\frac{\pi}{2} \equiv \frac{j\pi}{2m} \pmod{2\pi}
\]

\noindent for some positive integer $j \leq 2m-1$ such that $0<\tfrac{j\pi}{2m} \leq \tfrac{\pi}{6}$. Using the subadditivity of $\sin(x)$ on $[0,\pi]$ now gives

\begin{align*}
\sin \left ( \psi+\frac{2\pi}{m} \right)-2\sin(\psi) & \leq \sin(\psi)+\sin \left ( \frac{2\pi}{m} \right ) - 2\sin(\psi) \\
& = \sin \left ( \frac{2\pi}{m} \right ) - \sin \left ( \frac{j\pi}{2m} \right)
\end{align*}

\noindent This expression cannot be zero unless $j=4$ because of the restriction on $j$. In this case,

\begin{align*}
\sin \left ( \psi+\frac{2\pi}{m} \right ) -2\sin(\psi) & = \sin \left ( \frac{4\pi}{m} \right)-2\sin \left ( \frac{2\pi}{m} \right ) \\
& = 2\sin \left ( \frac{2\pi}{m} \right) \cos \left ( \frac{2\pi}{m} \right ) - 2 \sin \left ( \frac{2\pi}{m} \right )\\
& = 2\sin \left ( \frac{2\pi}{m} \right) \left [ \cos \left ( \frac{2\pi}{m} \right ) - 1 \right ]
\end{align*}

\noindent This quantity can only be zero if $m=1$, which we quickly rule out after observing that $d_g(1)=6g-1$. The same argument works if we assume $\psi$ is in the third quadrant. Therefore $d_g(t)$ has no cyclotomic factor, and hence is irreducible.
\end{proof}


\subsection{Completing the proof of Theorem \ref{Salem}}


\begin{proof}[Proof of Theorem \ref{Salem}]
Lemma \ref{Salem4} proves statement (1). Since the $\lambda_g$-uniform expander $f_g$ corresponds to the fraction $q_g=\frac{1}{2g}$, the point $x=1$ is periodic of length $2g+1$. Each of these points lifts to a one-pronged singularity of the pseudo-Anosov $\phi_g$ on a punctured sphere, as does the fixed postcritical point $x=0$. There are no other one-pronged singularities, so taking the double cover of the surface by branching at each of these $2g+2$ points produces a surface on which the lift of each point is a flat point, i.e. has cone angle $2\pi$. The only other cone point downstairs is the fixed point at infinity, with cone angle $2g \cdot \pi$. This point lifts to two points of angle $g \cdot 2\pi$.

The derivative of the lifted pseudo-Anosov is 

\[
D\tilde{\phi_g}=\begin{pmatrix}
\lambda_g & 0 \\ 0 & \lambda_g^{-1}
\end{pmatrix},
\] 

\noindent an element of the Veech group $\SL(X_g, \omega_g)$. The trace of this matrix is $\lambda_g+\lambda_g^{-1}$, contained in the trace field $K_g$ by definition. But now

\[
g \geq [K_g:\bbQ] \geq [\bbQ(\lambda_g+\lambda_g^{-1}):\bbQ]=g,
\]

\noindent so in fact we have $[K_g: \bbQ]=g$, implying that $(X_g, \omega_g)$ is algebraically primitive.
\end{proof}

\begin{figure}[h!]
\centering
\includegraphics[scale=.4]{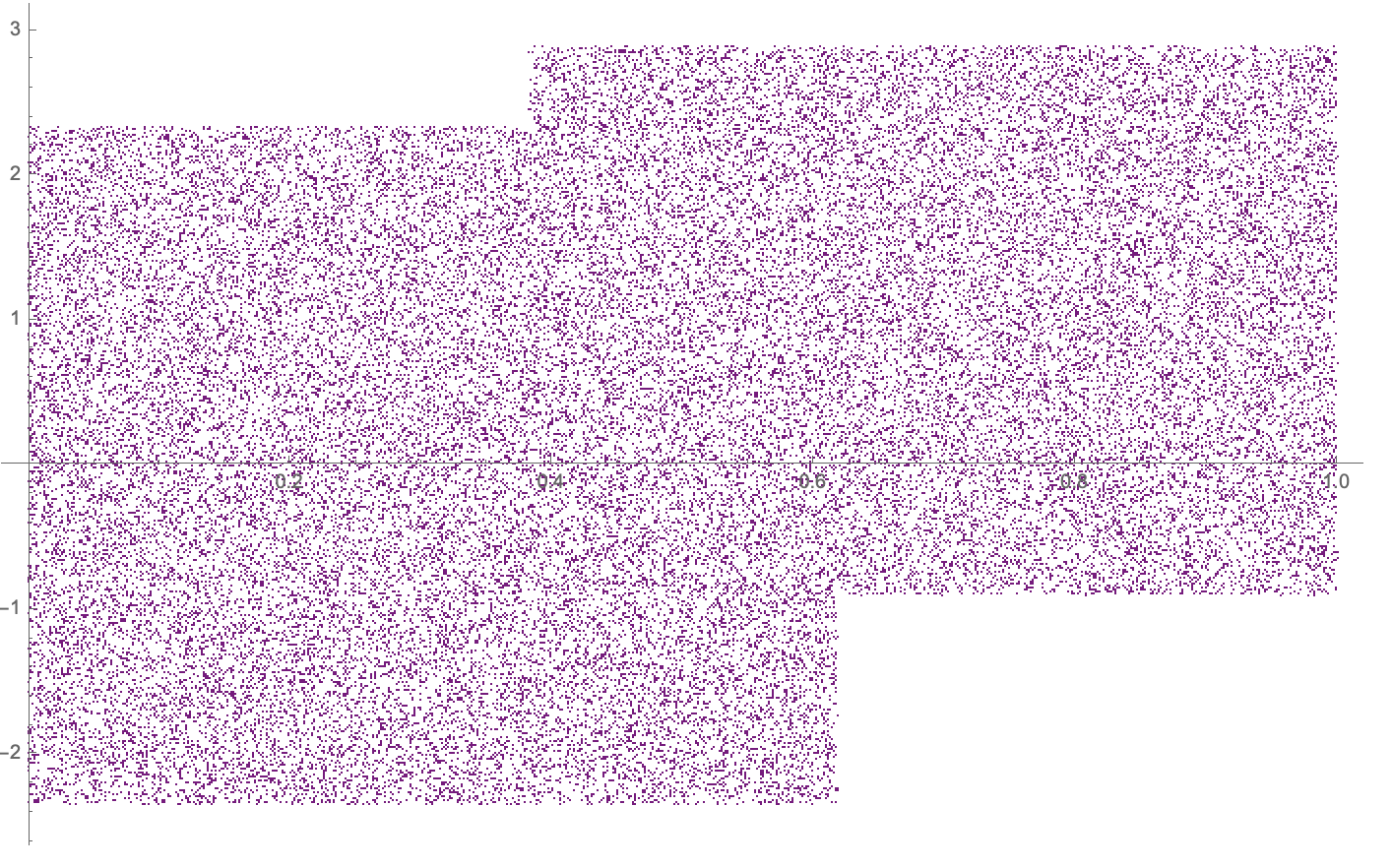}
\caption{The limit set of the Galois lift of $f_1$. This glues to a sphere with 4 marked points, which then lifts to a torus. The lift of the pseudo-Anosov $\phi_1$ is a linear Anosov diffeomorphism of the torus with stretch factor $\lambda_1=\frac{3+\sqrt{5}}{2}$, hence is conjugate to Arnold's cat map.}
\end{figure}

\begin{figure}
\centering
\includegraphics[scale=.4]{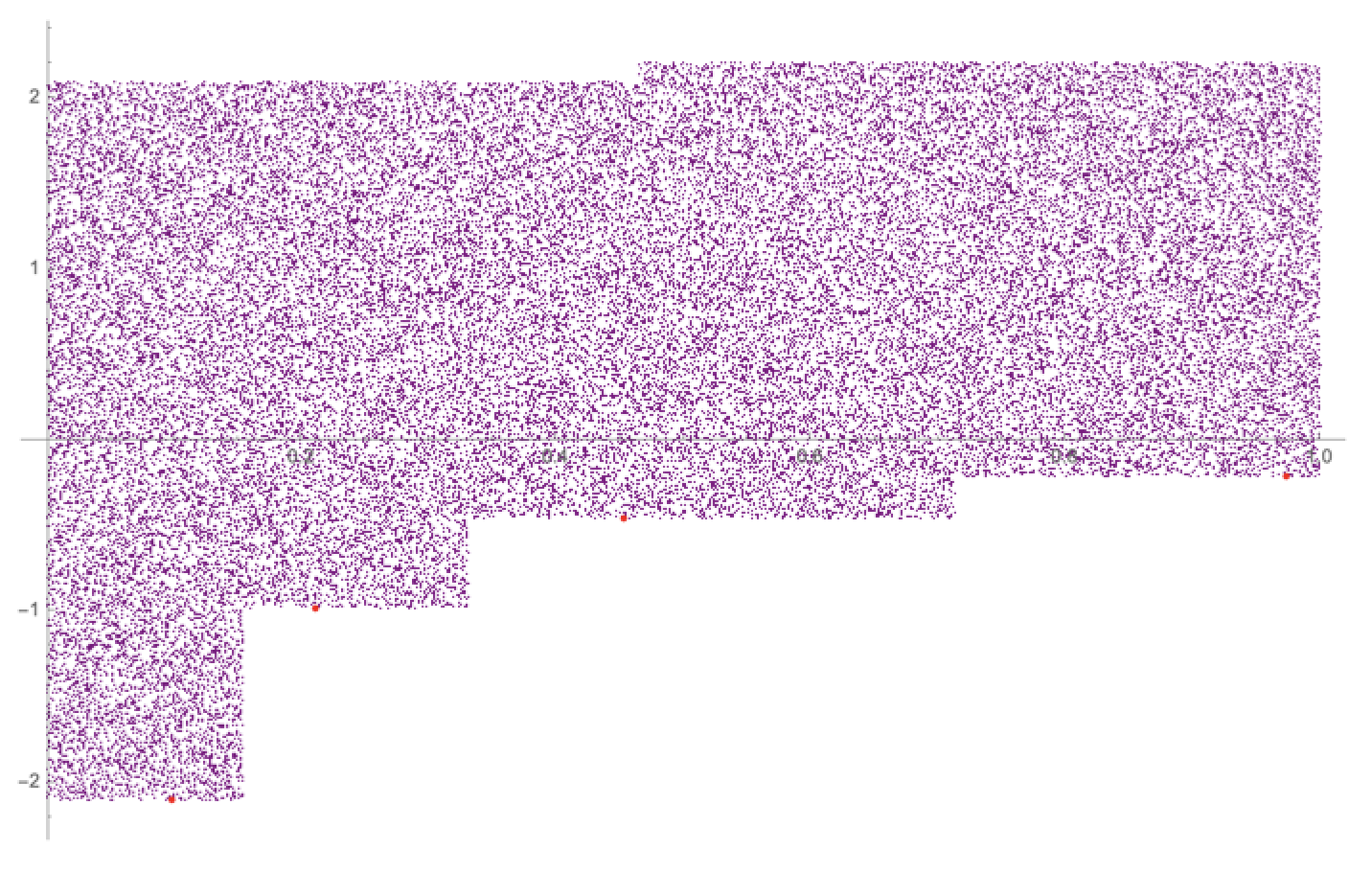}
\caption{The limit set of the Galois lift of $f_2$. This glues to a sphere with 6 cone points of angle $\pi$, one at the center of each vertical edge. Taking the double cover of this sphere branched at the 6 cone points produces a genus 2 surface.}
\end{figure}

\begin{figure}
\centering
\includegraphics[scale=.4]{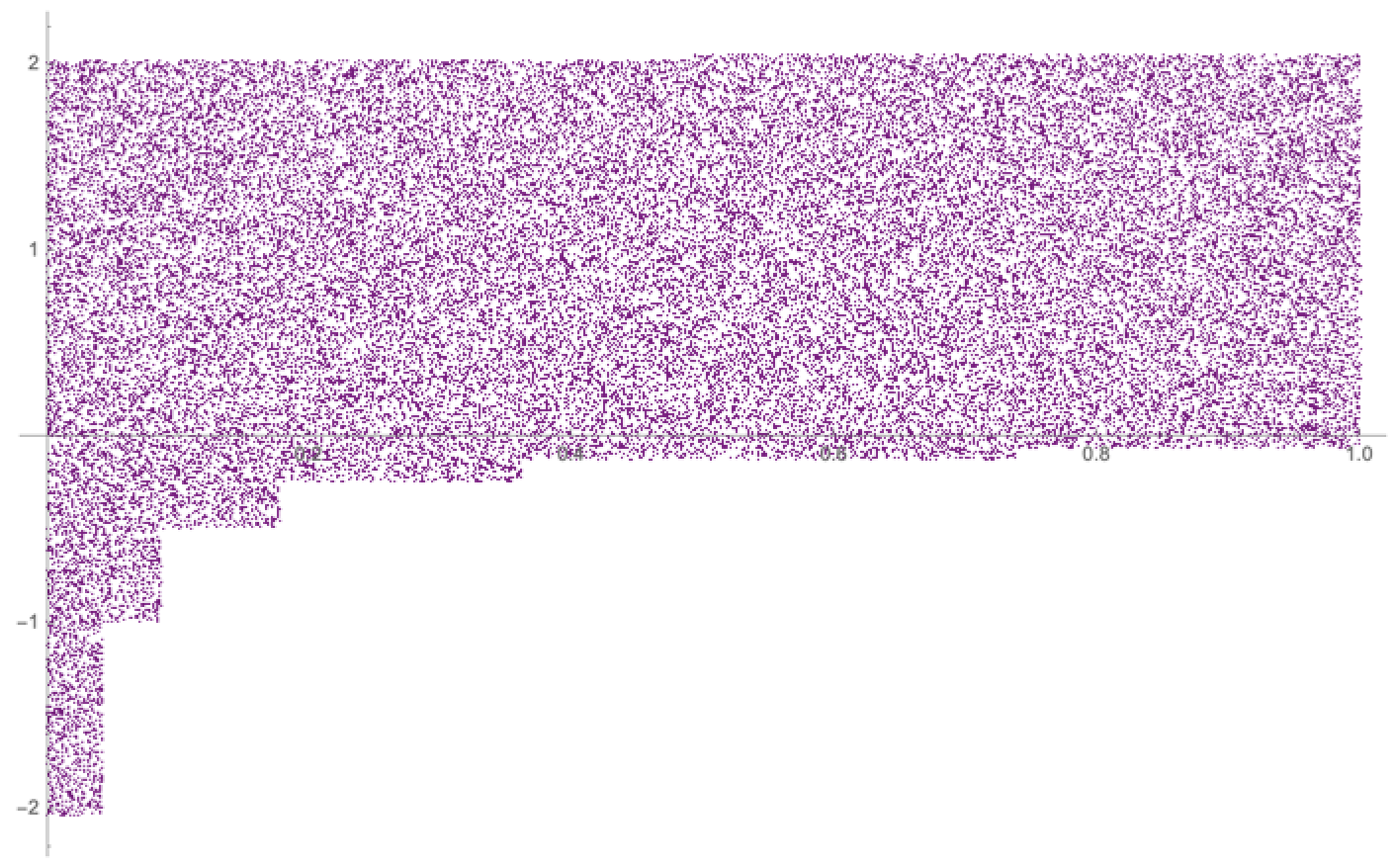}
\caption{The limit set of the Galois lift of $f_3$. This glues to a sphere with 8 marked points, which then lifts to a genus 3 surface.}
\end{figure}

\newpage

\bibliography{Expander}
\bibliographystyle{alpha}

\end{document}